\pdfoutput=1
\documentclass[12pt,british,a4paper,reqno, oneside]{amsart}
    \usepackage[T1]{fontenc}
\usepackage[utf8]{inputenc}
	\usepackage[british]{babel}
    \DeclareMathSizes{12}{12}{7}{6}

	\usepackage{amsmath}
	\usepackage{amssymb}
	\usepackage{amsthm}
	\usepackage{adjustbox}
	\usepackage{braket}
	\usepackage{xcolor}
	\usepackage[shortlabels]{enumitem}
	\usepackage{etoolbox}
	\usepackage{fancyhdr}
	\usepackage{geometry}
	\usepackage{graphicx}
	\usepackage{indentfirst}
	\usepackage{mathrsfs}
	\usepackage{mathtools}
	\usepackage{oplotsymbl} %
	\usepackage{parskip}
	\usepackage{proof}
	\usepackage{setspace}
	\usepackage{soul}
	\usepackage{stmaryrd}
	\usepackage{tensor}
	\usepackage{tikz}
	\usepackage{tikz-cd}
	\usepackage[normalem]{ulem}
	\usepackage{cancel}
    \usepackage{float}
    \usepackage{marvosym}
    \usepackage{wasysym}

    \usepackage[hidelinks, hyperfootnotes=false]{hyperref}
		\usepackage[capitalise]{cleveref}
		\usepackage{bookmark}
\makeatletter
\if@cref@capitalise
\crefname{theorem}{TheoUpperCase}{TheoUppercasesS}
\else
\crefname{theorem}{theoLowercase}{theoLowercaseS}
\fi
\makeatother
	\bookmarksetup{numbered,open}
	\newcommand{\wt}[1]{\widetilde{#1}}
	\newcommand{\wh}[1]{\widehat{#1}}

	\renewcommand{\geq}{\geqslant}
	\renewcommand{\leq}{\leqslant}
	\renewcommand{\phi}{\varphi}

	\providecommand{\corollaryname}{Corollary}
	\providecommand{\definitionname}{Definition}
	\providecommand{\notationname}{Notation}
	\providecommand{\examplename}{Example}
	\providecommand{\lemmaname}{Lemma}
	\providecommand{\propositionname}{Proposition}
	\providecommand{\remarkname}{Remark}
	\providecommand{\theoremname}{Theorem}
	\providecommand{\setupname}{Setup}
	\providecommand{\conjecturename}{Conjecture}
	\providecommand{\questionname}{Question}
	\providecommand{\calculationname}{Calculation}

	\theoremstyle{plain}
		\newtheorem{thm}{\protect\theoremname}[section] %
		\newtheorem{prop}[thm]{\protect\propositionname}
		    \crefalias{prop}{proposition}
		\newtheorem{lem}[thm]{\protect\lemmaname}
		\newtheorem*{lem*}{Lemma}
		\newtheorem{cor}[thm]{\protect\corollaryname}
		
		\newtheorem{thmx}{\protect\theoremname}
		\newtheorem{corx}[thmx]{\protect\corollaryname}

	\theoremstyle{definition}
		\newtheorem{defn}[thm]{\protect\definitionname}
		\newtheorem{notn}[thm]{\protect\notationname}
		    \crefname{notn}{Notation}{Notations}
		
		\newtheorem{setup}[thm]{\setupname}
		    \Crefname{setup}{Setup}{Setups}

	\theoremstyle{remark}
		\newtheorem{rem}[thm]{\protect\remarkname}

	\numberwithin{figure}{section}
	\numberwithin{equation}{section}

	\usetikzlibrary{matrix,arrows,decorations.pathmorphing,positioning,decorations.pathreplacing}
	\tikzset{commutative diagrams/.cd, 
		mysymbol/.style = {start anchor=center, end anchor = center, draw = none}}
	\tikzcdset{every label/.append style = {font = \footnotesize}}
	\tikzcdset{scale cd/.style={every label/.append style={scale=#1},
    cells={nodes={scale=#1}}}}

	\let\amph=& %

	\newcommand{\BE}{\mathbb{E}}
	\newcommand{\BG}{\mathbb{G}}
	\newcommand{\BF}{\mathbb{F}}
	\newcommand{\BN}{\mathbb{N}}

	\newcommand{\BZ}{\mathbb{Z}}

	\newcommand{\CA}{\mathcal{A}}
	\newcommand{\CB}{\mathcal{B}}
	\newcommand{\CC}{\mathcal{C}}
	\newcommand{\CD}{\mathcal{D}}
	\newcommand{\CE}{\mathcal{E}}

	\newcommand{\CX}{\mathcal{X}}

	\newcommand{\SE}{\mathscr{E}}
	\newcommand{\SF}{\mathscr{F}}
	\newcommand{\SG}{\mathscr{G}}
	\newcommand{\SH}{\mathscr{H}}
	\newcommand{\SI}{\mathscr{I}}
	
	\newcommand{\SK}{\mathscr{K}}
    \newcommand{\SL}{\mathscr{L}}

\DeclareFontFamily{U}{rcjhbltx}{}
\DeclareFontShape{U}{rcjhbltx}{m}{n}{<->rcjhbltx}{}
\DeclareSymbolFont{hebrewletters}{U}{rcjhbltx}{m}{n}
\let\aleph\relax
\let\beth\relax
\let\gimel\relax
\let\daleth\relax

\DeclareMathSymbol{\aleph}{\mathord}{hebrewletters}{39}
\DeclareMathSymbol{\beth}{\mathord}{hebrewletters}{98}
\DeclareMathSymbol{\gimel}{\mathord}{hebrewletters}{103}
\DeclareMathSymbol{\daleth}{\mathord}{hebrewletters}{100}

\DeclareMathSymbol{\lamed}{\mathord}{hebrewletters}{108}
\DeclareMathSymbol{\mem}{\mathord}{hebrewletters}{109}
\DeclareMathSymbol{\ayin}{\mathord}{hebrewletters}{96}
\DeclareMathSymbol{\tsadi}{\mathord}{hebrewletters}{118}
\DeclareMathSymbol{\qof}{\mathord}{hebrewletters}{113}
\DeclareMathSymbol{\shin}{\mathord}{hebrewletters}{152}

\newcommand{\Beth}{\mathbin{\scalebox{1.3}{$\beth$}}}

\newcommand{\Daleth}{\mathbin{\scalebox{1.3}{$\daleth$}}}

\newcommand{\Mem}{\mathbin{\scalebox{1.3}{$\mem$}}}

\newcommand{\Tsadi}{\mathbin{\scalebox{1.3}{$\tsadi$}}}

		\newcommand{\Ab}{\operatorname{\mathsf{Ab}}\nolimits}

		\newcommand{\op}{\raisebox{0.5pt}{\scalebox{0.6}{\textup{op}}}}
		
		\newcommand{\sse}{\subseteq}

		\newcommand{\obj}{\operatorname{obj}\nolimits}

		\newcommand{\Ker}{\operatorname{Ker}\nolimits}

		\newcommand{\triv}{\operatorname{triv}\nolimits}
		
		\newcommand{\iso}{\cong}

		\newcommand{\End}{\operatorname{End}\nolimits}

		\makeatletter
            \newcommand{\xmapsfrom}[2][]{%
            \ext@arrow3095\leftarrowfill@{#1}{#2}\mapsfromchar
            }
        \makeatother

		\newcommand{\id}[1]{\tensor*[]{\mathrm{id}}{_{#1}}}
		
		\newcommand{\wid}[1]{\tensor*[]{\wt{\mathrm{id}}}{_{#1}}}		

	    \newcommand\restr[2]{{\left.\kern-\nulldelimiterspace#1
						\right|_{#2}}}

		\newcommand{\Ext}{\operatorname{Ext}\nolimits}
		\newcommand{\fs}{\mathfrak{s}}
		\newcommand{\ft}{\mathfrak{t}}
		\newcommand{\fr}{\mathfrak{r}}

		\newcommand{\sus}{\Sigma} %

		\newcommand{\kom}[1]{\mathsf{K}({#1})}
		\newcommand{\com}[1]{\mathsf{Ch}({#1})}
		\newcommand{\ch}{\mathsf{Ch}}

		\let\amsamp=&

	\newcommand{\deff}{\coloneqq}

	\newcommand{\lan}{\langle}
	\newcommand{\ran}{\rangle}

    \newcommand{\au}{a}
    \newcommand{\at}{\tilde{a}}
    \newcommand{\cu}{c}
    \newcommand{\ct}{\tilde{c}}
    \newcommand{\pe}[1]{\tilde{p}_{#1}}
    \newcommand{\ie}[1]{\tilde{i}_{#1}}

\makeatletter

	\renewcommand{\andify}{%
		\nxandlist{\unskip, }{\unskip{} \@@and~}{\unskip \penalty-2 \space \@@and~}}
    
	\renewcommand\author@andify{%
  		\nxandlist {\unskip ,\penalty-1 \space\ignorespaces}%
		{\unskip {} \@@and~}%
		{\unskip \penalty-2 \space \@@and~}
	}

	    \newenvironment{acknowledgements}{%
	    \renewcommand\abstractname{\textbf{Acknowledgements}}
\global\setbox\abstractbox=\vtop \bgroup
\normalfont\Small
\list{}{\labelwidth\z@
\leftmargin3pc \rightmargin\leftmargin
\listparindent\normalparindent \itemindent\z@
\parsep\z@ \@plus\p@

}%
\item[\hskip\labelsep\scshape\abstractname.]%
}{%
\endlist\egroup
\ifx\@setabstract\relax \@setabstracta \fi
}

\def\@setaddresses{\par
    \nobreak \begingroup
    \setstretch{0.5} %
    \footnotesize
    \def\author##1{\nobreak\addvspace\bigskipamount}%
    \def\\{\unskip, \ignorespaces}%
    \interlinepenalty\@M
    \def\address##1##2{\begingroup
        \par\addvspace\bigskipamount\indent
    \@ifnotempty{##1}{(\ignorespaces##1\unskip) }%
    {\scshape\ignorespaces##2}\par\endgroup}%
    \def\curraddr##1##2{\begingroup
    \@ifnotempty{##2}{\nobreak\indent\curraddrname
    \@ifnotempty{##1}{, \ignorespaces##1\unskip}\/:\space
    ##2\par}\endgroup}%
    \def\email##1##2{\begingroup
    \@ifnotempty{##2}{\nobreak\indent\emailaddrname
    \@ifnotempty{##1}{, \ignorespaces##1\unskip}\/:\space
    \ttfamily##2\par}\endgroup}%
    \def\urladdr##1##2{\begingroup
    \def~{\char'\~}%
    \@ifnotempty{##2}{\nobreak\indent\urladdrname
    \@ifnotempty{##1}{, \ignorespaces##1\unskip}\/:\space
    \ttfamily##2\par}\endgroup}%
    \addresses
    \endgroup
}

\let\oldtocsection=\tocsection
\let\oldtocsubsection=\tocsubsection

\renewcommand{\tocsection}[2]{\vspace*{-5pt}\hspace{0em}\oldtocsection{#1}{#2}}

    \renewcommand{\tocsubsection}[2]{\vspace*{-5pt}\hspace{21pt}\oldtocsubsection{#1}{#2}}

\makeatother

\allowdisplaybreaks %
\begin{document}

\title{Idempotent completions of $n$-exangulated categories}
        
\author[Klapproth]{Carlo Klapproth}
        \address{
		Department of Mathematics\\
		Aarhus University\\
		8000 Aarhus C\\
		Denmark}
        \email{carlo.klapproth@math.au.dk}
        
\author[Msapato]{Dixy Msapato}
        \address{
		School of Mathematics\\
		University of Leeds\\
		Leeds LS2 9JT\\
		UK}
        \email{mmdmm@leeds.ac.uk}

\author[Shah]{Amit Shah*}
        \address{
		Department of Mathematics\\
		Aarhus University\\
		8000 Aarhus C\\
		Denmark}
        \email{amit.shah@math.au.dk}
        \thanks{*Corresponding author}

\date{\today}
{\setstretch{1}\begin{abstract}
Suppose $(\mathcal{C},\mathbb{E},\mathfrak{s})$ is an $n$-exangulated category. We show that the idempotent completion and the weak idempotent completion of $\mathcal{C}$ are again $n$-exangulated categories. Furthermore, we also show that the canonical inclusion functor of $\mathcal{C}$ into its (resp.\ weak) idempotent completion is $n$-exangulated and $2$-universal among $n$-exangulated functors from $(\mathcal{C},\mathbb{E},\mathfrak{s})$ to (resp.\ weakly) idempotent complete $n$-exangulated categories. Furthermore, we prove that if $(\mathcal{C},\mathbb{E},\mathfrak{s})$ is $n$-exact, then so too is its (resp.\ weak) idempotent completion. We note that our methods of proof differ substantially from the extriangulated and $(n+2)$-angulated cases. However, our constructions recover the known structures in the established cases up to $n$-exangulated isomorphism of $n$-exangulated categories. 
\end{abstract}}

\maketitle
\vspace{-0.5cm}
{\setstretch{1} \tableofcontents}

{\setstretch{1}\begin{acknowledgements}
The authors would like to thank Theo B\"{u}hler, Ruben Henrard and Adam-Christiaan van Roosmalen for useful email communications, and Andrew Brooke-Taylor and Peter J{\o{}}rgensen for helpful discussions. 
The authors are very grateful to the referees for their comments and suggestions. In particular, these comments led to Corollaries~\ref{cor:idempotent-n-exact} and \ref{cor:WIC-n-exact}, and improvements to the exposition.
\end{acknowledgements}}
\newpage

\section{Introduction}
\label{sec:introduction}

Idempotent completion 
began 
with Karoubi's work \cite{Karoubi-algebres-de-Clifford-et-K-theorie} on additive categories. 
It was shown that 
an additive category 
embeds into an associated one which is \emph{idempotent complete}, 
that is, in which all idempotent morphisms
admit a kernel. 
Particularly nice examples of idempotent complete categories include Krull-Schmidt categories, which can be characterised as idempotent complete additive categories in which each object has a semi-perfect endomorphism ring 
(see 
Chen--Ye--Zhang \cite[Thm.\ A.1]{ChenYeZhang-Algebras-of-derived-dimension-zero},  
Krause \cite[Cor.\ 4.4]{Krause-KS-cats-and-projective-covers}).
Other 
examples include the vast class of pre-abelian categories (see e.g.\ \cite[Rem.\ 2.2]{Shah-AR-theory-quasi-abelian-cats-KS-cats}); e.g.\ a module category, or the category of Banach spaces (over the reals, say).

Suppose $\CC$ is an additive category. 
The objects of the \emph{idempotent completion $\wt{\CC}$ of $\CC$} are pairs $(X,e)$, where $X$ is an object of $\CC$ and $e\colon X \to X$ is an \emph{idempotent} morphism in $\CC$, i.e.\ $e^2 = e$. 
What is particularly nice is that if $\CC$ has a certain kind of structure, then in several cases this induces the same structure on 
$\wt{\CC}$. 
For example, Karoubi had already shown that the idempotent completion of an additive category is again additive (see \cite[(1.2.2)]{Karoubi-algebres-de-Clifford-et-K-theorie}).
Furthermore, it has been shown for the following, amongst others, extrinsic structures that if $\CC$ has such a structure, then so too does $\wt{\CC}$:
\begin{enumerate}[label=\textup{(\roman*)}]
	\item\label{item:intro-list-triangulated} triangulated (see Balmer--Schlichting \cite[Thm.\ 1.5]{BalmerSchlichting-idempotent-completion-of-triangulated-categories});
	\item\label{item:intro-list-exact} exact (see B\"{u}hler \cite[Prop.\ 6.13]{Buhler-exact-categories});
	\item\label{item:intro-list-extriangulated} extriangulated (see \cite[Thm.\ 3.1]{Msapato-the-karoubi-envelope-and-weak-idempotent-completion-of-an-extriangulated-category}); and 
	\item\label{item:intro-list-nplus2-angulated} $(n+2)$-angulated, where $n\geq 1$ is an integer (see Lin \cite[Thm.\ 3.1]{Lin-Idempotent-completion-of-n-Angulated-categories}). 
\end{enumerate}
See also 
Liu--Sun \cite{LiuSun-Idempotent-completion-of-pretriangulated-categories} and 
Zhou \cite{Zhou-Idempotent-completion-of-right-suspended-categories}.

Idempotent complete exact and triangulated categories are verifiably important in algebra and algebraic geometry. 
As a classical example, in Neeman \cite{Neeman-the-derived-category-of-an-exact-category} an idempotent complete exact category $\CE$ is needed to 
give a clean description of the kernel of the localisation functor from the homotopy category of $\CE$ to its derived category.
And, more generally, many equivalences only hold up to direct summands, i.e.\ up to idempotents (see, for example, \mbox{Orlov \cite[Thm.\ 2.11]{Orl11}}, or Kalck--Iyama--Wemyss--Yang \cite[Thm.\ 1.1]{KIWY15}). 
Therefore, it is usually helpful to view an algebraic structure as sitting 
inside its idempotent completion.

The idempotent completion $\wt{\CC}$ comes equipped with an inclusion functor $\tensor*[]{\SI}{_{\CC}}\colon\CC\to\wt{\CC}$ given by $\tensor*[]{\SI}{_{\CC}}(X) = (X,\id{X})$ on objects. 
Moreover, in several of the cases above it has been shown that this functor is \emph{$2$-universal} in an appropriate sense; see e.g.\ 
\cref{prop:Buehler-universal-property-of-karoubi-envelope}
for a precise formulation.
For example, without any assumptions other than additivity, the functor $\tensor*[]{\SI}{_{\CC}}$ is additive and $2$-universal amongst additive functors from $\CC$ to idempotent complete additive categories. 
On the other hand, if e.g.\ $\CC$ has an exact structure, then $\tensor*[]{\SI}{_{\CC}}$ is exact and $2$-universal amongst exact functors from $\CC$ to idempotent complete exact categories.

In homological algebra two parallel generalisations have been made from the classical settings of exact and triangulated categories. 
One of these has been the introduction of \emph{extriangulated} categories as defined by Nakaoka--Palu \cite{NakaokaPalu-extriangulated-categories-hovey-twin-cotorsion-pairs-and-model-structures}. 
An extriangulated category is a triplet $(\CC,\BE,\fs)$, where 
$\CC$ is an additive category, 
$\BE\colon \tensor*[]{\CC}{^{\op}} \times \CC \to \Ab$ is a biadditive functor to the category of abelian groups,
and $\fs$ is a so-called additive realisation of $\BE$. 
The realisation $\fs$ associates to each $\delta\in\BE(Z,X)$ a certain equivalence class 
$\fs(\delta) = [\begin{tikzcd}[column sep=0.5cm]X \arrow{r}{x}&Y \arrow{r}{y}&Z\end{tikzcd}]$
of a $3$-term complex. 
As an example, each triangulated category $(\CC,\Sigma,\triangle)$, where $\Sigma$ is a suspension functor and $\triangle$ is a triangulation, is an extriangulated category. 
Indeed, one defines the corresponding bifunctor by $\tensor*[]{\BE}{_{\Sigma}}(Z,X) \deff \CC(Z,\sus X)$. 
See \cite[Prop.\ 3.22]{NakaokaPalu-extriangulated-categories-hovey-twin-cotorsion-pairs-and-model-structures} for more details. 
In addition, each suitable exact category is extriangulated; see \cite[Exam.\ 2.13]{NakaokaPalu-extriangulated-categories-hovey-twin-cotorsion-pairs-and-model-structures}. 
A particular advantage of this theory is that the collection of extriangulated categories is closed under taking extension-closed subcategories. 
Although an extension-closed subcategory of an exact category is again exact, the same does not hold in general for triangulated categories. 

We note here that, importantly, it was shown in \cite[Sec.\ 3.1]{Msapato-the-karoubi-envelope-and-weak-idempotent-completion-of-an-extriangulated-category} that the extriangulated structure on $\wt{\CC}$ produced from case \ref{item:intro-list-extriangulated} above is compatible with the more classical constructions of \ref{item:intro-list-triangulated} and \ref{item:intro-list-exact}. 
For instance, given a triangulated category $\CC$, one can equip its idempotent completion $\wt{\CC}$ with a triangulation by \ref{item:intro-list-triangulated}  
or with an extriangulation by \ref{item:intro-list-extriangulated}, 
but these structures are the same in the sense of \cite[Prop.\ 3.22]{NakaokaPalu-extriangulated-categories-hovey-twin-cotorsion-pairs-and-model-structures}. 
Analogously, \ref{item:intro-list-extriangulated} also recovers \ref{item:intro-list-exact} if one starts with an extriangulated category that is exact.

Let $n\geq 1$ be an integer. 
The other aforementioned generalisation in homological algebra has been the development of higher homological algebra. 
This includes the introduction of 
\emph{$n$-exact} and \emph{$n$-abelian} categories by Jasso \cite{Jasso-n-abelian-and-n-exact-categories}, 
and \emph{$(n+2)$-angulated} categories by Geiss--Keller--Oppermann \cite{GeissKellerOppermann-n-angulated-categories}. 
Respectively, these generalise exact, abelian and triangulated categories, in that one recovers the classical notions by setting $n=1$.
For instance, an $(n+2)$-angulated category is a triplet $(\CC,\Sigma,\pentagon)$ satisfying some axioms, where $\Sigma$ is still an automorphism of $\CC$, but now $\pentagon$ consists of a collection of \emph{$(n+2)$-angles} each of which has $n+3$ terms.

The focal point of this paper is on the idempotent completion of an \emph{$n$-exangulated} category. These categories were axiomatised by Herschend--Liu--Nakaoka \cite{HerschendLiuNakaoka-n-exangulated-categories-I-definitions-and-fundamental-properties}, 
and simultaneously generalise 
extriangulated, 
$(n+2)$-angulated, and
suitable $n$-exact 
categories (see \mbox{\cite[Sec.\ 4]{HerschendLiuNakaoka-n-exangulated-categories-I-definitions-and-fundamental-properties}}). 
Like an extriangulated category, an $n$-exangulated category $(\CC,\BE,\fs)$ consists of an additive category $\CC$, a biadditive functor $\BE\colon \tensor*[]{\CC}{^{\op}} \times \CC \to \Ab$, and a so-called exact realisation $\fs$ of $\BE$, which satisfy some axioms (see Subsection~\ref{sec:n-exangulated-categories}). 
The realisation $\fs$ now associates to each $\delta\in\BE(Z,X)$ a certain equivalence class (see Subsection~\ref{sec:n-exangulated-categories}) 
\[\fs(\delta) = [\begin{tikzcd}%
\tensor*[]{X}{_{0}} \arrow{r}{\tensor*[]{d}{_{0}^{X}}}&\tensor*[]{X}{_{1}} \arrow{r}{\tensor*[]{d}{_{1}^{X}}}&\cdots \arrow{r}{\tensor*[]{d}{_{n}^{X}}}&\tensor*[]{X}{_{n+1}}\end{tikzcd}]\]
of an $(n+2)$-term complex. 
In this case, the pair $\lan \tensor*[]{X}{_{\bullet}},\delta\ran$ is called an \emph{$\fs$-distinguished $n$-exangle}.
We recall that structure-preserving functors between $n$-exangulated categories were defined in \cite[Def.\ 2.32]{Bennett-TennenhausShah-transport-of-structure-in-higher-homological-algebra}. 
They are known as \emph{$n$-exangulated} functors and they send distinguished $n$-exangles to distinguished $n$-exangles.

Suppose that $(\CC,\BE,\fs)$ is an $n$-exangulated category. 
Let $\wt{\CC}$ denote the idempotent completion of $\CC$ as an additive category. 
We define a biadditive functor $\BF\colon \tensor*[]{\wt{\CC}}{^{\op}} \times \wt{\CC} \to \Ab$ as follows. 
For any pair of objects $(X,e),(Z,e')\in\wt{\CC}$, we let 
$\BF((Z,e'),(X,e))$
consist of triplets $(e,\delta,e')$ where 
$\delta\in\BE(Z,X)$ such that $\BE(Z,e)(\delta) = \delta = \BE(e',X)(\delta)$. 
On morphisms $\BF$ is essentially a restriction of $\BE$; see \cref{def:BF} for details. 
Now we define 
a realisation $\ft$ of $\BF$. 
For $(e,\delta,e') \in \BF((Z,e'),(X,e))$, we have that $\fs(\delta) = [\tensor*[]{X}{_{\bullet}}]$ for some $(n+2)$-term complex $\tensor*[]{X}{_{\bullet}}$ with $\tensor*[]{X}{_{0}} = X$ and $\tensor*[]{X}{_{n+1}} = Z$ since $\fs$ is a realisation of $\BE$. 
We choose an idempotent morphism 
$\tensor*[]{e}{_{\bullet}} \colon \tensor*[]{X}{_{\bullet}}  \to \tensor*[]{X}{_{\bullet}}$
of complexes, such that $\tensor*[]{e}{_{0}} = e$ and $\tensor*[]{e}{_{n+1}} = e'$; see \cref{cor:newlift}. 
Lastly, we set $\ft((e,\delta,e'))$ to be the equivalence class of the complex 
\[
\begin{tikzcd}[column sep=2cm]
(X,e)
	\arrow{r}{\tensor*[]{e}{_{1}} \tensor*[]{d}{_{0}^{X}} \tensor*[]{e}{_{0}}}
& (\tensor*[]{X}{_{1}},\tensor*[]{e}{_{1}})
	\arrow{r}{\tensor*[]{e}{_{2}} \tensor*[]{d}{_{1}^{X}} \tensor*[]{e}{_{1}}}
& \cdots 
	\arrow{r}{\tensor*[]{e}{_{n}} \tensor*[]{d}{_{n-1}^{X}} \tensor*[]{e}{_{n-1}}}
& (\tensor*[]{X}{_{n}},\tensor*[]{e}{_{n}})
	\arrow{r}{\tensor*[]{e}{_{n+1}} \tensor*[]{d}{_{n}^{X}} \tensor*[]{e}{_{n}}}
& (Z,e')
\end{tikzcd}
\]
in $\wt{\CC}$. 
We say that an $n$-exangulated category is \emph{idempotent complete} if its underlying additive category is (see \cref{def:idempotent-complete-n-exangulated-category}).

\begin{thmx}[\cref{thm:mainthm-ctilde}, \cref{thm:IC-is-2-universal}]
\label{thmx:mainthm-ctilde}
The triplet $(\wt{\CC}, \BF, \ft)$ is an idempotent complete $n$-exangulated category. 
Furthermore, the inclusion functor $\tensor*[]{\SI}{_{\CC}} \colon \CC \to \wt{\CC}$ extends to an $n$-exangulated functor $(\tensor*[]{\SI}{_{\CC}}, \Gamma) \colon (\CC, \BE, \fs) \to (\wt{\CC}, \BF, \ft)$, which is $2$-universal among $n$-exangulated functors from $(\CC, \BE, \fs)$ to idempotent complete $n$-exangulated categories. 
\end{thmx}

An
\emph{$n$-exact category} $(\CC, \CX)$ (see \cite[Def.\ 4.2]{Jasso-n-abelian-and-n-exact-categories}) induces an $n$-exangulated category $(\CC, \BE, \fs)$ 
if, for each pair of objects $A,C\in\CC$, the 
collection 
$\BE(C,A) = \tensor*[]{\Ext}{_{\CC}^{n}}(C,A)$
of $n$-extensions of $C$ by $A$ forms a set; see \cite[Prop.\ 4.34]{HerschendLiuNakaoka-n-exangulated-categories-I-definitions-and-fundamental-properties}. 
As in \cite[Def.\ 4.6]{klapproth2023nextension}, we say that an $n$-exangulated category $(\CC, \BE, \fs)$ is \emph{$n$-exact} if its $n$-exangulated structure arises in this way. 
Combining \cref{thmx:mainthm-ctilde} with \cite[Cor.\ 4.12]{klapproth2023nextension}, we deduce the following.

\begin{corx}[\cref{cor:idempotent-n-exact}]
\label{corx:n-exact}
    If $(\CC,\BE,\fs)$ is an $n$-exangulated category that is $n$-exact, then the idempotent completion $(\wt{\CC}, \BF, \ft)$ is $n$-exact.
\end{corx}

We explain in \cref{rem:recovering-previous-cases} how \cref{thmx:mainthm-ctilde} unifies the constructions in cases \ref{item:intro-list-triangulated}--\ref{item:intro-list-nplus2-angulated} above. 
Furthermore, we comment on some obstacles faced in proving the $n$-exangulated case in \cref{rem:proofs-are-hard}.

From \cref{thmx:mainthm-ctilde} we deduce the following corollary, giving a way to produce Krull-Schmidt $n$-exangulated categories.

\begin{corx}[\cref{cor:Krull-Schmidt}]
\label{corx:Krull-Schmidt}
    If each object in $(\CC, \BE, \fs)$ has a semi-perfect endomorphism ring, then the idempotent completion $(\wt{\CC}, \BF, \ft)$ is a Krull-Schmidt $n$-exangulated category.
\end{corx}

Finally, we note that analogues of \cref{thmx:mainthm-ctilde} and \cref{corx:n-exact} are shown for the weak idempotent completion in \cref{sec:the-WIC}. 
The importance of being weakly idempotent complete for extriangulated categories was very recently demonstrated in \cite[Prop.\ 2.7]{klapproth2023nextension}. It turns out that for an extriangulated category, the underlying category being weakly idempotent complete is equivalent to the condition (WIC) defined in \cite[Cond.\ 5.8]{NakaokaPalu-extriangulated-categories-hovey-twin-cotorsion-pairs-and-model-structures}. Moreover, (WIC) is a key assumption in many results on extriangulated categories, e.g.\ 
\cite[\S\S 5--7]{NakaokaPalu-extriangulated-categories-hovey-twin-cotorsion-pairs-and-model-structures}, \mbox{\cite[\S 3]{HerschendLiuNakaoka-n-exangulated-categories-II}}, Zhao--Zhu--Zhuang \cite{ZhaoZhuZhuang-tilting-pairs-in-extriangulated-categories}.
We remark that the analogue of (WIC) for $n$-exangulated categories is automatic if $n\geq 2$, but it is not equivalent to the weak idempotent completeness of the underlying category; see \mbox{\cite[Thm.\ B]{klapproth2023nextension}} for more details.

\newpage
\section{On the splitting of idempotents}
\label{sec:on-the-splitting-of-idempotents}

In this section we recall some key definitions regarding idempotents and idempotent completions of categories. 
We focus on the idempotent completion of an additive category in Subsection~\ref{subsec:idempotent-completion} and on the weak idempotent completion in Subsection~\ref{subsec:weak-idempotent-completion}. 
Throughout this section, we let $\CA$ denote an additive category. 
For a more in-depth treatment, we refer the reader to \cite[Secs.\ 6--7]{Buhler-exact-categories}.

\subsection{Idempotent completion}
\label{subsec:idempotent-completion}

Recall that by an \emph{idempotent (in $\CA$)} we mean a morphism $e\colon X\to X$ satisfying $e^2 = e$ for some object $X\in\CA$.

The following definition is from Borceux \cite{Borceux-handbook-1}.
\begin{defn} 
\label{defn:split-idempotents} 
\cite[Defs.\ 6.5.1, 6.5.3]{Borceux-handbook-1} 
An idempotent $e \colon X \to X$ in $\CA$ is said to \emph{split} if there exist morphisms 
$r \colon X \to Y$ and $s \colon Y \to X$, such that $sr=e$ and $rs=\id{Y}$. 
The category $\CA$ is \emph{idempotent complete}, or has \emph{split idempotents}, if every idempotent in $\CA$ splits.  
\end{defn}

If $\CA$ has split idempotents and $e\colon X \to X$ is an idempotent in $\CA$, then the object $X$ admits a direct sum decomposition $X \iso \Ker(e) \oplus \Ker(\id{X} - e)$ (see e.g.\ Auslander \cite[p.\ 188]{Auslander-Rep-theory-of-Artin-algebras-I}). 
In particular, the idempotent $e$ and its counterpart $\id{X} -e$ each admit a kernel. 
Idempotent complete additive categories can be characterised by such a criterion and its dual.

\begin{prop} 
\label{prop:characterisation-of-idempotent-complete}
\cite[Prop.\ 6.5.4]{Borceux-handbook-1} 
An additive category is idempotent complete 
if and only if every idempotent admits a kernel,
if and only if every idempotent admits a cokernel.
\end{prop}

From this point of view, idempotent complete categories sit between additive categories and \emph{pre-abelian} categories, the latter being additive categories in which every morphism admits a kernel and a cokernel; 
see for example Bucur--Deleanu \cite[\S 5.4]{BucurDeleanu-intro-to-theory-of-cats-and-functors}.

Every additive category can be viewed as a full subcategory of an idempotent complete one. This goes back to Karoubi \cite[Sec.\ 1.2]{Karoubi-algebres-de-Clifford-et-K-theorie}, so the idempotent completion of $\CA$ is also often referred to as the \emph{Karoubi envelope of $\CA$}.

\begin{defn}
\label{defn:karoubi-envelope}
The \emph{idempotent completion $\wt{\CA}$ of $\CA$} is the category defined as follows.
Objects of $\wt{\CA}$ are pairs $(X, e)$, where $X$ is an object of $\CA$ and $e\in\tensor*[]{\End}{_{\CA}}(X)$ is idempotent. 
For objects $(X,e), (Y,e') \in \obj \wt{\CA}$, 
a morphism from $(X,e)$ to $(Y,e')$ is a triplet 
$(e',r,e)$, 
where $r \in \CA(X,Y)$ satisfies 
\[
re = r = e'r
\] 
in $\CA$. 
Composition of morphisms is defined by 
\[(e'',s,e')\circ(e',r,e) \deff (e'', sr, e),\] 
whenever 
$(e',r,e)\in\wt{\CA}((X,e), (Y, e'))$ and $(e'',s,e')\in\wt{\CA}((Y,e'), (Z, e''))$. 
The identity of an object $(X,e)\in\obj \wt{\CA}$ will be denoted $\wid{(X,e)}$ and is the morphism $(e,e,e)$.
\end{defn}

A morphism $(e',r,e) \colon (X,e) \to (Y,e')$ in the idempotent completion $\wt{\CA}$  of $\CA$ is usually denoted more simply as $r$; see e.g.\  \cite[Def.\ 1.2]{BalmerSchlichting-idempotent-completion-of-triangulated-categories} and \cite[Rem.\ 6.3]{Buhler-exact-categories}. 
However, 
for precision
in Sections~\ref{sec:the-idempotent-completion}--\ref{sec:the-WIC}, we use triplets for morphisms in $\wt{\CA}$ so that we can easily distinguish morphisms in $\CA$ from morphisms in its idempotent completion. 
Our choice of notation also has the added benefit of keeping track of the (co)domain of a morphism in $\wt{\CA}$. This becomes important later when different morphisms in $\wt{\CA}$ have the same underlying morphism; see \cref{item:projection-inclusion-of-X-e-as-summand}.

By a functor we always mean a covariant functor.
The 
\emph{inclusion functor $\tensor*[]{\SI}{_{\CA}} \colon \CA \to \wt{\CA}$} is defined as follows. An object $X \in \obj \CA$ is 
sent
to $\tensor*[]{\SI}{_{\CA}}(X) \deff (X, \id{X}) \in \obj \wt{\CA}$ 
and a morphism $r \in \CA(X,Y)$ 
is mapped to $\tensor*[]{\SI}{_{\CA}}(r) \deff (\id{Y}, r,\id{X}) \in \wt{\CA}(\tensor*[]{\SI}{_{\CA}}(X), \tensor*[]{\SI}{_{\CA}}(Y))$.

\begin{lem} 
\label{lem:split-summand}
    If $e\in\tensor*[]{\End}{_{\CA}}(X)$ is a split idempotent, 
    with a splitting $e=sr$ where $r \colon X \to Y$ and $s \colon Y \to X$, 
    then $(X, e) \iso \tensor*[]{\SI}{_{\CA}}(Y)$.
\end{lem}
\begin{proof}
We have $re = rsr = \id{Y}r = r$ and $es = srs = s \id{Y} = s$. Hence, there are morphisms 
$\tilde{r} \deff (\id{Y},r,e) \colon (X,e) \to \tensor*[]{\SI}{_{\CA}}(Y)$ and $\tilde{s} \deff (e, s, \id{Y})\colon \tensor*[]{\SI}{_{\CA}}(Y)\to (X,e)$ in $\wt{\CA}$ with 
$\tilde{s} \tilde{r} = \wid{(X,e)} $ and $\tilde{r} \tilde{s} = \wid{\tensor*[]{\SI}{_{\CA}}(Y)}$. 
Hence, $\tilde{r}$ and $\tilde{s}$ are mutually inverse isomorphisms in $\wt{\CA}$. 
\end{proof}

If $\CA$ is an idempotent complete category, then the functor 
$\tensor*[]{\SI}{_{\CA}}$ 
is an equivalence of categories; see e.g.\ 
\cite[Rem.\ 6.5]{Buhler-exact-categories}. 
But more generally we have the following.

\begin{prop}
\label{prop:karoubi-envelope-idempotent-complete}
\cite[Rem.\ 6.3]{Buhler-exact-categories} 
    The idempotent completion $\wt{\CA}$ is an idempotent complete additive category with biproduct given by $(X, e) \oplus (Y, e') = (X \oplus Y, e \oplus e')$. The inclusion functor $\tensor*[]{\SI}{_{\CA}} \colon \CA \to \wt{\CA}$ is fully faithful and additive.  
\end{prop}

\begin{rem}
\label{rem:direct-sum-decomposition-of-X-id-wrt-idempotent}
Let $(X,e)$ be an arbitrary object of $\wt{\CA}$.
Then $(X,e)$ is a direct summand of $\tensor*[]{\SI}{_{\CA}}(X) = (X,\id{X})$. 
Indeed, there is an isomorphism 
\mbox{$(X,\id{X}) \iso (X,e) \oplus (X,\id{X}-e)$}. 
The canonical inclusion of $(X,e)$ into $(X,\id{X})$ is given by the morphism $(\id{X},e,e)$, and 
the projection of $(X,\id{X})$ onto $(X,e)$ by $(e,e,\id{X})$. 
Similarly 
for $(X,\id{X}-e)$.
\end{rem}

The 
functor $\tensor*[]{\SI}{_{\CA}} \colon \CA \to \wt{\CA}$ is $2$-universal in some sense; see \cref{prop:Buehler-universal-property-of-karoubi-envelope}. 
For this we recall the notion of 
\emph{whiskering} a natural transformation by a functor. 
We will use Hebrew letters (e.g.\ $\Beth$ (beth), $\Tsadi$ (tsadi), $\Daleth$ (daleth), $\Mem$ (mem)) for natural transformations. 
Suppose $\CB,\CC,\CD$ are categories and that we have a diagram
\[
\begin{tikzcd}[row sep=0.2cm]
{}&{}&{}
&{}
    \arrow[Rightarrow,
            shorten <= 8pt,
            shorten >= 4pt,
            ]{dd}
                [xshift=2pt, yshift=0pt]{\Beth}
&{}
\\
\CB 
	\arrow{rr}{\SF}
&{}
& \CC
	\arrow[bend left=30]{rr}
            [description, xshift=0pt, yshift=0pt]{\SG} 
	\arrow[bend right=30]{rr}
            [description, xshift=0pt, yshift=0pt]{\SH}
&{}
& \CD,
\\
{}&{}&{}&{}&{}
\end{tikzcd}
\]
where $\SF,\SG,\SH$ are functors and $\Beth\colon \SG \Rightarrow \SH$ is a natural transformation. 

\begin{defn}
\label{def:whiskering}
The \emph{whiskering of $\SF$ and $\Beth$} is the natural transformation
$\tensor*[]{\Beth}{_{\SF}}\colon \SG\SF \Rightarrow \SH\SF$ defined by 
$(\tensor*[]{\Beth}{_{\SF}})_{X} 
    \deff \tensor*[]{\Beth}{_{\SF(X)}}
    \colon \SG\SF(X) \to \SH\SF(X)
$ 
for each $X\in \CB$.
\end{defn}

The next proposition explains the $2$-universal property satisfied by $\tensor*[]{\SI}{_{\CA}} \colon \CA \to \wt{\CA}$.

\begin{prop}
\label{prop:Buehler-universal-property-of-karoubi-envelope}
    \cite[Prop.\ 6.10]{Buhler-exact-categories}
    For any additive functor $\SF \colon \CA \to \CB$ with $\CB$ idempotent complete:
    \begin{enumerate}[label=\textup{(\roman*)}]
        \item\label{2.8.1} there is an additive functor $\SE \colon \wt{\CA} \to \CB$ and a natural isomorphism 
        \mbox{$\Tsadi \colon \SF \overset{\iso}{\Longrightarrow}\SE \tensor*[]{\SI}{_{\CA}}$}; 
        and, in addition,
        
        \item\label{2.8.2} for any functor 
        $\SG \colon \wt{\CA} \to \CB$
        and 
        any natural transformation $\Daleth \colon \SF \Rightarrow \SG \tensor*[]{\SI}{_{\CA}}$, 
        there exists a unique natural transformation $\Mem \colon \SE \Rightarrow \SG$ with $\Daleth = \tensor*[]{\Mem}{_{\tensor*[]{\SI}{_{\CA}}}} \Tsadi$.

    \end{enumerate}
\end{prop}

\subsection{Weak idempotent completion}
\label{subsec:weak-idempotent-completion}

A weaker notion than being idempotent complete is that of being weakly idempotent complete. 
This was introduced in the context of exact categories by Thomason--Trobaugh \cite[Axiom A.5.1]{ThomasonTrobaugh-higher-algebraic-K-theory-of-schemes-and-of-derived-categories}. It is, however, a property of the underlying additive category and gives rise to the following definition.

\begin{defn}
\label{def:weakly-idempotent-complete}
\cite[Def.\ 7.2]{Buhler-exact-categories} 
An additive category is \emph{weakly idempotent complete} if every retraction has a kernel. 
\end{defn}

\cref{def:weakly-idempotent-complete} is actually self-dual. 
Indeed, in an additive category, every retraction has a kernel if and only if every section has a cokernel; see e.g.\ {\cite[Lem.\ 7.1]{Buhler-exact-categories}}. 

If $r \colon X \to Y$ is a retraction in $\CA$, with corresponding section $s \colon Y \to X$, and $r$ admits a kernel $k$, then the split idempotent $e\deff sr\in\tensor*[]{\End}{_{\CA}}(X)$ also has kernel $k$. 
Conversely, if $e\colon X\to X$ is a split idempotent, with splitting given by $e=sr$ where $r \colon X \to Y$, then a kernel of $e$ is also a kernel of $r$. 
Therefore, weakly idempotent complete categories are those additive categories in which split idempotents admit kernels, in contrast to idempotent complete categories in which all idempotents admit kernels (see \cref{prop:characterisation-of-idempotent-complete}).

\begin{defn}
\label{defn:weak-idempotent-completion}
The \emph{weak idempotent completion $\wh{\CA}$ of $\CA$} is the full subcategory of $\wt{\CA}$ consisting of all objects $(X, e)\in \wt{\CA}$ such that $\id{X} - e$ is a split idempotent in $\CA$. 
\end{defn}

\begin{rem}
\label{rem:WIC-defn-different-from-Buhler}
We note that \cref{defn:weak-idempotent-completion} above differs slightly from 
the definition of the weak idempotent completion of $\CA$ suggested in \cite[Rem.\ 7.8]{Buhler-exact-categories}. If, as in \cite{Buhler-exact-categories}, we ask that objects of $\wh{\CA}$ are pairs $(X,e)$ where $e\colon X\to X$ splits, then $\wh{\CA}$ is equivalent to $\CA$. 
Indeed, if $sr=e$ and $rs=\id{Y}$, where $r\colon X\to Y$ and $s\colon Y\to X$, 
then $(X,e) \iso (Y,\id{Y})$ in $\wt{\CA}$ by \cref{lem:split-summand}. 
That is, we have not added any objects that are not already isomorphic to some object of $\tensor*[]{\SI}{_{\CA}}(\CA)$. 
On the other hand, if we take objects in $\wh{\CA}$ to be pairs $(X,e)$ where $\id{X}-e$ splits (as in  \cref{defn:weak-idempotent-completion}), then we have $(X,\id{X}) \iso (X,e) \oplus (Y',\id{Y'})$ in $\wh{\CA}$, where $s'r'=\id{X}-e$ and $r's'=\id{Y'}$, where $r'\colon X\to Y'$ and $s'\colon Y'\to X$. 
In this case, since $(X,\id{X}-e) \iso (Y',\id{Y'})$ in $\wt{\CA}$, we see that a ``complementary'' summand of $(X,\id{X}-e)$ in $(X,\id{X})$ has been added. This discrepancy has been noticed previously; see e.g.\ Henrard--van~Roosmalen \cite[Prop.\ A.11]{HenrardvanRoosmalen-Derived-categories-of-one-sided-exact-categories-and-their-localizations}. 
\end{rem}

It follows that $\wh{\CA}$ is an additive subcategory of $\wt{\CA}$ and that it is weakly idempotent complete; see e.g.\ \cite[Rem.\ 7.8]{Buhler-exact-categories} or \cite[Sec.\ A.2]{HenrardvanRoosmalen-Derived-categories-of-one-sided-exact-categories-and-their-localizations}. From this observation, we immediately have the next lemma.

\begin{lem} 
\label{lem:isoclosure-chat}
        Suppose $\wt{X}, \wt{Y}, \wt{Z} \in \wt{\CA}$ with $\wt{X} \oplus \wt{Y} \iso \wt{Z}$. Then any two of $\wt{X}, \wt{Y}, \wt{Z}$ being isomorphic to objects in $\wh{\CA}$ implies that the third object is also isomorphic to an object in $\wh{\CA}$. 
\end{lem}

Analogously to the construction in Subsection~\ref{subsec:idempotent-completion}, there is an inclusion functor $\tensor*[]{\SK}{_{\CA}}\colon \CA \to \wh{\CA}$, given by $\tensor*[]{\SK}{_{\CA}}(X) \deff (X,\id{X})$ on objects, which is $2$-universal among additive functors from $\CA$ to weakly idempotent complete categories; see e.g.\ \cite[Rem.\ 1.12]{Neeman-the-derived-category-of-an-exact-category} or \cite[Rem.\ 7.8]{Buhler-exact-categories}.

\begin{prop}
\label{prop:universal-property-of-WIC}
For any additive functor $\SF \colon \CA \to \CB$ with $\CB$ weakly idempotent complete:
\begin{enumerate}[label=\textup{(\roman*)}]
    \item there is an additive functor $\SE \colon \wh{\CA} \to \CB$ and a natural isomorphism \mbox{${\Tsadi \colon \SF \overset{\iso}{\Longrightarrow} \SE \tensor*[]{\SK}{_{\CA}}}$;} 
    and, in addition,
        
    \item for any additive functor
    $\SG \colon \wh{\CA} \to \CB$
    and any natural transformation $\Daleth \colon \SF \Rightarrow \SG \tensor*[]{\SK}{_{\CA}}$, there exists a unique natural transformation $\Mem \colon \SE \Rightarrow \SG$ with $\Daleth = \tensor*[]{\Mem}{_{\tensor*[]{\SK}{_{\CA}}}} \Tsadi$.
\end{enumerate}
\end{prop}

Let $\tensor*[]{\SL}{_{\wh{\CA}}}\colon \wh{\CA} \to \wt{\CA}$ denote the inclusion functor of the subcategory $\wh{\CA}$ into $\wt{\CA}$. 
The functor $\tensor*[]{\SI}{_{\CA}}\colon \CA \to \wt{\CA}$ factors through $\tensor*[]{\SK}{_{\CA}}$ as $\tensor*[]{\SI}{_{\CA}} = \tensor*[]{\SL}{_{\wh{\CA}}} \tensor*[]{\SK}{_{\CA}}$. 
An additive functor $\SF\colon \wh{\CA} \to \CB$ to a weakly idempotent complete category $\CB$ is determined up to unique natural isomorphism by its behaviour on the image $\tensor*[]{\SK}{_{\CA}}(\CA)$ of $\CA$ in $\wh{\CA}$; 
similarly, a natural transformation $\Beth\colon \SF \Rightarrow \SG$ of additive functors $\wh{\CA}\to \CB$ is also completely determined by its action on $\tensor*[]{\SK}{_{\CA}}(\CA)$; 
see \mbox{\cite[Rems.\ 6.7, 6.9]{Buhler-exact-categories}}.

\begin{rem}
\label{rem:set-theory}
In \cite[Rem.\ 7.9]{Buhler-exact-categories}, it is remarked that there is a subtle set-theoretic issue regarding the existence of the weak idempotent completion of an additive category. 
Let NBG denote von Neumann-Bernays-G\"{o}del class theory
(see Fraenkel--Bar-Hillel--Levy \mbox{\cite[p.\ 128]{FraenkelBar-HillelLevy-Foundations-of-set-theory}}), 
and let (AGC) denote the Axiom of Global Choice 
\cite[p.\ 133]{FraenkelBar-HillelLevy-Foundations-of-set-theory}. 
The combination NBG + (AGC) is a conservative extension of ZFC 
\mbox{\cite[p.\ 131--132, 134]{FraenkelBar-HillelLevy-Foundations-of-set-theory}.} 
If one chooses an appropriate class theory to work with, such as NBG + (AGC), then the weak idempotent completion always exists as a category. 
This would follow from the Axiom of Predicative Comprehension for Classes (see \cite[p.\ 123]{FraenkelBar-HillelLevy-Foundations-of-set-theory}); 
this is also known as the Axiom of Separation (e.g.\ Smullyan--Fitting \cite[p.\ 15]{SmullyanFitting-Set-theory-and-the-continuum-problem}). 
Furthermore, a priori it is not clear to the authors if \cref{prop:Buehler-universal-property-of-karoubi-envelope,prop:universal-property-of-WIC} follow in an arbitrary setting without (AGC). 
This is because in showing that, for example, an additive functor $\SF\colon \wt{\CA} \to \CB$, where $\CB$ is idempotent complete, is determined by its values on $\tensor*[]{\SI}{_{\CA}}(\CA)$, one must \emph{choose} a kernel and an image of the idempotent $\SF(e)$ for each idempotent $e$ in $\CA$.
\end{rem}

\newpage
\section{\texorpdfstring{$n$}{n}-exangulated categories, functors and natural transformations}
\label{sec:n-exangulated-categories-functors-nat-transformations}

Let $n\geq 1$ be an integer. 
In this section we recall the theory of $n$-exangulated categories established 
in  
\cite{HerschendLiuNakaoka-n-exangulated-categories-I-definitions-and-fundamental-properties}, 
$n$-exangulated functors as defined in \cite{Bennett-TennenhausShah-transport-of-structure-in-higher-homological-algebra}, and 
$n$-exangulated natural transformations as recently introduced in \cite{Bennett-TennenhausHauglandSandoyShah-the-category-of-extensions-and-a-characterisation-of-n-exangulated-functors}. 
We also use this opportunity to set up some notation.

\subsection{\texorpdfstring{$n$}{n}-exangulated categories}
\label{sec:n-exangulated-categories}

The definitions in this subsection and more details can be found in \cite[Sec.\ 2]{HerschendLiuNakaoka-n-exangulated-categories-I-definitions-and-fundamental-properties}.
For this subsection, suppose that $\CC$ is an additive category and that $\BE\colon\tensor*[]{\CC}{^{\op}}\times\CC\to \Ab$ is a biadditive functor.

Let $A,C$ be objects in $\CC$. 
We denote by $\tensor*[_{A}]{0}{_{C}}$ the identity element of the abelian group $\BE(C,A)$. 
Suppose $\delta\in\BE(C,A)$ and that $a\colon A\to B$ and $d\colon D\to C$ are morphisms in $\CC$. 
We put 
$
\tensor*[]{a}{_{\BE}}\delta\coloneqq\BE(C,a)(\delta)\in\BE(C,B)$
and
$
\tensor*[]{d}{^{\BE}}\delta\coloneqq\BE(d,A)(\delta)\in\BE(D,A)
$.
Since $\BE$ is a bifunctor, we have that 
$
\tensor*[]{d}{^{\BE}}\tensor*[]{a}{_{\BE}}\delta
    = \BE(d,a)(\delta)
    = \tensor*[]{a}{_{\BE}}\tensor*[]{d}{^{\BE}}\delta
$. 

An \emph{$\BE$-extension} is an element $\delta\in\BE(C,A)$ for some $A,C\in\CC$. A \emph{morphism of $\BE$-extensions} from $\delta\in\BE(C,A)$ to $\rho\in\BE(D,B)$ is given by a pair $(a,c)$ of morphisms $a\colon A\to B$ and $c\colon C\to D$ in $\CC$ such that 
$\tensor*[]{a}{_{\BE}}\delta = \tensor*[]{c}{^{\BE}}\rho$. 

Let $\begin{tikzcd}[column sep=0.6cm]
A 
& A\oplus B \arrow{l}[swap,pos=0.3]{\tensor*[]{p}{_{A}}}\arrow{r}{\tensor*[]{p}{_{B}}}
& B
\end{tikzcd}$ 
be a product and 
$\begin{tikzcd}[column sep=0.6cm]
C \arrow{r}{\tensor*[]{i}{_{C}}}
& C\oplus D 
& D \arrow{l}[swap,pos=0.4]{\tensor*[]{i}{_{D}}}
\end{tikzcd}$
be a coproduct 
in $\CC$, 
and let ${\delta\in\BE(C,A)}$ and $\rho\in\BE(D,B)$ be $\BE$-extensions.
The \emph{direct sum of $\delta$ and $\rho$} is the unique $\BE$-extension $\delta\oplus\rho\in\BE(C\oplus D,A\oplus B)$ such that the following equations hold. 
\begin{align*}
\BE(\tensor*[]{i}{_{C}},\tensor*[]{p}{_{A}})(\delta\oplus\rho) &=  \delta\\
\BE(\tensor*[]{i}{_{C}},\tensor*[]{p}{_{B}})(\delta\oplus\rho) &=  \tensor*[_{B}]{0}{_{C}}\\
\BE(\tensor*[]{i}{_{D}},\tensor*[]{p}{_{A}})(\delta\oplus\rho) &=  \tensor*[_{A}]{0}{_{D}}\\
\BE(\tensor*[]{i}{_{D}},\tensor*[]{p}{_{B}})(\delta\oplus\rho) &=  \rho
\end{align*}

From the Yoneda Lemma, each $\BE$-extension $\delta\in\BE(C,A)$ induces two natural transformations. The first is 
$\tensor*[_{\BE}]{\delta}{}\colon\CC(A,-) \Rightarrow \BE(C,-)$ 
given by 
$\tensor*[_{\BE}]{\delta}{_{B}}(a) \deff \tensor*[]{a}{_{\BE}}\delta$
for all objects $B\in\CC$ and all morphisms $a\colon A\to B$. 
The second is 
$\tensor*[^{\BE}]{\delta}{}\colon \CC(-,C)\Rightarrow \BE(-,A)$ 
and defined by 
$\tensor*[^{\BE}]{\delta}{_{D}}(d) \deff \tensor*[]{d}{^{\BE}}\delta$ 
for all objects $D\in\CC$ and all morphisms $d\colon D\to C$.

Let $\com{\CC}$ be the category of complexes in $\CC$. 
Its full subcategory consisting of complexes concentrated in degrees $0,1,\ldots, n,n+1$ is denoted $\com{\CC}^{\raisebox{0.5pt}{\scalebox{0.6}{$n$}}}$. 
If $\tensor*[]{X}{_{\bullet}}\in\com{\CC}^{\raisebox{0.5pt}{\scalebox{0.6}{$n$}}}$, we 
depict
$\tensor*[]{X}{_{\bullet}}$ as 
\[
\begin{tikzcd}
\tensor*[]{X}{_{0}} \arrow{r}{\tensor*[]{d}{^{X}_{0}}} & \tensor*[]{X}{_{1}} \arrow{r}{\tensor*[]{d}{^{X}_{1}}}& \cdots \arrow{r}{\tensor*[]{d}{^{X}_{n-1}}}& \tensor*[]{X}{_{n}} \arrow{r}{\tensor*[]{d}{^{X}_{n}}} & \tensor*[]{X}{_{n+1}},
\end{tikzcd}
\]
omitting the trails of zeroes at each end.

\begin{defn}
\label{def:attached-complexes-n-exangles-and-morphisms}
Let $\tensor*[]{X}{_{\bullet}}, \tensor*[]{Y}{_{\bullet}}\in\com{\CC}^{\raisebox{0.5pt}{\scalebox{0.6}{$n$}}}$ be complexes, and suppose that $\delta\in\BE(\tensor*[]{X}{_{n+1}},\tensor*[]{X}{_{0}})$ and $\rho\in\BE(\tensor*[]{Y}{_{n+1}},\tensor*[]{Y}{_{0}})$ are $\BE$-extensions.  
\begin{enumerate}[label=\textup{(\roman*)}]
    \item The pair $\lan \tensor*[]{X}{_{\bullet}},\delta\ran$ is known as an \emph{$\BE$-attached complex} if $(\tensor*[]{d}{_{0}^{X}}\tensor*[]{)}{_{\BE}}\delta = 0$ and $(\tensor*[]{d}{^{X}_{n}}\tensor*[]{)}{^{\BE}}\delta = 0$. 
    An $\BE$-attached complex $\lan \tensor*[]{X}{_{\bullet}},\delta\ran$ is called an \emph{$n$-exangle} (for $(\CC,\BE)$) if, further, 
the sequences
\begin{center}
\begin{tikzcd}[column sep=0.5cm]
\CC(-,\tensor*[]{X}{_{0}})\arrow[Rightarrow]{r}[yshift=2pt]{\CC(-,\tensor*[]{d}{^{X}_{0}})}
&[1.1cm]\CC(-,\tensor*[]{X}{_{1}})\arrow[Rightarrow]{r}[yshift=2pt]{\CC(-,\tensor*[]{d}{^{X}_{1}})}
&[1.1cm]\cdots\arrow[Rightarrow]{r}[yshift=2pt]{\CC(-,\tensor*[]{d}{^{X}_{n}})}
&[1.1cm]\CC(-,\tensor*[]{X}{_{n+1}})\arrow[Rightarrow]{r}[yshift=2pt]{\tensor*[^{\BE}]{\delta}{}}
& \BE(-,\tensor*[]{X}{_{0}})
\end{tikzcd}
\begin{flushleft}and\end{flushleft}
\begin{tikzcd}[column sep=0.5cm]
\CC(\tensor*[]{X}{_{n+1}},-)\arrow[Rightarrow]{r}[yshift=2pt]{\CC(\tensor*[]{d}{^{X}_{n}},-)}
&[1.1cm]\CC(\tensor*[]{X}{_{n}},-)\arrow[Rightarrow]{r}[yshift=2pt]{\CC(\tensor*[]{d}{^{X}_{n-1}},-)}
&[1.1cm]\cdots\arrow[Rightarrow]{r}[yshift=2pt]{\CC(\tensor*[]{d}{^{X}_{0}},-)}
&[1.1cm]\CC(\tensor*[]{X}{_{0}},-)\arrow[Rightarrow]{r}[yshift=2pt]{\tensor*[_{\BE}]{\delta}{}}
& \BE(\tensor*[]{X}{_{n+1}},-)
\end{tikzcd}
\end{center}
of functors are exact.

\item A \emph{morphism $\tensor*[]{f}{_{\bullet}}\colon \lan \tensor*[]{X}{_{\bullet}},\delta\ran \to \lan \tensor*[]{Y}{_{\bullet}},\rho\ran$ of $\BE$-attached complexes} 
is given by a morphism $\tensor*[]{f}{_{\bullet}}\in\com{\CC}^{\raisebox{0.5pt}{\scalebox{0.6}{$n$}}}(\tensor*[]{X}{_{\bullet}},\tensor*[]{Y}{_{\bullet}})$ such that $(\tensor*[]{f}{_{0}}\tensor*[]{)}{_{\BE}} \delta = (\tensor*[]{f}{_{n+1}}\tensor*[]{)}{^{\BE}}\rho$. 
Such an $\tensor*[]{f}{_{\bullet}}$ is called a \emph{morphism of $n$-exangles} if $\lan \tensor*[]{X}{_{\bullet}},\delta\ran$ and $\lan \tensor*[]{Y}{_{\bullet}},\rho\ran$ are both $n$-exangles.

\item The \emph{direct sum of the $\BE$-attached complexes} (or \emph{the $n$-exangles}) $\lan \tensor*[]{X}{_{\bullet}},\delta\ran$ and $\lan \tensor*[]{Y}{_{\bullet}},\rho\ran$ is the pair
$\lan \tensor*[]{X}{_{\bullet}} \oplus \tensor*[]{Y}{_{\bullet}},\delta\oplus\rho\ran$. 
\end{enumerate}
\end{defn}

From the definition above, one can form the additive category of $\BE$-attached complexes, and its additive full subcategory of $n$-exangles.

Given a pair of objects $A,C\in\CC$, we define a subcategory 
$\com{\CC}_{(A,C)}^{\raisebox{0.5pt}{\scalebox{0.6}{$n$}}}$  
of $\com{\CC}^{\raisebox{0.5pt}{\scalebox{0.6}{$n$}}}$
in the following way. 
An object $\tensor*[]{X}{_{\bullet}}\in\com{\CC}_{(A,C)}^{\raisebox{0.5pt}{\scalebox{0.6}{$n$}}}$ is an object of 
$\com{\CC}^{\raisebox{0.5pt}{\scalebox{0.6}{$n$}}}$ that satisfies 
$\tensor*[]{X}{_{0}}=A$ and $\tensor*[]{X}{_{n+1}}=C$. 
For $\tensor*[]{X}{_{\bullet}}, \tensor*[]{Y}{_{\bullet}} \in\com{\CC}_{(A,C)}^{\raisebox{0.5pt}{\scalebox{0.6}{$n$}}}$, 
a morphism $\tensor*[]{f}{_{\bullet}}\in\com{\CC}_{(A,C)}^{\raisebox{0.5pt}{\scalebox{0.6}{$n$}}}(\tensor*[]{X}{_{\bullet}},\tensor*[]{Y}{_{\bullet}})$ is a 
morphism ${\tensor*[]{f}{_{\bullet}} = (f_{0},\ldots, \tensor*[]{f}{_{n+1}})\in\com{\CC}^{\raisebox{0.5pt}{\scalebox{0.6}{$n$}}}(\tensor*[]{X}{_{\bullet}},\tensor*[]{Y}{_{\bullet}})}$ with 
$f_{0}=\id{A}$ and $\tensor*[]{f}{_{n+1}}=\id{C}$. 
Note that this implies $\com{\CC}_{(A,C)}^{\raisebox{0.5pt}{\scalebox{0.6}{$n$}}}$ is not necessarily a full subcategory of $\com{\CC}^{\raisebox{0.5pt}{\scalebox{0.6}{$n$}}}$, nor necessarily additive.

Let $\tensor*[]{X}{_{\bullet}}, \tensor*[]{Y}{_{\bullet}} \in\com{\CC}_{(A,C)}^{\raisebox{0.5pt}{\scalebox{0.6}{$n$}}}$ be complexes. 
Two morphisms in $\com{\CC}_{(A,C)}^{\raisebox{0.5pt}{\scalebox{0.6}{$n$}}}(\tensor*[]{X}{_{\bullet}},\tensor*[]{Y}{_{\bullet}})$ are said to be 
\emph{homotopic} if they are homotopic in the standard sense viewed as morphisms in $\com{\CC}^{\raisebox{0.5pt}{\scalebox{0.6}{$n$}}}$. 
This induces an equivalence relation $\sim$ on $\com{\CC}_{(A,C)}^{\raisebox{0.5pt}{\scalebox{0.6}{$n$}}}(\tensor*[]{X}{_{\bullet}},\tensor*[]{Y}{_{\bullet}})$. 
We define $\kom{\CC}_{(A,C)}^{\raisebox{0.5pt}{\scalebox{0.6}{$n$}}}$ as the category with the same objects as $\com{\CC}_{(A,C)}^{\raisebox{0.5pt}{\scalebox{0.6}{$n$}}}$ and with
$
\kom{\CC}_{(A,C)}^{\raisebox{0.5pt}{\scalebox{0.6}{$n$}}}(\tensor*[]{X}{_{\bullet}},\tensor*[]{Y}{_{\bullet}})\deff \com{\CC}_{(A,C)}^{\raisebox{0.5pt}{\scalebox{0.6}{$n$}}}(\tensor*[]{X}{_{\bullet}},\tensor*[]{Y}{_{\bullet}})/{\sim}
$.

A morphism  $\tensor*[]{f}{_{\bullet}}\in\com{\CC}_{(A,C)}^{\raisebox{0.5pt}{\scalebox{0.6}{$n$}}}(\tensor*[]{X}{_{\bullet}},\tensor*[]{Y}{_{\bullet}})$ is called a \emph{homotopy equivalence} if its image in the category $\kom{\CC}_{(A,C)}^{\raisebox{0.5pt}{\scalebox{0.6}{$n$}}}(\tensor*[]{X}{_{\bullet}},\tensor*[]{Y}{_{\bullet}})$ is an isomorphism. In this case,  $\tensor*[]{X}{_{\bullet}}$ and $\tensor*[]{Y}{_{\bullet}}$ are said to be \emph{homotopy equivalent}. 
The isomorphism class of $\tensor*[]{X}{_{\bullet}}$ in $\kom{\CC}_{(A,C)}^{\raisebox{0.5pt}{\scalebox{0.6}{$n$}}}$ (equivalently, its homotopy class in $\com{\CC}_{(A,C)}^{\raisebox{0.5pt}{\scalebox{0.6}{$n$}}}$) is denoted $[\tensor*[]{X}{_{\bullet}}]$. 
Since the (usual) homotopy class of $\tensor*[]{X}{_{\bullet}}$ in $\com{\CC}$ may differ from its homotopy class in $\com{\CC}_{(A,C)}^{\raisebox{0.5pt}{\scalebox{0.6}{$n$}}}$, we reserve the notation $[\tensor*[]{X}{_{\bullet}}]$ specifically for its isomorphism class in $\kom{\CC}_{(A,C)}^{\raisebox{0.5pt}{\scalebox{0.6}{$n$}}}$.

\begin{notn}
\label{notn:triv}
For $X \in \CC$ and $i \in\{ 0, \dots, n\}$, 
we denote by $\tensor*[]{\triv}{_{i}}(X)_{\bullet}$ the object in $\com{\CC}^{\raisebox{0.5pt}{\scalebox{0.6}{$n$}}}$ 
given by $\tensor*[]{\triv}{_{i}}(X)_j = X$ for $j = i, i+1$ and $\tensor*[]{\triv}{_{i}}(X)_j = 0$ for $0 \leq j \leq i-1$ and $i+2 \leq j \leq n+1$, as well as $\tensor*[]{d}{^{\, \tensor*[]{\triv}{_{i}}(X)}_{i}} = \id{X}$.
\end{notn}

\begin{defn}
\label{def:exact-realisation}
Let $\fs$ be an assignment that, for each pair of objects $A,C\in\CC$ and each $\BE$-extension ${\delta\in\BE(C,A)}$, 
associates to $\delta$ 
an isomorphism class $\fs(\delta)=[\tensor*[]{X}{_{\bullet}}]$ 
in 
$\kom{\CC}_{(A,C)}^{\raisebox{0.5pt}{\scalebox{0.6}{$n$}}}$. 
The correspondence $\fs$ is called an \emph{exact realisation of $\BE$} if it satisfies the following conditions. 

\begin{enumerate}[label=\textup{(R\arabic*)},
    labelsep=5pt, 
    leftmargin=35.00003pt,
]
\setcounter{enumi}{-1}

    \item\label{R0}
    For any morphism $(a,c)\colon \delta\to\rho$ of $\BE$-extensions 
    with $\delta\in\BE(C,A)$, $\rho\in\BE(D,B)$, $\fs(\delta)=[\tensor*[]{X}{_{\bullet}}]$ and $\fs(\rho)=[\tensor*[]{Y}{_{\bullet}}]$, 
    there exists $\tensor*[]{f}{_{\bullet}}\in\com{\CC}^{\raisebox{0.5pt}{\scalebox{0.6}{$n$}}}(\tensor*[]{X}{_{\bullet}},\tensor*[]{Y}{_{\bullet}})$ such that $f_{0}=a$ and $\tensor*[]{f}{_{n+1}}=c$. 
    In this setting, we say that $\tensor*[]{X}{_{\bullet}}$ \emph{realises} $\delta$ and 
    $\tensor*[]{f}{_{\bullet}}$ is a \emph{lift} of $(a,c)$.

    \item\label{R1}
    If $\fs(\delta)=[\tensor*[]{X}{_{\bullet}}]$, then $\lan \tensor*[]{X}{_{\bullet}},\delta\ran$ is an $n$-exangle.

    \item\label{R2}
    For each object $A\in\CC$, we have   
    $\fs(\tensor*[_{A}]{0}{_{0}})=[
    \tensor*[]{\triv}{_{0}}(A)_{\bullet}
    ]$ 
    and 
    $\fs(\tensor*[_{0}]{0}{_{A}})=[
    \tensor*[]{\triv}{_{n}}(A)_{\bullet}
    ]$.
\end{enumerate}

In case $\fs$ is an exact realisation of $\BE$ and 
$
\fs(\delta) 
	= [\tensor*[]{X}{_{\bullet}}],
$
the following terminology is used. 
The morphism $\tensor*[]{d}{^{X}_{0}}$ is said to be an \emph{$\fs$-inflation} 
and the morphism $\tensor*[]{d}{^{X}_{n}}$ an \emph{$\fs$-deflation}. 
The pair $\lan \tensor*[]{X}{_{\bullet}},\delta\ran$ is known as an \emph{$\fs$-distinguished} $n$-exangle.
\end{defn}

Suppose $\fs$ is an exact realisation of $\BE$ and $\fs(\delta) = [\tensor*[]{X}{_{\bullet}}]$. 
We will often use the diagram
\[
\begin{tikzcd}
\tensor*[]{X}{_{0}} \arrow{r}{\tensor*[]{d}{_{0}^{X}}}
	& \tensor*[]{X}{_{1}} \arrow{r}{\tensor*[]{d}{^{X}_{1}}}
	& \cdots \arrow{r}{\tensor*[]{d}{^{X}_{n-1}}}
	& \tensor*[]{X}{_{n}} \arrow{r}{\tensor*[]{d}{^{X}_{n}}}
	& \tensor*[]{X}{_{n+1}} \arrow[r, "\delta", dashed]
	& {}
\end{tikzcd}
\]
to 
express
that $\lan \tensor*[]{X}{_{\bullet}},\delta\ran$ is an $\fs$-distinguished $n$-exangle. 
If we also have that $\fs(\rho)=[\tensor*[]{Y}{_{\bullet}}]$ and $\tensor*[]{f}{_{\bullet}} \colon \lan \tensor*[]{X}{_{\bullet}},\delta\ran \to \lan \tensor*[]{Y}{_{\bullet}},\rho\ran$ 
is a morphism of $n$-exangles, then we call $\tensor*[]{f}{_{\bullet}}$ a \emph{morphism of $\fs$-distinguished $n$-exangles} and we depict this by the following commutative diagram. 
\[
\begin{tikzcd}
\tensor*[]{X}{_{0}} \arrow{r}{\tensor*[]{d}{_{0}^{X}}}
    \arrow{d}{\tensor*[]{f}{_{0}}}
	& \tensor*[]{X}{_{1}} \arrow{r}{\tensor*[]{d}{^{X}_{1}}}
	    \arrow{d}{\tensor*[]{f}{_{1}}}
	& \cdots \arrow{r}{\tensor*[]{d}{^{X}_{n-1}}}
	& \tensor*[]{X}{_{n}} \arrow{r}{\tensor*[]{d}{^{X}_{n}}}
	    \arrow{d}{\tensor*[]{f}{_{n}}}
	& \tensor*[]{X}{_{n+1}} \arrow[r, "\delta", dashed]
	    \arrow{d}{\tensor*[]{f}{_{n+1}}}
	& {}
\\
\tensor*[]{Y}{_{0}} \arrow{r}{\tensor*[]{d}{_{0}^{Y}}}
	& \tensor*[]{Y}{_{1}} \arrow{r}{\tensor*[]{d}{^{Y}_{1}}}
	& \cdots \arrow{r}{\tensor*[]{d}{^{Y}_{n-1}}}
	& \tensor*[]{Y}{_{n}} \arrow{r}{\tensor*[]{d}{^{Y}_{n}}}
	& \tensor*[]{Y}{_{n+1}} \arrow[r, "\rho", dashed]
	& {}
\end{tikzcd}
\]

We need one last definition before being able to define an $n$-exangulated category.

\begin{defn}
\label{def:mapping-cone}
Suppose 
$
\tensor*[]{f}{_{\bullet}}\colon \tensor*[]{X}{_{\bullet}} \to \tensor*[]{Y}{_{\bullet}}
$ 
is a morphism in $\com{\CC}^{\raisebox{0.5pt}{\scalebox{0.6}{$n$}}}$, such that 
$f_{0}=\id{A}$ for some $A=\tensor*[]{X}{_{0}}=\tensor*[]{Y}{_{0}}$. 
The \emph{mapping cone} ${M^{\CC}_{f}}_{\bullet}\in \com{\CC}^{\raisebox{0.5pt}{\scalebox{0.6}{$n$}}}$ of $\tensor*[]{f}{_{\bullet}}$ is the complex 
\[
\begin{tikzcd}
\tensor*[]{X}{_{1}}
    \arrow{r}{d^{M^{\CC}_{f}}_{0}}
& \tensor*[]{X}{_{2}}\oplus \tensor*[]{Y}{_{1}} 
    \arrow{r}{d^{M^{\CC}_{f}}_{1}}
&\tensor*[]{X}{_{3}}\oplus \tensor*[]{Y}{_{2}}
    \arrow{r}{d^{M^{\CC}_{f}}_{2}}
& \cdots
    \arrow{r}{d^{M^{\CC}_{f}}_{n-1}}
&\tensor*[]{X}{_{n+1}}\oplus \tensor*[]{Y}{_{n}} 
    \arrow{r}{d^{M^{\CC}_{f}}_{n}}
& \tensor*[]{Y}{_{n+1}},
\end{tikzcd}
\]
with 
$
d^{M^{\CC}_{f}}_{0}
    \deff 
        \begin{bsmallmatrix}
        -\tensor*[]{d}{^{X}_{1}} &\;
        f_{1} 
        \end{bsmallmatrix}^\top
$, 
$
d^{M^{\CC}_{f}}_{n}
    \deff
        \begin{bsmallmatrix} \tensor*[]{f}{_{n+1}} &\; \tensor*[]{d}{^{Y}_{n}} \end{bsmallmatrix}
$, 
and 
$
d^{M^{\CC}_{f}}_{i}
    \deff
        \begin{bsmallmatrix}
        -\tensor*[]{d}{^{X}_{i+1}} & 0\\
        \tensor*[]{f}{_{i+1}} & \tensor*[]{d}{^{Y}_{i}} 
        \end{bsmallmatrix}
$ 
for $i\in\{1,\ldots,n-1\}$. 
\end{defn}

We are in position to state the main definition of this subsection.

\begin{defn}
\label{def:n-exangulated-category}
An \emph{$n$-exangulated category} is a triplet $(\CC,\BE,\fs)$, 
consisting of  
an additive category $\CC$, 
a biadditive functor $\BE\colon\tensor*[]{\CC}{^{\op}}\times\CC\to \Ab$ 
and an exact realisation $\fs$ of $\BE$, such that the following conditions are met.

\begin{enumerate}[label=\textup{(EA\arabic*)},
wide=0pt, leftmargin=46pt, labelwidth=41pt, labelsep=5pt, align=right]

    \item\label{EA1}
    The collection of $\fs$-inflations is closed under composition. 
    Dually, the collection of $\fs$-deflations is closed under composition.
        
    \item\label{EA2}
    Suppose $\delta\in\BE(D,A)$ and $c\in\CC(C,D)$. 
    If $\fs(\tensor*[]{c}{^{\BE}}\delta)=[\tensor*[]{Y}{_{\bullet}}]$ and $\fs(\delta)=[\tensor*[]{X}{_{\bullet}}]$, 
    then there exists a morphism  
    $\tensor*[]{f}{_{\bullet}}\colon \tensor*[]{Y}{_{\bullet}}\to \tensor*[]{X}{_{\bullet}}$ 
    lifting $(\id{A},c)$, 
    such that $\fs((\tensor*[]{d}{^{Y}_{0}}\tensor*[]{)}{_{\BE}}\delta)=[ {M^{\CC}_{f}}_{\bullet}]$. 
    In this case, the morphism $\tensor*[]{f}{_{\bullet}}$ is called a \emph{good lift of $(\id{A},c)$}.

    \item[\textup{(EA$2\tensor*[]{)}{^{\op}}$}]\label{EA2op}
    The dual of \ref{EA2}.
    
\end{enumerate}
\end{defn}

Notice that the definition of an $n$-exangulated category is self-dual. In particular, the dual statements of several  results in Sections~\ref{sec:the-idempotent-completion}--\ref{sec:the-WIC} are used without proof.

\subsection{\texorpdfstring{$n$}{n}-exangulated functors and natural transformations}
\label{sec:n-exangulated-functors}

In order to show that the canonical functor from an $n$-exangulated category $(\CC,\BE,\fs)$ to its idempotent completion is $2$-universal among structure-preserving functors from $(\CC,\BE,\fs)$ to idempotent complete $n$-exangulated categories, we will need the notion of a morphism of $n$-exangulated categories and that of a morphism between such morphisms.

For this subsection, suppose $(\CC,\BE,\fs)$, $(\CC',\BE',\fs')$ and $(\CC'',\BE'',\fs'')$ are $n$-exangulated categories. 
If $\SF\colon \CC \to \CC'$ is an additive functor, then it induces several other additive functors, e.g.\ $\tensor*[]{\SF}{^{\op}}\colon \tensor*[]{\CC}{^{\op}} \to (\CC'\tensor*[]{)}{^{\op}}$, 
or $\SF_{\ch}\colon \com{\CC} \to \com{\CC'}$ and obvious restrictions thereof. 
These are all defined in the usual way. 
However, by abuse of notation, we simply write $\SF$ for each of these.

\begin{defn}
\label{def:n-exangulated-functor}
\cite[Def.\ 2.32]{Bennett-TennenhausShah-transport-of-structure-in-higher-homological-algebra} 
Suppose that $\SF\colon \CC\to\CC'$ is an additive functor and that 
$
\Gamma
        \colon 
    \BE(-,-)
        \Rightarrow 
    \BE'(\SF-, \SF-)
$
is a natural transformation of functors $\tensor*[]{\CC}{^{\op}}\times \CC \to \Ab$.  
The pair $(\SF,\Gamma) \colon (\CC,\BE,\fs) \to (\CC',\BE',\fs')$ is called an 
\emph{$n$-exang\-ulat\-ed functor} if,
for all $A,C\in\CC$ and each $\delta\in\BE(A,C)$, we have that 
$\fs'(\tensor*[]{\Gamma}{_{(C,A)}}(\delta))=[\SF(\tensor*[]{X}{_{\bullet}})]$
whenever $\fs(\delta)=[\tensor*[]{X}{_{\bullet}}]$.
\end{defn}

If we have a sequence 
$
\begin{tikzcd}[column sep=1.2cm]
(\CC,\BE,\fs) 
	\arrow{r}{(\SF,\Gamma)}
& (\CC',\BE',\fs')
	\arrow{r}{(\SL,\Phi)}
& (\CC'',\BE'',\fs'')
\end{tikzcd}
$
of $n$-exangulated functors, then the \emph{composite} of $(\SF,\Gamma)$ and $(\SL,\Phi)$ 
is defined to be
\[
(\SL,\Phi) \circ (\SF,\Gamma)\deff(\SL \circ \SF, \tensor*[]{\Phi}{_{\SF\times\SF}} \circ\Gamma).
\] 
This is an $n$-exangulated functor $(\CC,\BE,\fs)\to (\CC'',\BE'',\fs'')$; see \cite[Lem.\ 3.19(ii)]{Bennett-TennenhausHauglandSandoyShah-the-category-of-extensions-and-a-characterisation-of-n-exangulated-functors}.

The next result implies that $n$-exangulated functors preserve finite direct sum decompositions of distinguished $n$-exangles. It will be used in the main result of Subsection~\ref{subsec:main-results}.

\begin{prop}
\label{prop:morphisms-of-n-exangles-preserved-under-n-exang-functor}
Let $\SF\colon \CC \to \CC'$ be an additive functor and ${\Gamma\colon\BE(-,-) \Rightarrow \BE'(\SF-,\SF-)}$ a natural transformation. 
Suppose $\delta\in\BE(C,A)$ and $\rho\in\BE(D,B)$ are $\BE$-extensions, and 
$\lan \tensor*[]{X}{_{\bullet}},\delta\ran$ and $\lan \tensor*[]{Y}{_{\bullet}}, \rho\ran$ are $\fs$-distinguished. 
\begin{enumerate}[label=\textup{(\roman*)}]
	
	\item\label{item:morphisms-of-attached-complexes-preserved}
	If $\tensor*[]{f}{_{\bullet}} \colon \lan \tensor*[]{X}{_{\bullet}},\delta\ran \to \lan \tensor*[]{Y}{_{\bullet}}, \rho\ran$ is a morphism of $\BE$-attached complexes, 
	then the induced morphism 
	$\SF(\tensor*[]{f}{_{\bullet}}) \colon \lan \SF(\tensor*[]{X}{_{\bullet}}),\Gamma_{(C,A)}(\delta)\ran \to \lan \SF(\tensor*[]{Y}{_{\bullet}}), \Gamma_{(D,B)}(\rho)\ran$ 
	is 
	a morphism of $\BE'$-at\-tach\-ed complexes. 
	
	\item\label{item:SF-Gamma-preserves-direct-sums-of-n-exangle}
	We have 
	$\lan \SF(\tensor*[]{X}{_{\bullet}} \oplus \tensor*[]{Y}{_{\bullet}}), \Gamma_{(C\oplus D, A\oplus B)}(\delta\oplus\rho)\ran \iso \lan \SF(\tensor*[]{X}{_{\bullet}}), \Gamma_{(C,A)}(\delta)\ran \oplus \lan \SF(\tensor*[]{Y}{_{\bullet}}), \Gamma_{(D,B)}(\rho)\ran $ as $\BE'$-attached complexes.
\end{enumerate}
\end{prop}

\begin{proof}
\ref{item:morphisms-of-attached-complexes-preserved}\; 
Note that 
$(\SF(\tensor*[]{d}{_{0}^{X}})\tensor*[]{)}{_{\BE'}}(\tensor*[]{\Gamma}{_{(\tensor*[]{X}{_{n+1}},\tensor*[]{X}{_{0}})}}(\delta)) 
	= \Gamma_{(\tensor*[]{X}{_{n+1}},\tensor*[]{X}{_{1}})} ((\tensor*[]{d}{_{0}^{X}}\tensor*[]{)}{_{\BE}}\delta)
	= \tensor*[_{\SF(\tensor*[]{X}{_{1}})}]{0}{_{\SF(\tensor*[]{X}{_{n+1}})}}
$
since $\Gamma$ is natural and $\lan \tensor*[]{X}{_{\bullet}},\delta\ran$ is an $\BE$-attached complex. 
Similar computations show that both 
$\lan \SF(\tensor*[]{X}{_{\bullet}}),\Gamma_{(C,A)}(\delta)\ran$
and 
$\lan \SF(\tensor*[]{Y}{_{\bullet}}), \Gamma_{(D,B)}(\rho)\ran$
are $\BE'$-attached complexes.
As $\SF(\tensor*[]{f}{_{\bullet}})$ is a morphism $\SF(\tensor*[]{X}{_{\bullet}})\to\SF(\tensor*[]{Y}{_{\bullet}})$ of complexes, it suffices to prove  \[{\SF(f_{0}\tensor*[]{)}{_{\BE'}}\tensor*[]{\Gamma}{_{(\tensor*[]{X}{_{n+1}},\tensor*[]{X}{_{0}})}}(\delta) = \SF(\tensor*[]{f}{_{n+1}}\tensor*[]{)}{^{\BE'}}\Gamma_{(\tensor*[]{Y}{_{n+1}},\tensor*[]{Y}{_{0}})}(\rho)}.\] 
This follows immediately from $(f_{0}\tensor*[]{)}{_{\BE}}\delta = (\tensor*[]{f}{_{n+1}}\tensor*[]{)}{^{\BE}}\rho$ and the naturality of $\Gamma$.

\ref{item:SF-Gamma-preserves-direct-sums-of-n-exangle}\; 
This follows from applying \ref{item:morphisms-of-attached-complexes-preserved} to the morphisms in the appropriate biproduct diagram of $\BE$-attached complexes.
\end{proof}

Lastly, we recall the notion of a morphism of $n$-exangulated functors. 
The extriangulated version was defined in Nakaoka--Ogawa--Sakai \cite[Def.\ 2.11(3)]{NakaokaOgawaSakai-localization-of-extriangulated-categories}.

\begin{defn}
\label{def:n-exangulated-natural-transformation}
\cite[Def.\ 4.1]{Bennett-TennenhausHauglandSandoyShah-the-category-of-extensions-and-a-characterisation-of-n-exangulated-functors} 
Suppose $(\SF,\Gamma), (\SG,\Lambda)\colon(\CC,\BE,\fs)\to(\CC',\BE',\fs')$ are $n$-exangulated functors. 
A natural transformation 
$\Beth\colon\SF\Rightarrow\SG$ 
of functors
is said to be \emph{$n$-exangulated} 
if, for all $A,C\in\CC$ and each $\delta\in\BE(C,A)$, we have 
    \begin{equation}
\label{eqn:n-exangulated-natural-transformation-property}
    (\tensor*[]{\Beth}{_{A}}\tensor*[]{)}{_{\BE'}}\tensor*[]{\Gamma}{_{(C,A)}}(\delta)
    = (\tensor*[]{\Beth}{_{C}}\tensor*[]{)}{^{\BE'}}\tensor*[]{\Lambda}{_{(C,A)}}(\delta).
\end{equation}
We denote this by $\Beth\colon(\SF,\Gamma)\Rightarrow(\SG,\Lambda)$. 
In addition, if $\Beth$ has an $n$-exangulated inverse, then it is called 
an \emph{$n$-exangulated natural isomorphism}. It is straightforward to check that $\Beth$ has an $n$-exangulated inverse if and only if 
$\tensor*[]{\Beth}{_{X}}$ is an isomorphism for each $X\in\CC$.
\end{defn}

\newpage
\section{The idempotent completion of an \texorpdfstring{$n$}{n}-exangulated category} \label{sec:the-idempotent-completion}

Throughout this section we work with the following setup.
\begin{setup}
Let $n\geq 1$ be an integer. 
Let $(\CC,\BE,\fs)$ be an $n$-exangulated category.
We denote by $\tensor*[]{\SI}{_{\CC}}$ the inclusion of the category $\CC$ into its idempotent completion $\wt{\CC}$; 
see \cref{sec:on-the-splitting-of-idempotents}.
\end{setup}

In this section, we will construct a biadditive functor 
$\BF\colon\tensor*[]{\wt{\CC}}{^{\op}}\times \wt{\CC}\to \Ab$ (see Subsection~\ref{subsec:def-of-F}) and 
an exact realisation $\ft$ of $\BF$ (see Subsection~\ref{subsec:defining-t}), 
and then show that
$(\wt{\CC},\BF,\ft)$ is an $n$-exangulated category (see Subsections~\ref{subsec:EA1}--\ref{subsec:main-results}).
For $n=1$, we recover the main results of \cite{Msapato-the-karoubi-envelope-and-weak-idempotent-completion-of-an-extriangulated-category}. 
First, we establish some notation to help our exposition.

\begin{notn}
\label{not:morphisms-objects-in-tildeC}
We reserve notation with a tilde for objects and morphisms in $\wt{\CC}$.
\begin{enumerate}[label={(\roman*)}]

    \item \label{item:identities-with-tilde}
    If $\wt{X}\in\wt{\CA}$ is some object, then we will denote the identity morphism of $\wt{X}$ by $\wid{\wt{X}}$. 
    Recall from \cref{defn:karoubi-envelope} that the identity of an object $(X,e)\in\wt{\CA}$ is 
    $\wid{(X,e)} = (e,e,e)$.
    
    \item \label{item:underlying-morphism}
    Given a morphism $(e',r,e)\in\wt{\CC}((X,e),(Y,e'))$, we call $r\colon X\to Y$ the \emph{underlying morphism of $(e',r,e)$}. 
    
    \item \label{item:lift-to-tildeC}
    Suppose $(X,e), (Y,e') \in \wt{\CC}$ and $r \in \CC(X, Y)$ with $e' r = r = r e$. 
    Then there is a unique morphism $\tilde{r} \in \wt{\CC}((X,e), (Y,e'))$ with underlying morphism $r$.
    This morphism $\tilde{r}$ is the triplet $(e', r, e)$. 
    Moreover, we will use this notation specifically for this correspondence. That is, we write $\tilde{s}\colon (X,e) \to (Y,e')$ is a morphism in $\wt{\CC}$ if and only if we implicitly mean that the underlying morphism of 
    $\tilde{s}$ is denoted $s$,
    i.e.\ we have 
    $\tilde{s} = (e',s,e)$.     
\end{enumerate}
\end{notn}

\begin{rem}
\label{rem:equality-of-morphisms-and-commutativity-in-tildeC}
By \cref{not:morphisms-objects-in-tildeC}\ref{item:lift-to-tildeC}, two morphisms $\tilde{r}, \tilde{s} \in \wt{\CC} ((X,e), (Y,e'))$ are equal if and only if their underlying morphisms $r$ and $s$, respectively, are equal in $\CC$. 
Thus, for all objects $\wt{X},\wt{Y}\in\wt{\CC}$, removing the tilde from morphisms in $\wt{\CC}(\wt{X},\wt{Y})$ defines an injective abelian group homomorphism 
$\wt{\CC}(\wt{X},\wt{Y}) \to \CC(X,Y)$. 
In particular, a diagram in $\wt{\CC}$ commutes if and only if its diagram of underlying morphisms commutes.
\end{rem}

\subsection{Defining the biadditive functor \texorpdfstring{$\BF$}{F}}
\label{subsec:def-of-F}

The following construction is the higher version of the one given in \cite[Sec.\ 3.1]{Msapato-the-karoubi-envelope-and-weak-idempotent-completion-of-an-extriangulated-category} for extriangulated categories.

\begin{defn}
\label{def:BF}
We define a 
functor $\BF\colon \tensor*[]{\wt{\CC}}{^{\op}} \times \wt{\CC} \to \Ab$ as follows. For objects $(\tensor*[]{X}{_{0}}, \tensor*[]{e}{_{0}})$ and $(\tensor*[]{X}{_{n+1}}, \tensor*[]{e}{_{n+1}})$ in $\wt{\CC}$, we put
\begin{align*}
\BF((\tensor*[]{X}{_{n+1}}, \tensor*[]{e}{_{n+1}}),(\tensor*[]{X}{_{0}}, \tensor*[]{e}{_{0}}))
    & \deff \Set{ (\tensor*[]{e}{_{0}}, \delta, \tensor*[]{e}{_{n+1}}) | \text{$\delta\in\BE(\tensor*[]{X}{_{n+1}},\tensor*[]{X}{_{0}})$ 
    and $(\tensor*[]{e}{_{0}}\tensor*[]{)}{_{\BE}} \delta = \delta = (\tensor*[]{e}{_{n+1}}\tensor*[]{)}{^{\BE}} \delta$}}.
\end{align*}
For morphisms 
$\tilde{a}\colon (\tensor*[]{X}{_{0}},\tensor*[]{e}{_{0}})\to (\tensor*[]{Y}{_{0}},e'_{0})$
and 
$\tilde{c}\colon (\tensor*[]{Z}{_{n+1}},e''_{n+1})\to (\tensor*[]{X}{_{n+1}},\tensor*[]{e}{_{n+1}})$
in $\wt{\CC}$, 
we define 
\begin{align*}
\BF(\tilde{c},\tilde{a}) \colon \BF((\tensor*[]{X}{_{n+1}}, \tensor*[]{e}{_{n+1}}), (\tensor*[]{X}{_{0}}, \tensor*[]{e}{_{0}})) 
    &\longrightarrow \BF((\tensor*[]{Z}{_{n+1}}, e''_{n+1}), (\tensor*[]{Y}{_{0}}, e'_{0})) \\ 
(\tensor*[]{e}{_{0}}, \delta, \tensor*[]{e}{_{n+1}}) 
    &\longmapsto (e'_{0}, \BE(c,a)(\delta), e''_{n+1}).
\end{align*}
\end{defn}

\begin{rem}
\label{rem:comments-on-BF}
We make some comments on \cref{def:BF}.
\begin{enumerate}[label=\textup{(\roman*)}]
    
    \item\label{item:BF-well-defined-on-morphisms} 
    The assignment $\BF$ on morphisms takes values where claimed due to the following. 
    For morphisms 
    $\tilde{a}\colon (\tensor*[]{X}{_{0}},\tensor*[]{e}{_{0}})\to (\tensor*[]{Y}{_{0}},e'_{0})$ and $\tilde{c}\colon (\tensor*[]{Z}{_{n+1}},e''_{n+1})\to (\tensor*[]{X}{_{n+1}},\tensor*[]{e}{_{n+1}})$, 
    and an $\BF$-extension 
    $(\tensor*[]{e}{_{0}},\delta,\tensor*[]{e}{_{n+1}}) \in \BF((\tensor*[]{X}{_{n+1}}, \tensor*[]{e}{_{n+1}}),(\tensor*[]{X}{_{0}}, \tensor*[]{e}{_{0}}))$, 
    we have
    \begin{align*}
    \BE(e''_{n+1}, e'_{0})\BE(c,a)(\delta)
        &= \BE(ce''_{n+1},e'_{0}a)(\delta) \\
        &= \BE(c,a)(\delta).
    \end{align*}
    Therefore, $\BF(\tilde{c},\tilde{a})(\tensor*[]{e}{_{0}},\delta,\tensor*[]{e}{_{n+1}}) = (e''_{n+1}, \BE(c,a)(\delta), e'_{0})$ lies in $\BF((\tensor*[]{Z}{_{n+1}},e''_{n+1}),(\tensor*[]{Y}{_{0}},e'_{0}))$. 
    It is then straightforward to verify that $\BF$ is indeed a functor.
    
    \item\label{item:BF-additive-subfunctor-of-BE}

    The set $\BF((\tensor*[]{X}{_{n+1}}, \tensor*[]{e}{_{n+1}}),(\tensor*[]{X}{_{0}}, \tensor*[]{e}{_{0}}))$ is an abelian group by defining
    \[(\tensor*[]{e}{_{0}},\delta, \tensor*[]{e}{_{n+1}}) + (\tensor*[]{e}{_{0}}, \rho, \tensor*[]{e}{_{n+1}}) \deff (\tensor*[]{e}{_{0}}, \delta + \rho, \tensor*[]{e}{_{n+1}})\]
    for $(\tensor*[]{e}{_{0}}, \delta, \tensor*[]{e}{_{n+1}}), (\tensor*[]{e}{_{0}}, \rho, \tensor*[]{e}{_{n+1}}) \in \BF((\tensor*[]{X}{_{n+1}}, \tensor*[]{e}{_{n+1}}),(\tensor*[]{X}{_{0}}, \tensor*[]{e}{_{0}}))$. 
    The 
    additive identity element 
    of $\BF((\tensor*[]{X}{_{n+1}}, \tensor*[]{e}{_{n+1}}),(\tensor*[]{X}{_{0}}, \tensor*[]{e}{_{0}}))$ is 
    $\tensor*[_{(\tensor*[]{X}{_{0}}, \tensor*[]{e}{_{0}})}]{\wt{0}}{_{(\tensor*[]{X}{_{n+1}}, \tensor*[]{e}{_{n+1}})}} \deff (\tensor*[]{e}{_{0}}, \tensor*[_{\tensor*[]{X}{_{0}}}]{0}{_{\tensor*[]{X}{_{n+1}}}}, \tensor*[]{e}{_{n+1}})$. 
    The inverse of $(\tensor*[]{e}{_{0}}, \delta, \tensor*[]{e}{_{n+1}})$ is 
    $(\tensor*[]{e}{_{0}}, -\delta, \tensor*[]{e}{_{n+1}})$. 
    Notice that we get an abelian group monomorphism:  
    \begin{align*}
    \BF((\tensor*[]{X}{_{n+1}}, \tensor*[]{e}{_{n+1}}),(\tensor*[]{X}{_{0}}, \tensor*[]{e}{_{0}})) 
        &\longrightarrow \BE(\tensor*[]{X}{_{n+1}}, \tensor*[]{X}{_{0}}) \\
    (\tensor*[]{e}{_{0}}, \delta, \tensor*[]{e}{_{n+1}})
        &\longmapsto \delta. 
    \end{align*}
    This homomorphism plays a role later in the proof of \cref{thm:IC-is-2-universal}. 
    
    \item\label{item:BF-biadditive} 
    It follows from the definition of $\BF$ that it is biadditive since $\BE$ is.
    
    \item\label{item:idempotents-give-morphism} Given $(\tensor*[]{e}{_{0}},\delta,\tensor*[]{e}{_{n+1}})\in\BF((\tensor*[]{X}{_{n+1}}, \tensor*[]{e}{_{n+1}}),(\tensor*[]{X}{_{0}}, \tensor*[]{e}{_{0}}))$, the pair $(\tensor*[]{e}{_{0}},\tensor*[]{e}{_{n+1}})$ is a morphism of $\BE$-extensions $\delta\to\delta$. Indeed, we have that $(\tensor*[]{e}{_{0}}\tensor*[]{)}{_{\BE}}\delta = \delta = (\tensor*[]{e}{_{n+1}}\tensor*[]{)}{^{\BE}}\delta$ from \cref{def:BF}. 
\end{enumerate}
\end{rem}

\begin{notn}
\label{not:extensions-in-tildeC}
As for objects and morphisms in $\wt{\CC}$, we use tilde notation for $\BF$-extensions, which gives us a way to pass back to $\BE$-extensions.    
\begin{enumerate}[label=(\roman*)]
    
    \item 
    We will denote an $\BF$-extension of the form $(\tensor*[]{e}{_{0}}, \delta, \tensor*[]{e}{_{n+1}}) \in \BF((\tensor*[]{X}{_{n+1}}, \tensor*[]{e}{_{n+1}}),(\tensor*[]{X}{_{0}},\tensor*[]{e}{_{0}}))$ by $\tilde{\delta}$. We call $\delta \in \BE(\tensor*[]{X}{_{n+1}}, \tensor*[]{X}{_{0}})$ the \emph{underlying $\BE$-extension of $\tilde{\delta}$}. 

    \item \label{item:lift-to-BF}
    For $(\tensor*[]{X}{_{n+1}}, \tensor*[]{e}{_{n+1}}), (\tensor*[]{X}{_{0}}, \tensor*[]{e}{_{0}}) \in \wt{\CC}$ and $\delta \in \BE(\tensor*[]{X}{_{n+1}},\tensor*[]{X}{_{0}})$ with $(\tensor*[]{e}{_{0}}\tensor*[]{)}{_{\BE}} \delta = \delta = (\tensor*[]{e}{_{n+1}}\tensor*[]{)}{^{\BE}} \delta$, there is a unique $\BF$-extension $\tilde{\delta} \in \BF((\tensor*[]{X}{_{n+1}},\tensor*[]{e}{_{n+1}}), (\tensor*[]{X}{_{0}},\tensor*[]{e}{_{0}}))$ with underlying $\BE$-extension $\delta$. This $\BF$-extension is $\tilde{\delta} = (\tensor*[]{e}{_{0}}, \delta, \tensor*[]{e}{_{n+1}})$. 
    Again, we use this instance of the tilde notation for this correspondence: 
    we write $\tilde{\rho}\in\BF((\tensor*[]{X}{_{n+1}},\tensor*[]{e}{_{n+1}}), (\tensor*[]{X}{_{0}},\tensor*[]{e}{_{0}}))$ if and only if 
    the underlying $\BE$-extension of $\tilde{\rho}$ is $\rho$, i.e.\ $\tilde{\rho} = (\tensor*[]{e}{_{0}},\rho,\tensor*[]{e}{_{n+1}})$.
\end{enumerate}
\end{notn}

\begin{rem}
\label{rem:equality-of-extensions-in-tildeC}
Analogously to our observations in \cref{rem:equality-of-morphisms-and-commutativity-in-tildeC}, 
we note that by \cref{not:extensions-in-tildeC}\ref{item:lift-to-BF} any two $\BF$-extensions $\tilde{\delta}, \tilde{\rho} \in \BF ((\tensor*[]{X}{_{n+1}},\tensor*[]{e}{_{n+1}}), (\tensor*[]{X}{_{0}},\tensor*[]{e}{_{0}}))$ are equal if and only if their underlying $\BE$-extensions are equal.
Hence, removing the tilde from $\BF$-extensions defines an injective abelian group homomorphism $\BF((Y,e'), (X,e)) \to \BE(Y,X)$ for $(X,e), (Y,e') \in \wt{\CC}$.
\end{rem}

\subsection{Defining the realisation \texorpdfstring{$\ft$}{t}}
\label{subsec:defining-t}

To define an exact realisation $\ft$ of the functor $\BF$ defined in \cref{subsec:def-of-F}, given a morphism of extensions consisting of two idempotents, we will need to lift this morphism by an $(n+2)$-tuple of idempotents. That is, we require a higher version of the idempotent lifting trick (see \cite[Lem.\ 3.5]{Msapato-the-karoubi-envelope-and-weak-idempotent-completion-of-an-extriangulated-category} and \cite[Lem.\ 1.13]{BalmerSchlichting-idempotent-completion-of-triangulated-categories}). 
This turns out to be quite non-trivial and requires an abstraction of the case when $n=1$ in order to understand the mechanics of why this trick is successful. 

We start with two lemmas related to the polynomial ring $\BZ[x]$. 
Recall that $\BZ[x]$ has the universal property that for any (unital, associative) ring $R$ and any element $r \in R$ there is a unique (identity preserving) ring homomorphism $\phi_r \colon \BZ[x] \to R$ with $\phi_r(x) = r$.
For $p=p(x) \in \BZ[x]$, we denote $\phi_r(p)$ by $p(r)$ as is usual.

\begin{lem} \label{lemma:coprime}
   For each $m \in \BN$, the ideals $(x^m) = {(x)}^m$ and $((x-1)^m) = {(x-1)}^m$ of $\BZ[x]$ are coprime.
\end{lem}
\begin{proof}
   The ideals $\sqrt{{(x)}^m} = (x)$ and $\sqrt{{(x-1)}^m} = (x-1)$ are coprime in $\BZ[x]$.
   Hence, $(x^m)$ and $( (x-1)^m )$ are also coprime, by Atiyah--MacDonald \cite[Prop.\ 1.16]{AtiyahMacDonald-Introduction-to-commutative-algebra}.
\end{proof}

\begin{lem} \label{lemma:lift}
   For each $m \in \BN_{\geq 1}$, there is a polynomial $p_m \in (x^m) \unlhd \BZ[x]$, such that for every (unital, associative) ring $R$ we have: 
   \begin{enumerate}[label=\textup{(\roman*)}]
      \item\label{4.9.1} $p_m(e) = e$ for each idempotent $e \in R$; and
      \item\label{4.9.2} the element $p_m(r) \in R$ is an idempotent for each $r \in R$ satisfying $(r^2-r)^m = 0$.
   \end{enumerate}
\end{lem}

\begin{proof}
   Fix an integer $m \geq 1$. 
   By \cref{lemma:coprime}, we can write $1 = x^m p_m' + (x-1)^m q_m'$ for some polynomials $p_m'$ and $q_m'$ in $\BZ[x]$.
   We set $p_m \deff x^m p_m'$.
   
   Let $R$ be a ring.
   For any idempotent $e \in R$, evaluating $x = x^{m+1} p_m' + x(x-1)^m q_m'$ at $e$ and using $e(e-1) = 0$ yields $e = e^{m+1} p_m'(e) = e^m p_m'(e) = p_m(e)$, proving \ref{4.9.1}. 
   
   Now suppose $r \in R$ is an element with $(r^2-r)^m = 0$.  Evaluation of
   \[ 
   p_m 
    = (x^m p'_m) \cdot 1
    = (x^m p'_m) \cdot (x^m p'_m + (x-1)^m q_m') 
    = p_m^2 + (x^2-x)^m p_m' q_m' 
    \]
   at $r$ shows $p_m(r)^2 = p_m(r)$ since $(r^2-r)^m = 0$, which finishes the proof.
\end{proof}

The following is an abstract formulation of \cite[Lem.\ 3.5]{Msapato-the-karoubi-envelope-and-weak-idempotent-completion-of-an-extriangulated-category} and \cite[Lem.\ 1.13]{BalmerSchlichting-idempotent-completion-of-triangulated-categories}.

\begin{lem} \label{lem:IdempotentTrick}
    Let 
    $\tensor*[]{X}{_{\bullet}}\colon
    \begin{tikzcd}[column sep=0.5cm]
    \tensor*[]{X}{_{0}} \arrow{r}{\tensor*[]{d}{_{0}^{X}}}& \tensor*[]{X}{_{1}} \arrow{r}{\tensor*[]{d}{^{X}_{1}}}& \tensor*[]{X}{_{2}} 
    \end{tikzcd}$
    be a complex in an additive category $\CA$ and suppose  
    $\tensor*[]{d}{^{X}_{1}}$ is a weak cokernel of $\tensor*[]{d}{_{0}^{X}}$. 
    Suppose $(\tensor*[]{e}{_{0}}, f_1, \tensor*[]{e}{_{2}}) \colon \tensor*[]{X}{_{\bullet}} \to \tensor*[]{X}{_{\bullet}}$ is a morphism of complexes with $\tensor*[]{e}{_{0}} \in \tensor*[]{\End}{_{\CA}}(\tensor*[]{X}{_{0}})$ and $\tensor*[]{e}{_{2}} \in \tensor*[]{\End}{_{\CA}}(\tensor*[]{X}{_{2}})$ both idempotent.
    Then there exists a morphism $f_1'\colon \tensor*[]{X}{_{1}} \to \tensor*[]{X}{_{1}}$, such that the following hold.
    \begin{enumerate}[label=\textup{(\roman*)}]
        \item\label{item:chain-map} The triplet $(\tensor*[]{e}{_{0}}, f_1', \tensor*[]{e}{_{2}})\colon \tensor*[]{X}{_{\bullet}} \to \tensor*[]{X}{_{\bullet}}$ is a morphism of complexes.
        \item\label{item:idempotent} The element $\tensor*[]{e}{_{1}} \deff f_1 f_1' \in \tensor*[]{\End}{_{\CA}}(\tensor*[]{X}{_{1}})$ is idempotent and satisfies $\tensor*[]{e}{_{1}} = f_1'f_1$. 
        \item\label{item:null-homotopic} The triplet $(\tensor*[]{e}{_{0}}, \tensor*[]{e}{_{1}}, \tensor*[]{e}{_{2}})\colon \tensor*[]{X}{_{\bullet}} \to \tensor*[]{X}{_{\bullet}}$ is an idempotent morphism of complexes.
        \item\label{item:the-magical-fourth-item} If $(\tensor*[]{h}{_{1}}, \tensor*[]{h}{_{2}}) \colon (\tensor*[]{e}{_{0}},f_1,\tensor*[]{e}{_{2}}) \sim \tensor*[]{0}{_{\bullet}}$ is a homotopy of morphisms $\tensor*[]{X}{_{\bullet}} \to \tensor*[]{X}{_{\bullet}}$, then the pair $(\tensor*[]{e}{_{0}} \tensor*[]{h}{_{1}}, f_1' \tensor*[]{h}{_{2}})$ yields a homotopy $(\tensor*[]{e}{_{0}}, \tensor*[]{e}{_{1}}, \tensor*[]{e}{_{2}}) \sim \tensor*[]{0}{_{\bullet}}$.
    \end{enumerate}
\end{lem}

\begin{proof}
Choose a polynomial $\tensor*[]{p}{_{2}} = x^2p'_{2} \in (x^2) \unlhd \BZ[x]$ as obtained in \cref{lemma:lift}. 
Define $q \deff xp'_{2}$ and set $f_1' \deff q(f_1) \colon \tensor*[]{X}{_{1}} \to \tensor*[]{X}{_{1}}$. 
We show this morphism satisfies the claims in the statement. 
For this, we will make use of the following.
Let $p=p(x)\in\BZ[x]$ be any polynomial. 
Since $(\tensor*[]{e}{_{0}}, f_1, \tensor*[]{e}{_{2}})\colon \tensor*[]{X}{_{\bullet}} \to \tensor*[]{X}{_{\bullet}}$ is a morphism of complexes, we have that 
$(p(\tensor*[]{e}{_{0}}), p(f_1), p(\tensor*[]{e}{_{2}}))\colon \tensor*[]{X}{_{\bullet}} \to \tensor*[]{X}{_{\bullet}}$ is also a morphism of complexes, i.e.\ 
the diagram 
\begin{equation}
\label{eqn:idempoten-trick-comm-diag}
    \begin{tikzcd}%
         \tensor*[]{X}{_{0}} 
            \arrow{r}{\tensor*[]{d}{_{0}^{X}}} 
            \arrow{d}[swap]{p(\tensor*[]{e}{_{0}})}
        & \tensor*[]{X}{_{1}} 
            \arrow{r}{\tensor*[]{d}{^{X}_{1}}}
            \arrow{d}{p(f_1)}
        & \tensor*[]{X}{_{2}} 
            \arrow{d}{p(\tensor*[]{e}{_{2}})}\\
         \tensor*[]{X}{_{0}} 
            \arrow{r}[swap]{\tensor*[]{d}{_{0}^{X}}}
        & \tensor*[]{X}{_{1}} 
            \arrow{r}[swap]{\tensor*[]{d}{^{X}_{1}}}
        & \tensor*[]{X}{_{2}}                 
    \end{tikzcd}
\end{equation}
commutes.

\ref{item:chain-map}\;
    Note that $q(\tensor*[]{e}{_{0}}) = \tensor*[]{e}{_{0}}p'_{2}(\tensor*[]{e}{_{0}}) = \tensor*[]{e}{_{0}^{2}}p'_{2}(\tensor*[]{e}{_{0}}) = \tensor*[]{p}{_{2}}(\tensor*[]{e}{_{0}}) = \tensor*[]{e}{_{0}}$, where the last equality follows from \cref{lemma:lift}\ref{4.9.1}. Similarly, $q(\tensor*[]{e}{_{2}}) = \tensor*[]{e}{_{2}}$. Thus, using $p=q$ in the commutative diagram \eqref{eqn:idempoten-trick-comm-diag} shows that
    $(\tensor*[]{e}{_{0}}, f_1', \tensor*[]{e}{_{2}}) = (q(\tensor*[]{e}{_{0}}), q(f_1), q(\tensor*[]{e}{_{2}}))\colon \tensor*[]{X}{_{\bullet}} \to \tensor*[]{X}{_{\bullet}}$ is a morphism of complexes. 
    
\ref{item:idempotent}\;
    Since $f_1' = q(f_1)$ is a polynomial in $f_1$, we immediately have that $\tensor*[]{e}{_{1}} \deff f_1 f_1' = f_1' f_1$. 
    Furthermore, we see that $\tensor*[]{e}{_{1}} = f_1 q(f_1) = \tensor*[]{p}{_{2}}(f_1)$. Thus, to show that $\tensor*[]{e}{_{1}}$ is idempotent, it is enough to show that $(f_1^2 - f_1)^2 = 0$ by \cref{lemma:lift}\ref{4.9.2}. 
    Let $r(x) = x^2 -x$. We see that $r(\tensor*[]{e}{_{0}})$ and $r(\tensor*[]{e}{_{2}})$ vanish as $\tensor*[]{e}{_{0}}$ and $\tensor*[]{e}{_{2}}$ are idempotents.
   Therefore, by choosing $p = r$ in \eqref{eqn:idempoten-trick-comm-diag} we have $r(f_1)\tensor*[]{d}{_{0}^{X}}  = 0$ and so
   there is $h \colon \tensor*[]{X}{_{2}} \to \tensor*[]{X}{_{1}}$ with $h \tensor*[]{d}{^{X}_{1}}  = r(f_1)$, 
   because $\tensor*[]{d}{^{X}_{1}}$ is a weak cokernel of $\tensor*[]{d}{_{0}^{X}}$. 
   This implies  
   $(f_1^2 - f_1)^2 
        = r(f_1)^2 
        = h \tensor*[]{d}{^{X}_{1}} r(f_1) 
        = h r(\tensor*[]{e}{_{2}})\tensor*[]{d}{^{X}_{1}} 
        = 0
    $ as $r(\tensor*[]{e}{_{2}}) = 0$, and hence $\tensor*[]{e}{_{1}}$ is idempotent.

\ref{item:null-homotopic}\;
    Note that $(\tensor*[]{e}{_{0}}, \tensor*[]{e}{_{1}}, \tensor*[]{e}{_{2}})^2 = (\tensor*[]{e}{_{0}}, \tensor*[]{e}{_{1}}, \tensor*[]{e}{_{2}}) = (\tensor*[]{p}{_{2}}(\tensor*[]{e}{_{0}}),\tensor*[]{p}{_{2}}(f_1),\tensor*[]{p}{_{2}}(\tensor*[]{e}{_{2}}))\colon \tensor*[]{X}{_{\bullet}} \to \tensor*[]{X}{_{\bullet}}$ is a morphism of complexes using $p=\tensor*[]{p}{_{2}}$ in \eqref{eqn:idempoten-trick-comm-diag}. 
    
\ref{item:the-magical-fourth-item}\;
    Suppose $(\tensor*[]{h}{_{1}}, \tensor*[]{h}{_{2}})\colon (\tensor*[]{e}{_{0}},f_1,\tensor*[]{e}{_{2}}) \sim \tensor*[]{0}{_{\bullet}}$ is a homotopy. 
    Then we see that 
    \begin{align*} 
        (\tensor*[]{e}{_{0}},\tensor*[]{e}{_{1}},\tensor*[]{e}{_{2}}) &= (\tensor*[]{e}{_{0}},f'_1,\tensor*[]{e}{_{2}}) (\tensor*[]{e}{_{0}},\tensor*[]{f}{_{1}},\tensor*[]{e}{_{2}}) && \text{by \ref{item:idempotent}}\\
        &= (\tensor*[]{e}{_{0}},f'_1,\tensor*[]{e}{_{2}}) (\tensor*[]{h}{_{1}} \tensor*[]{d}{^{X}_{0}}, \tensor*[]{h}{_{2}} \tensor*[]{d}{^{X}_{1}} + \tensor*[]{d}{^{X}_{0}} \tensor*[]{h}{_{1}}, \tensor*[]{d}{^{X}_{1}} \tensor*[]{h}{_{2}} ) && \text{as $(\tensor*[]{h}{_{1}}, \tensor*[]{h}{_{2}})\colon (\tensor*[]{e}{_{0}},f_1,\tensor*[]{e}{_{2}}) \sim \tensor*[]{0}{_{\bullet}}$}\\ 
        &= ( \tensor*[]{e}{_{0}} \tensor*[]{h}{_{1}} \tensor*[]{d}{^{X}_{0}}, f'_1 \tensor*[]{h}{_{2}} \tensor*[]{d}{^{X}_{1}} + f'_1 \tensor*[]{d}{^{X}_{0}} \tensor*[]{h}{_{1}}, \tensor*[]{e}{_{2}} \tensor*[]{d}{^{X}_{1}} \tensor*[]{h}{_{2}} ) && \\
        &= ( (\tensor*[]{e}{_{0}} \tensor*[]{h}{_{1}}) \tensor*[]{d}{^{X}_{0}}, (f'_1 \tensor*[]{h}{_{2}}) \tensor*[]{d}{^{X}_{1}} + \tensor*[]{d}{^{X}_{0}} (\tensor*[]{e}{_{0}} \tensor*[]{h}{_{1}}), \tensor*[]{d}{^{X}_{1}} (f'_1 \tensor*[]{h}{_{2}}) ) && \text{by \ref{item:chain-map}}.
    \end{align*}
    Hence, $(\tensor*[]{e}{_{0}} \tensor*[]{h}{_{1}}, f'_1 \tensor*[]{h}{_{2}}) \colon (\tensor*[]{e}{_{0}},\tensor*[]{e}{_{1}},\tensor*[]{e}{_{2}}) \sim \tensor*[]{0}{_{\bullet}}$ is a null homotopy as desired.
\end{proof}

\begin{rem} \label{rem:TrickExplained}
   Let $p'_2 = -2x + 3$ and $q'_2 = 2x + 1$.
   Then indeed $1 = x^2p'_2 + (x-1)^2q'_2$.
   Hence, $\tensor*[]{p}{_{2}} = x^2 p'_2 = 3x^2-2x^3$ is a possible choice for $m=2$ in \cref{lemma:lift}.
   Letting $h = x^2 - x$ and $i = x$, we see that $\tensor*[]{p}{_{2}} = i + h - 2ih$. 
   Then the idempotent $\tensor*[]{e}{_{1}}$ obtained in \cref{lem:IdempotentTrick} is the idempotent obtained through the idempotent lifting trick in \cite[Lem.\ 3.5]{Msapato-the-karoubi-envelope-and-weak-idempotent-completion-of-an-extriangulated-category}.
\end{rem}

\begin{lem} 
\label{lem:leftlift}
   Suppose $\langle \tensor*[]{X}{_{\bullet}}, \delta \rangle$ is an $\fs$-distinguished $n$-exangle and $\tensor*[]{e}{_{0}} \in \tensor*[]{\End}{_{\CC}}(\tensor*[]{X}{_{0}})$ is an idempotent with $(\tensor*[]{e}{_{0}}\tensor*[]{)}{_{\BE}} \delta = 0$.
	Then $\tensor*[]{e}{_{0}}$ can be extended to a null homotopic, idempotent morphism $\tensor*[]{e}{_{\bullet}} \colon \langle \tensor*[]{X}{_{\bullet}}, \delta \rangle \to \langle \tensor*[]{X}{_{\bullet}}, \delta \rangle$ with $\tensor*[]{e}{_{i}} = 0$ for $2 \leq i \leq n+1$.
	Further, the null homotopy of $e_{\bullet}$ can be chosen to be of the shape $\tensor*[]{h}{_{\bullet}} = (\tensor*[]{h}{_{1}}, 0, \dots, 0) \colon \tensor*[]{e}{_{\bullet}} \sim \tensor*[]{0}{_{\bullet}}$. 
\end{lem} 

\begin{proof}
   We have $(\tensor*[]{e}{_{0}}\tensor*[]{)}{_{\BE}} \delta = 0 = \tensor*[]{0}{^{\BE}} \delta$ so $(\tensor*[]{e}{_{0}}, 0) \colon \delta \to \delta$ is a morphism of $\BE$-extensions.
   The solid morphisms of the diagram
   \[\begin{tikzcd}
      \tensor*[]{X}{_{0}} \arrow{r}{\tensor*[]{d}{_{0}^{X}}} \arrow{d}{\tensor*[]{e}{_{0}}} & \tensor*[]{X}{_{1}} \arrow{r}{\tensor*[]{d}{^{X}_{1}}} \arrow[dotted]{d}{f_1} & \tensor*[]{X}{_{2}} \arrow{r}{d_2^X} \arrow[d, "0"] & \cdots \arrow{r}{\tensor*[]{d}{_{n-2}^{X}}} & \tensor*[]{X}{_{n-1}} \arrow{r}{\tensor*[]{d}{^{X}_{n-1}}} \arrow[d, "0"] & \tensor*[]{X}{_{n}} \arrow{r}{\tensor*[]{d}{^{X}_{n}}} \arrow[d, "0"] & \tensor*[]{X}{_{n+1}} \arrow[r, "\delta", dashed] \arrow[d, "0"] & {} \\
      \tensor*[]{X}{_{0}} \arrow{r}{\tensor*[]{d}{_{0}^{X}}}                  & \tensor*[]{X}{_{1}} \arrow{r}{\tensor*[]{d}{^{X}_{1}}}                          & \tensor*[]{X}{_{2}} \arrow{r}{d_2^X}                & \cdots \arrow{r}{\tensor*[]{d}{_{n-2}^{X}}} & \tensor*[]{X}{_{n-1}} \arrow{r}{\tensor*[]{d}{^{X}_{n-1}}}                & \tensor*[]{X}{_{n}} \arrow{r}{\tensor*[]{d}{^{X}_{n}}}                & \tensor*[]{X}{_{n+1}} \arrow[r, "\delta", dashed]                & {}
   \end{tikzcd}\]
   clearly commute, so we need to find a morphism $f_1\colon \tensor*[]{X}{_{1}} \to \tensor*[]{X}{_{1}}$ making the two leftmost squares commute.
   Since $\langle \tensor*[]{X}{_{\bullet}}, \delta \rangle$ is an $\fs$-distinguished $n$-exangle, there is an exact sequence 
   \[
   \begin{tikzcd}[column sep=1.9cm]
	\CC(\tensor*[]{X}{_{1}}, \tensor*[]{X}{_{0}})
   		\arrow{r}{\CC(\tensor*[]{d}{_{0}^{X}}, \tensor*[]{X}{_{0}})}
	& \CC(\tensor*[]{X}{_{0}}, \tensor*[]{X}{_{0}}) 
		\arrow{r}{\tensor*[_{\BE}]{\delta}{}}
	& \BE(\tensor*[]{X}{_{n+1}}, \tensor*[]{X}{_{0}}).
   \end{tikzcd}
   \]
   The morphism $\tensor*[]{e}{_{0}}$ is in the kernel of $\tensor*[_{\BE}]{\delta}{}$ as $\tensor*[_{\BE}]{\delta}{}(\tensor*[]{e}{_{0}}) = (\tensor*[]{e}{_{0}}\tensor*[]{)}{_{\BE}} \delta = 0$.
   Therefore, there exists $\tensor*[]{k}{_{1}} \colon \tensor*[]{X}{_{1}} \to \tensor*[]{X}{_{0}}$ with $\tensor*[]{e}{_{0}} = \tensor*[]{k}{_{1}} \tensor*[]{d}{_{0}^{X}}$.
   If we put $f_1 \deff \tensor*[]{d}{_{0}^{X}} \tensor*[]{k}{_{1}}$, then $(\tensor*[]{e}{_{0}}, f_1, 0, \ldots, 0) \colon \langle \tensor*[]{X}{_{\bullet}}, \delta \rangle \to \langle \tensor*[]{X}{_{\bullet}}, \delta \rangle$ is morphism of $\fs$-distinguished $n$-exangles and $(\tensor*[]{k}{_{1}}, 0, \dots, 0) \colon \tensor*[]{e}{_{\bullet}} \sim \tensor*[]{0}{_{\bullet}}$ is a homotopy.
   By \cref{lem:IdempotentTrick}, using that $\tensor*[]{e}{_{0}}$ and $0$ are idempotents, there is an idempotent $\tensor*[]{e}{_{1}} \in \tensor*[]{\End}{_{\CC}}(\tensor*[]{X}{_{1}})$, such that $(\tensor*[]{e}{_{0}}, \tensor*[]{e}{_{1}}, 0, \dots, 0) \colon \tensor*[]{X}{_{\bullet}} \to \tensor*[]{X}{_{\bullet}}$ is an idempotent morphism of complexes and that $\tensor*[]{h}{_{\bullet}} \deff (\tensor*[]{e}{_{0}} \tensor*[]{k}{_{1}}, 0, \dots, 0) \colon \tensor*[]{e}{_{\bullet}} \sim \tensor*[]{0}{_{\bullet}}$ is a homotopy. Finally, $\tensor*[]{e}{_{\bullet}}$ is a morphism of $\fs$-distinguished $n$-exangles since $(\tensor*[]{e}{_{0}}\tensor*[]{)}{_{\BE}} \delta = 0 = \tensor*[]{0}{^{\BE}} \delta$.
\end{proof}

\begin{cor} 
\label{cor:newlift}
   Suppose $\tilde{\delta} \in \BF((\tensor*[]{X}{_{n+1}}, \tensor*[]{e}{_{n+1}}),(\tensor*[]{X}{_{0}}, \tensor*[]{e}{_{0}}))$ and that $\langle \tensor*[]{X}{_{\bullet}}, \delta \rangle$ is an $\fs$-distinguished $n$-exangle. 
	The morphism $(\tensor*[]{e}{_{0}},\tensor*[]{e}{_{n+1}})\colon \delta\to \delta$ of $\BE$-extensions has a lift $\tensor*[]{e}{_{\bullet}} \colon \langle \tensor*[]{X}{_{\bullet}}, \delta \rangle \to \langle \tensor*[]{X}{_{\bullet}}, \delta \rangle$ that is idempotent and satisfies $e_i = \id{\tensor*[]{X}{_i}}$ for all $2 \leq i \leq n -1$, such that there is a homotopy $\tensor*[]{h}{_{\bullet}} = (\tensor*[]{h}{_{1}}, 0, \dots, 0, \tensor*[]{h}{_{n+1}}) \colon \id{\tensor*[]{X}{_{\bullet}}} - \tensor*[]{e}{_{\bullet}} \sim \tensor*[]{0}{_{\bullet}}$.
\end{cor} 

\begin{proof}
   Define $e'_{0} \deff \id{\tensor*[]{X}{_{0}}} - \tensor*[]{e}{_{0}}$ and $e''_{n+1} \deff \id{\tensor*[]{X}{_{n+1}}} - \tensor*[]{e}{_{n+1}}$.
   Since $\tilde{\delta} \in \BF((\tensor*[]{X}{_{n+1}}, \tensor*[]{e}{_{n+1}}),(\tensor*[]{X}{_{0}}, \tensor*[]{e}{_{0}}))$, 
   we have $(\tensor*[]{e}{_{0}}\tensor*[]{)}{_{\BE}} \delta = \delta = (\tensor*[]{e}{_{n+1}}\tensor*[]{)}{^{\BE}} \delta$ and 
   so $(e'_{0}\tensor*[]{)}{_{\BE}} \delta = 0 = (e''_{n+1}\tensor*[]{)}{^{\BE}} \delta$. 
   Therefore, by \cref{lem:leftlift} we can extend $e'_{0}$ to an idempotent morphism $e'_{\bullet} \colon \langle \tensor*[]{X}{_{\bullet}}, \delta \rangle\to \langle \tensor*[]{X}{_{\bullet}}, \delta \rangle$ of $\fs$-distinguished $n$-exangles with $e'_{i} = 0$ for $i\in\{ 2,\ldots,n+1 \}$, having a homotopy $(\tensor*[]{k}{_{1}}, 0, \dots, 0) \colon e'_{\bullet} \sim \tensor*[]{0}{_{\bullet}}$.
   Similarly, by the dual of \cref{lem:leftlift}, we can extend $e''_{n+1}$ to an idempotent morphism $e''_{\bullet} \colon \langle \tensor*[]{X}{_{\bullet}}, \delta \rangle \to \langle \tensor*[]{X}{_{\bullet}}, \delta \rangle$ 
   with $e''_{i} = 0$ for $i\in\{ 0,\ldots,n-1 \}$, such that there is a homotopy $(0, \dots, 0, \tensor*[]{k}{_{n+1}}) \colon e''_{\bullet} \sim \tensor*[]{0}{_{\bullet}}$.
   Consider the morphism 
   $\tensor*[]{f}{_{\bullet}} \deff \id{\tensor*[]{X}{_{\bullet}}} - e'_{\bullet} -e''_{\bullet} \colon \langle \tensor*[]{X}{_{\bullet}}, \delta \rangle \to \langle \tensor*[]{X}{_{\bullet}}, \delta \rangle$.
   We have $\id{\tensor*[]{X}{_{\bullet}}} - \tensor*[]{f}{_{\bullet}} = e'_{\bullet} + e''_{\bullet}$ and hence $(\tensor*[]{k}{_{1}}, 0, \dots, 0, \tensor*[]{k}{_{n+1}}) \colon \id{\tensor*[]{X}{_{\bullet}}} - \tensor*[]{f}{_{\bullet}} \sim \tensor*[]{0}{_{\bullet}}$ is a homotopy.
   
If $n = 1$, then $e'_{\bullet} + e''_{\bullet} = (\id{\tensor*[]{X}{_{0}}} - \tensor*[]{e}{_{0}}, e'_{1} + e''_{1}, \id{\tensor*[]{X}{_{2}}} - \tensor*[]{e}{_{2}})$ and $(\tensor*[]{k}{_{1}}, \tensor*[]{k}{_{2}}) \colon e'_{\bullet} + e''_{\bullet} \sim \tensor*[]{0}{_{\bullet}}$ is a homotopy. \cref{lem:IdempotentTrick} yields an idempotent morphism $e'''_{\bullet} = (\id{\tensor*[]{X}{_{0}}} - \tensor*[]{e}{_{0}}, e'''_{1}, \id{\tensor*[]{X}{_{2}}} - \tensor*[]{e}{_{2}}) \colon \lan \tensor*[]{X}{_{\bullet}}, \delta \ran \to \lan \tensor*[]{X}{_{\bullet}}, \delta \ran$ and a homotopy $(\tensor*[]{h}{_{1}}, \tensor*[]{h}{_{2}}) \colon e'''_{\bullet} \sim \tensor*[]{0}{_{\bullet}}$. 
Then $\tensor*[]{e}{_{\bullet}} \deff \id{\tensor*[]{X}{_{\bullet}}} - e'''_{\bullet}$ and $\tensor*[]{h}{_{\bullet}} \deff (\tensor*[]{h}{_{1}}, \tensor*[]{h}{_{2}})$ are the desired idempotent morphism and homotopy, respectively.
   
If $n \geq 2$, then the compositions $e'_{\bullet}e''_{\bullet} $ and $e''_{\bullet} e'_{\bullet}$ are zero. This implies that $\tensor*[]{f}{_{\bullet}} = \id{\tensor*[]{X}{_{\bullet}}} - e'_{\bullet} -e''_{\bullet}$ is idempotent. 
Hence, $\tensor*[]{e}{_{\bullet}}\deff \tensor*[]{f}{_{\bullet}}$ and $(\tensor*[]{h}{_{1}}, 0 , \dots, 0, \tensor*[]{h}{_{n+1}}) \deff (\tensor*[]{k}{_{1}}, 0, \dots, 0, \tensor*[]{k}{_{n+1}})$ are the desired idempotent morphism and homotopy, respectively.
\end{proof}

The following simple lemma will be used several times.

\begin{lem}
\label{lem:induced-morphism} 
    Suppose that $(X,e), (Y,e')$ are objects in $\wt{\CC}$ and $r\colon X\to Y$ is a morphism in $\CC$. 
    Setting $s \deff e're$ yields a morphism $\tilde{s} = (e',s,e) \colon(X,e)\to(Y,e')$ in $\wt{\CC}$.
\end{lem}

The previous result allows us to view a complex in $\CC$ that is equipped with an idempotent endomorphism as a complex in the idempotent completion $\wt{\CC}$, as follows.

\begin{defn}
\label{def:induced-complex-in-Ctilde}
Suppose $\tensor*[]{X}{_{\bullet}}$ is a complex in $\CC$ and $\tensor*[]{e}{_{\bullet}} \colon \tensor*[]{X}{_{\bullet}} \to \tensor*[]{X}{_{\bullet}}$ is an idempotent morphism of complexes. We denote by $(\tensor*[]{X}{_{\bullet}}, \tensor*[]{e}{_{\bullet}})$ the complex in $\wt{\CC}$ with object $(\tensor*[]{X}{_{i}}, \tensor*[]{e}{_{i}})$ in degree $i$ and 
differential 
$\tensor*[]{\tilde{d}}{^{(X,e)}_{i}} \deff (\tensor*[]{e}{_{i+1}},\tensor*[]{e}{_{i+1}} \tensor*[]{d}{^{X}_{i}} \tensor*[]{e}{_{i}} ,\tensor*[]{e}{_{i}}) \colon (\tensor*[]{X}{_{i}}, \tensor*[]{e}{_{i}}) \to (\tensor*[]{X}{_{i+1}}, \tensor*[]{e}{_{i+1}})$.
\end{defn}

In the notation of \cref{def:induced-complex-in-Ctilde}, the underlying morphism of the differential $\tensor*[]{\tilde{d}}{^{(X,e)}_{i}}$ satisfies 
\begin{equation}\label{eqn:X-bullet-diffs}
\tensor*[]{d}{^{(X,e)}_{i}}  
    = \tensor*[]{e}{_{i+1}} \tensor*[]{d}{^{X}_{i}} \tensor*[]{e}{_{i}} 
    = \tensor*[]{d}{^{X}_{i}} \tensor*[]{e}{_{i}} 
    = \tensor*[]{e}{_{i+1}} \tensor*[]{d}{^{X}_{i}}, 
\end{equation}%
since $\tensor*[]{e}{_{\bullet}}$ is a morphism of complexes and consists of idempotents. 
Furthermore, whenever we write $(\tensor*[]{X}{_{\bullet}}, \tensor*[]{e}{_{\bullet}})$ to denote a complex in $\wt{\CC}$, we always mean that $\tensor*[]{e}{_{\bullet}} \colon \tensor*[]{X}{_{\bullet}} \to \tensor*[]{X}{_{\bullet}}$ is an idempotent morphism in $\com{\CC}$ and that $(\tensor*[]{X}{_{\bullet}}, \tensor*[]{e}{_{\bullet}})$ is the induced object in $\com{\wt{\CC}}$ as described in \cref{def:induced-complex-in-Ctilde}. 

We make a further remark on the notation $(\tensor*[]{X}{_{\bullet}}, \tensor*[]{e}{_{\bullet}})$. Because of the need to tweak the differentials in $\tensor*[]{X}{_{\bullet}}$ according to 
\eqref{eqn:X-bullet-diffs}, one cannot recover the original complex $\tensor*[]{X}{_{\bullet}}\in\com{\CC}$ with differentials $\tensor*[]{d}{^{X}_{i}}$ from the pair $(\tensor*[]{X}{_{\bullet}}, \tensor*[]{e}{_{\bullet}})\in\com{\wt{\CC}}$ defined in \cref{def:induced-complex-in-Ctilde}. 
This is in contrast to the description of an object in $\wt{\CC}$ as a pair $(X,e)$ where one can recover $X\in\CC$ uniquely. 
Thus, $(\tensor*[]{X}{_{\bullet}}, \tensor*[]{e}{_{\bullet}})$ is an abuse of notation but should hopefully cause no confusion.

\cref{lem:induced-morphism} allows us to induce morphisms of complexes in $\wt{\CC}$ given a morphism between complexes in $\CC$ if the complexes involved come with idempotent endomorphisms. The proof is also straightforward.

\begin{lem}
\label{lem:induced-chain-map} 
    Suppose that $(\tensor*[]{X}{_{\bullet}},\tensor*[]{e}{_{\bullet}}), (\tensor*[]{Y}{_{\bullet}},e'_{\bullet})$ are objects in $\com{\wt{\CC}}$ and that $\tensor*[]{r}{_{\bullet}}\colon \tensor*[]{X}{_{\bullet}}\to \tensor*[]{Y}{_{\bullet}}$ is a morphism in $\com{\CC}$. 
    Then defining $\tensor*[]{s}{_{i}} \deff e'_i \tensor*[]{r}{_{i}} \tensor*[]{e}{_{i}}$ for each $i\in\BZ$ gives rise to a morphism 
    $\tensor*[]{\tilde{s}}{_{\bullet}}\colon(\tensor*[]{X}{_{\bullet}},\tensor*[]{e}{_{\bullet}})\to(\tensor*[]{Y}{_{\bullet}},e'_{\bullet})$ in $\com{\wt{\CC}}$ 
    with $\tensor*[]{\tilde{s}}{_{i}} = (e'_i, \tensor*[]{s}{_{i}}, \tensor*[]{e}{_{i}})$. 
\end{lem}

\begin{notn}
In the setup of \cref{lem:induced-chain-map}, the composite $e'_{\bullet} \tensor*[]{r}{_{\bullet}} \tensor*[]{e}{_{\bullet}}$ is a morphism of complexes $\tensor*[]{X}{_{\bullet}} \to \tensor*[]{Y}{_{\bullet}}$. In this case, we call $\tensor*[]{s}{_{\bullet}} \deff e'_{\bullet} \tensor*[]{r}{_{\bullet}} \tensor*[]{e}{_{\bullet}}$ the \emph{underlying morphism} of $\tensor*[]{\tilde{s}}{_{\bullet}}$.
\end{notn}

We need two more lemmas before we can define a realisation of the functor $\BF$.

\begin{lem}
\label{lem:n-exangle-from-idempotent-morphism-and-n-exangle}
   Assume $\tilde{\delta} \in \BF((\tensor*[]{X}{_{n+1}}, \tensor*[]{e}{_{n+1}}),(\tensor*[]{X}{_{0}}, \tensor*[]{e}{_{0}}))$. Further, suppose that $\langle \tensor*[]{X}{_{\bullet}}, \delta \rangle$ is an $\fs$-distinguished $n$-exangle and $\tensor*[]{e}{_{\bullet}} \colon \langle \tensor*[]{X}{_{\bullet}}, \delta \rangle \to \langle \tensor*[]{X}{_{\bullet}}, \delta \rangle$ is an idempotent lift of $(\tensor*[]{e}{_{0}}, \tensor*[]{e}{_{n+1}}) \colon \delta \to \delta$. 
   Then $\langle (\tensor*[]{X}{_{\bullet}}, \tensor*[]{e}{_{\bullet}}), \tilde{\delta} \rangle$ is an $n$-exangle for $(\wt{\CC}, \BF)$.
\end{lem}
\begin{proof}
   Let $(Y,e')\in\wt{\CC}$ be arbitrary. 
   We will show that the induced sequence
   \[
   \begin{tikzcd}[column sep=0.99cm, scale cd=0.78]
   \wt{\CC}((Y,e'),(\tensor*[]{X}{_{0}},\tensor*[]{e}{_{0}})) 
   			\arrow{r}{(\tilde{d}^{(X,e)}_0 )_\ast} 
	&[0.2cm] \wt{\CC}((Y,e'),(\tensor*[]{X}{_{1}},\tensor*[]{e}{_{1}})) 
			\arrow{r}{(\tilde{d}^{(X,e)}_1 )_\ast} 
	&[0.2cm] \cdots \arrow{r}{(\tilde{d}^{(X,e)}_n )_\ast} 
	&[0.2cm] \wt{\CC}((Y,e'),(\tensor*[]{X}{_{n+1}},\tensor*[]{e}{_{n+1}}))
			\arrow{r}{%
			            \tensor*[^{\BF}]{\tilde{\delta}}{} 
			         } 
	&[-0.5cm] \BF((Y,e'),(\tensor*[]{X}{_{0}},\tensor*[]{e}{_{0}})),
   \end{tikzcd}
   \]
   where $(\tilde{d}_i^{(X,e)})_{\ast} = \wt{\CC}((Y,e'), \tilde{d}_i^{(X,e)})$, is exact. 
   The exactness of the dual sequence can be verified similarly.
   Checking the above sequence is a complex is straightforward using that $(\tensor*[]{d}{^{X}_{n}}\tensor*[]{)}{^{\BE}}\delta = 0$ and that 
   $\tensor*[]{d}{^{(X,e)}_{i}} = \tensor*[]{d}{^{X}_{i}} \tensor*[]{e}{_{i}} = \tensor*[]{e}{_{i+1}}\tensor*[]{d}{_{i}^{X}}$.
   
  To check exactness at $\wt{\CC}((Y,e'),(\tensor*[]{X}{_{i}},e_{i})) $ for some $1\leq i \leq n$, suppose we have a morphism $\tilde{r}\colon (Y,e') \to (\tensor*[]{X}{_{i}},\tensor*[]{e}{_{i}})$ with $\tilde{d}_i^{(X,e)} \tilde{r} = 0$, that is, $d_i^X\tensor*[]{e}{_{i}} r=0$. 
  As $\tensor*[]{e}{_{i}}r=r$, we see that $d_i^Xr=0$, whence there exists $s\colon Y \to \tensor*[]{X}{_{i-1}}$ such that $d_{i-1}^Xs = r$ because $\langle \tensor*[]{X}{_{\bullet}}, \delta \rangle$ is an $\fs$-distinguished $n$-exangle. 
  By \cref{lem:induced-morphism}, there is a morphism $\tilde{t} \colon (Y,e')\to (\tensor*[]{X}{_{i-1}},e_{i-1})$ with $t = e_{i-1} s e'$.
  Then we observe that 
  $
  d^{(X,e)}_{i-1}t 
  	= d_{i-1}^Xe_{i-1}e_{i-1}se' 
	= \tensor*[]{e}{_{i}}d_{i-1}^Xse' 
	= \tensor*[]{e}{_{i}} r e' 
	= r
   $, whence $\tilde{d}_{i-1}^{(X,e)} \tilde{t} = \tilde{r}$.
   
   Lastly, suppose $\tilde{u} \in\wt{\CC}((Y,e'),(\tensor*[]{X}{_{n+1}},\tensor*[]{e}{_{n+1}}))$ is a morphism with
   $\tensor*[^{\BF}]{\tilde{\delta}}{}(\tilde{u}) = (\tilde{u}\tensor*[]{)}{^{\BF}}\delta = 0$.
   Then we have
   $
        \tensor*[^{\BE}]{\delta}{}(u)
        = \tensor*[]{u}{^{\BE}}\delta 
        = 0
    $. 
   Hence, there is a morphism $v \colon Y \to \tensor*[]{X}{_{n}}$ such that $\tensor*[]{d}{^{X}_{n}} v = u$ as 
   $\langle \tensor*[]{X}{_{\bullet}}, \delta \rangle$ is an $\fs$-distinguished $n$-exangle.
   Then the morphism $\tilde{w} \colon (Y,e') \to (\tensor*[]{X}{_{n}}, \tensor*[]{e}{_{n}})$ with $w = \tensor*[]{e}{_{n}} v e'$ satisfies 
   $\tilde{d}^{(X,e)}_n \tilde{w} = \tilde{u}$, as required.
\end{proof}

\begin{lem} \label{lem:independent}
Suppose $\tilde{\delta} \in \BF((\tensor*[]{X}{_{n+1}}, \tensor*[]{e}{_{n+1}}),(\tensor*[]{X}{_{0}}, \tensor*[]{e}{_{0}}))$ and that $[\tensor*[]{Y}{_{\bullet}}] = \fs(\delta) = [\tensor*[]{X}{_{\bullet}}]$ in $(\CC, \BE, \fs)$.
If $\tensor*[]{e}{_{\bullet}} \colon \langle \tensor*[]{X}{_{\bullet}}, \delta \rangle \to \langle \tensor*[]{X}{_{\bullet}}, \delta \rangle$ and $e'_{\bullet} \colon \langle \tensor*[]{Y}{_{\bullet}}, \delta \rangle \to \langle \tensor*[]{Y}{_{\bullet}}, \delta \rangle$ are idempotent lifts of $(\tensor*[]{e}{_{0}}, \tensor*[]{e}{_{n+1}}) \colon \delta \to \delta$, 
then $(\tensor*[]{X}{_{\bullet}}, \tensor*[]{e}{_{\bullet}})$ and $(\tensor*[]{Y}{_{\bullet}}, e'_{\bullet})$ are isomorphic in 
$\kom{\wt{\CC}}_{((\tensor*[]{X}{_{0}},\tensor*[]{e}{_{0}}),(\tensor*[]{X}{_{n+1}},\tensor*[]{e}{_{n+1}}))}^{\raisebox{0.5pt}{\scalebox{0.6}{$n$}}}$, i.e.\ 
$[(\tensor*[]{X}{_{\bullet}}, \tensor*[]{e}{_{\bullet}})] = [(\tensor*[]{Y}{_{\bullet}}, e'_{\bullet})]$.
\end{lem}

\begin{proof}   
   We will use \cite[Prop.\ 2.21]{HerschendLiuNakaoka-n-exangulated-categories-I-definitions-and-fundamental-properties}. To this end, 
   note that $\langle (\tensor*[]{X}{_{\bullet}},\tensor*[]{e}{_{\bullet}}), \tilde{\delta} \rangle$ and $\langle (\tensor*[]{Y}{_{\bullet}},e'_{\bullet}), \tilde{\delta} \rangle$  are both $n$-exangles in $(\wt{\CC}, \BF)$ by \cref{lem:n-exangle-from-idempotent-morphism-and-n-exangle}. Hence, we only have to show that 
   \[ 
   \com{\wt{\CC}}_{((\tensor*[]{X}{_{0}},\tensor*[]{e}{_{0}}),(\tensor*[]{X}{_{n+1}},\tensor*[]{e}{_{n+1}}))}^{\raisebox{0.5pt}{\scalebox{0.6}{$n$}}}((\tensor*[]{X}{_{\bullet}},\tensor*[]{e}{_{\bullet}}),(\tensor*[]{Y}{_{\bullet}},e'_{\bullet}))
   \hspace{0.5cm}\text{and}\hspace{0.5cm}
   \com{\wt{\CC}}_{((\tensor*[]{X}{_{0}},\tensor*[]{e}{_{0}}),(\tensor*[]{X}{_{n+1}},\tensor*[]{e}{_{n+1}}))}^{\raisebox{0.5pt}{\scalebox{0.6}{$n$}}}((\tensor*[]{Y}{_{\bullet}},e'_{\bullet}),(\tensor*[]{X}{_{\bullet}},\tensor*[]{e}{_{\bullet}})) 
   \]
   are both non-empty.
   Since we have $[\tensor*[]{Y}{_{\bullet}}] = \fs(\delta) = [\tensor*[]{X}{_{\bullet}}]$, there are morphisms
   $\tensor*[]{f}{_{\bullet}}\colon \tensor*[]{X}{_{\bullet}} \to \tensor*[]{Y}{_{\bullet}}$
   and 
   $\tensor*[]{g}{_{\bullet}}\colon \tensor*[]{Y}{_{\bullet}} \to \tensor*[]{X}{_{\bullet}}$
   in $\com{\CC}_{(\tensor*[]{X}{_{0}},\tensor*[]{X}{_{n+1}})}^{\raisebox{0.5pt}{\scalebox{0.6}{$n$}}}$ (with $\tensor*[]{g}{_{\bullet}} \tensor*[]{f}{_{\bullet}} \sim \id{\tensor*[]{X}{_{\bullet}}}$ and $\tensor*[]{f}{_{\bullet}} \tensor*[]{g}{_{\bullet}} \sim \id{\tensor*[]{Y}{_{\bullet}}}$). 
   We then obtain morphisms 
   $\tensor*[]{\tilde{h}}{_{\bullet}} \colon (\tensor*[]{X}{_{\bullet}}, \tensor*[]{e}{_{\bullet}}) \to (\tensor*[]{Y}{_{\bullet}}, e'_{\bullet})$ 
   and 
   $\tilde{k}_{\bullet} \colon (\tensor*[]{Y}{_{\bullet}}, e'_{\bullet}) \to (\tensor*[]{X}{_{\bullet}}, \tensor*[]{e}{_{\bullet}})$ 
   in $\com{\wt{\CC}}^{\raisebox{0.5pt}{\scalebox{0.6}{$n$}}}$ with $\tensor*[]{h}{_{\bullet}} = e'_{\bullet} \tensor*[]{f}{_{\bullet}} \tensor*[]{e}{_{\bullet}}$ and $k_{\bullet} = \tensor*[]{e}{_{\bullet}} \tensor*[]{g}{_{\bullet}} e'_{\bullet}$ by \cref{lem:induced-chain-map}. 
   Note that $\tensor*[]{e}{_{0}} = e'_0, \tensor*[]{e}{_{n+1}} = e'_{n+1}, \tensor*[]{f}{_{0}} = \tensor*[]{g}{_{0}} =  \id{\tensor*[]{X}{_{0}}}$ and $\tensor*[]{f}{_{n+1}} = g_{n+1} = \id{\tensor*[]{X}{_{n+1}}}$. So, since $\id{(\tensor*[]{X}{_{i}},\tensor*[]{e}{_{i}})} = \tensor*[]{e}{_{i}}$, we have that 
   $\tensor*[]{\tilde{h}}{_{\bullet}}$ and $\tilde{k}_{\bullet}$ are morphisms in $\com{\wt{\CC}}_{((\tensor*[]{X}{_{0}},\tensor*[]{e}{_{0}}),(\tensor*[]{X}{_{n+1}},\tensor*[]{e}{_{n+1}}))}^{\raisebox{0.5pt}{\scalebox{0.6}{$n$}}}$ and we are done.
\end{proof}

Hence, the following is well-defined.

\begin{defn}
\label{def:t}
    For $\tilde{\delta} \in \BF((\tensor*[]{X}{_{n+1}}, \tensor*[]{e}{_{n+1}}),(\tensor*[]{X}{_{0}}, \tensor*[]{e}{_{0}}))$, 
    pick $\tensor*[]{X}{_{\bullet}}$ so that $\fs(\delta) = [\tensor*[]{X}{_{\bullet}}]$ and, by \cref{cor:newlift}, an idempotent morphism $\tensor*[]{e}{_{\bullet}} \colon \langle \tensor*[]{X}{_{\bullet}}, \delta \rangle \to \langle \tensor*[]{X}{_{\bullet}}, \delta \rangle$ lifting $(\tensor*[]{e}{_{0}}, \tensor*[]{e}{_{n+1}}) \colon \delta \to \delta$. We put $\ft(\tilde{\delta}) \deff [(\tensor*[]{X}{_{\bullet}}, \tensor*[]{e}{_{\bullet}})]$.
\end{defn}

\begin{rem}\label{rem:independence-of-t}
    For $\tilde{\delta} \in \BF((\tensor*[]{X}{_{n+1}}, \tensor*[]{e}{_{n+1}}),(\tensor*[]{X}{_{0}}, \tensor*[]{e}{_{0}}))$, the definition of $\ft(\tilde{\delta}) = [(\tensor*[]{X}{_{\bullet}}, \tensor*[]{e}{_{\bullet}})]$ depends on neither the choice of $\tensor*[]{X}{_{\bullet}}$ with $[\tensor*[]{X}{_{\bullet}}] = \fs(\delta)$, nor on the choice of $\tensor*[]{e}{_{\bullet}}$ lifting $(\tensor*[]{e}{_{0}}, \tensor*[]{e}{_{n+1}}) \colon \delta \to \delta$ by \cref{lem:independent}. 
    By \cref{cor:newlift}, for each $\tilde{\delta} \in \BF((\tensor*[]{X}{_{n+1}}, \tensor*[]{e}{_{n+1}}),(\tensor*[]{X}{_{0}}, \tensor*[]{e}{_{0}}))$, we can find an $\fs$-distinguished $n$-exangle $\langle \tensor*[]{X}{_{\bullet}} , \delta \rangle$
    and an idempotent morphism $\tensor*[]{e}{_{\bullet}} \colon \langle \tensor*[]{X}{_{\bullet}}, \delta \rangle \to \langle \tensor*[]{X}{_{\bullet}}, \delta \rangle$, 
    such that $\ft(\tilde{\delta}) = [(\tensor*[]{X}{_{\bullet}}, \tensor*[]{e}{_{\bullet}})]$ and $\id{\tensor*[]{X}{_{\bullet}}} - \tensor*[]{e}{_{\bullet}}$ is null homotopic in $\com{\CC}^{\raisebox{0.5pt}{\scalebox{0.6}{$n$}}}$.
\end{rem}

\begin{prop}
\label{prop:exact-realisation}
   The assignment $\ft$ is an exact realisation of $\BF$.
\end{prop}
\begin{proof}
   (R0)\; 
   Suppose $\tilde{\delta} \in \BF((\tensor*[]{X}{_{n+1}}, \tensor*[]{e}{_{n+1}}),(\tensor*[]{X}{_{0}}, \tensor*[]{e}{_{0}}))$
   and 
   $\tilde{\rho} \in \BF((\tensor*[]{Y}{_{n+1}}, e'_{n+1}),(\tensor*[]{Y}{_{0}}, e'_0))$, and 
   let $(\tilde{a},\tilde{c})\colon \tilde{\delta} \to \tilde{\rho}$ be a morphism of $\BF$-extensions. 
   Suppose $\ft(\tilde{\delta}) = [(\tensor*[]{X}{_{\bullet}}, \tensor*[]{e}{_{\bullet}})]$ and $\ft(\tilde{\rho}) = [(\tensor*[]{Y}{_{\bullet}}, e'_{\bullet})]$. 
   Since $(a,c)$ is a morphism of $\BE$-extensions, there is a lift $\tensor*[]{f}{_{\bullet}}\colon \tensor*[]{X}{_{\bullet}} \to \tensor*[]{Y}{_{\bullet}}$ of it using that $\fs$ is an exact realisation of $\BE$. 
   As $\tilde{a}\colon(\tensor*[]{X}{_{0}},\tensor*[]{e}{_{0}}) \to (\tensor*[]{Y}{_{0}},e'_0)$ and $\tilde{c}\colon(\tensor*[]{X}{_{n+1}},\tensor*[]{e}{_{n+1}}) \to (\tensor*[]{Y}{_{n+1}},e'_{n+1})$ are morphisms in $\wt{\CC}$, we have that 
   $e'_0 a = a = a \tensor*[]{e}{_{0}}$ and 
   $e'_{n+1} c = c = c \tensor*[]{e}{_{n+1}}$.
   Hence, by \cref{lem:induced-chain-map}, it follows that 
   $\tilde{g}_{\bullet} \colon (\tensor*[]{X}{_{\bullet}},\tensor*[]{e}{_{\bullet}})\to (\tensor*[]{Y}{_{\bullet}},e'_{\bullet})$ with $\tensor*[]{g}{_{\bullet}} = e'_{\bullet} \tensor*[]{f}{_{\bullet}} \tensor*[]{e}{_{\bullet}}$ lifts $(\tilde{a},\tilde{c})$.
   
   (R1)\; 
   This is \cref{lem:n-exangle-from-idempotent-morphism-and-n-exangle}.
   
   (R2)\; 
   Let $(X,e)\in\wt{\CC}$ be arbitrary. 
   By \cref{rem:comments-on-BF}\ref{item:BF-additive-subfunctor-of-BE}, we have that 
   the zero element of $\BF((0,0),(X,e))$ has the zero element $_{X}0_{0}$ of $\BE(0,X)$ as its underlying $\BE$-extension. 
   Since $\fs$ is an exact realisation of $\BE$, we know 
   \[
   \fs(_{X}0_{0})
   		=[\tensor*[]{X}{_{\bullet}}]
   		= [\begin{tikzcd}[column sep=0.7cm]
		X \arrow{r}{\id{X}}& X \arrow{r}& 0 \arrow{r}& \cdots \arrow{r}& 0
		\end{tikzcd}
		].
   \]
   The tuple $(e,e,0,\ldots,0)\colon \tensor*[]{X}{_{\bullet}} \to \tensor*[]{X}{_{\bullet}}$ is an idempotent morphism lifting $(e,0)\colon _{X}0_{0}\to \tensor*[_{X}]{0}{_{0}}$. 
   Thus, by \cref{def:t} and using $\id{(X,e)} = e$, we see that  
    \[
    \ft(_{(X,e)}\wt{0}_{(0,0)})
   		= [\begin{tikzcd}[column sep=1.1cm]
		(X,e) \arrow{r}{\wid{(X,e)}}& (X,e) \arrow{r}& (0,0) \arrow{r}& \cdots \arrow{r}& (0,0)
		\end{tikzcd}
		].
   \] 
Dually, 
   $
   \ft(_{(0,0)}\wt{0}_{(X,e)})
   		= [\begin{tikzcd}[column sep=1.1cm]
		 (0,0) \arrow{r}& \cdots \arrow{r}& (0,0) \arrow{r}& (X,e) \arrow{r}{\wid{(X,e)}}& (X,e)
		\end{tikzcd}
		]
   $.
\end{proof}

\subsection{The axiom (EA1) for \texorpdfstring{$(\smash{\wt{\CC}},\BF,\ft)$}{(C,F,t)}}
\label{subsec:EA1}

Now that we have a biadditive functor $\BF\colon \tensor*[]{\wt{\CC}}{^{\op}}\times \wt{\CC} \to \Ab$ and an exact realisation $\ft$ of $\BF$, we can begin to verify axioms \ref{EA1}, \ref{EA2} and \ref{EA2op}$\tensor*[]{}{^{\op}}$. 
In this subsection, we will check that the collection of $\ft$-inflations is closed under composition. One can dualise the results here to see that $\ft$-deflations compose to $\ft$-deflations.

The following result only needs that $\fs$ is an exact realisation of $\BE\colon \tensor*[]{\CC}{^{\op}}\times\CC\to\Ab$. It is an analogue of \cite[Lem.\ 2.1]{Klapproth-n-exact-categories-arising-from-nplus2-angulated-categories} for $n$-exangulated categories, allowing us to complete a ``partial'' lift of a morphism of extensions.

\begin{lem}[Completion Lemma] 
\label{lem:CompletionLemma}
	Let $\langle \tensor*[]{X}{_{\bullet}}, \delta \rangle$ and $\langle \tensor*[]{Y}{_{\bullet}}, \rho \rangle$ be $\fs$-distinguished $n$-exangles.
	Let $l,r$ be integers with $0 \leq l \leq r-2 \leq n-1$. 
	Suppose there are morphisms $\tensor*[]{f}{_{0}}, \ldots, \tensor*[]{f}{_{l}}$ and $\tensor*[]{f}{_{r}}, \ldots, \tensor*[]{f}{_{n+1}}$, where $\tensor*[]{f}{_{i}}\colon \tensor*[]{X}{_{i}} \to \tensor*[]{Y}{_{i}}$, such that $(\tensor*[]{f}{_{0}},\tensor*[]{f}{_{n+1}}) \colon \delta \to \rho$ is a morphism of $\BE$-extensions and the solid part of the diagram
	\begin{equation}\label{eqn:completion-lemma-stmt}
      \begin{tikzcd}[column sep = 0.6cm]
         \tensor*[]{X}{_{0}} \arrow{d}{\tensor*[]{f}{_{0}}} \arrow[r] & \cdots \arrow[r] & \tensor*[]{X}{_{l}} \arrow{d}{\tensor*[]{f}{_{l}}} \arrow[r] & \tensor*[]{X}{_{l+1}} \arrow[dotted]{d}{f_{l+1}} \arrow[r] & \cdots \arrow[r] & \tensor*[]{X}{_{r-1}} \arrow[dotted]{d}[swap]{\tensor*[]{f}{_{r-1}}} \arrow[r] & \tensor*[]{X}{_{r}} \arrow{d}[swap]{\tensor*[]{f}{_{r}}} \arrow[r] & \cdots \arrow[r] & \tensor*[]{X}{_{n+1}} \arrow{d}[swap]{\tensor*[]{f}{_{n+1}}} \ar[r, "\delta", dotted] &[-.2em] {} \\
         \tensor*[]{Y}{_{0}} \arrow[r]                  & \cdots \arrow[r] & \tensor*[]{Y}{_{l}} \arrow[r]                  & \tensor*[]{Y}{_{l+1}} \arrow[r]                              & \cdots \arrow[r] & \tensor*[]{Y}{_{r-1}} \arrow[r]                               & \tensor*[]{Y}{_{r}} \arrow[r]                   & \cdots \arrow[r] & \tensor*[]{Y}{_{n+1}}                       \ar[r, "\rho", dotted]   &        {}
      \end{tikzcd}
      \end{equation}
      commutes.
      Then there exist morphisms $\tensor*[]{f}{_{i}}\in \CC(\tensor*[]{X}{_{i}},\tensor*[]{Y}{_{i}})$ for $i\in\{l+1,\ldots, r-1\}$ such that 
      \eqref{eqn:completion-lemma-stmt} commutes. 
\end{lem} 

\begin{proof}
   We proceed by induction on $l \geq 0$. 
   Suppose $l=0$. We induct downwards on $r \leq n+1$. 
   If $r = n+1$, then the result follows from axiom (R0) for $\fs$ since 
   	$(\tensor*[]{f}{_{0}},\tensor*[]{f}{_{n+1}})\colon \delta \to \rho$ is a morphism of $\BE$-extensions.
   Now assume that the results holds for $l=0$ and some $3 \leq r\leq n+1$. 
   Suppose we are given morphisms $\tensor*[]{f}{_{0}} $ and $\tensor*[]{f}{_{r-1}},\tensor*[]{f}{_{r}},\ldots, \tensor*[]{f}{_{n+1}}$ such that 
   $\tensor*[]{f}{_{i+1}} \tensor*[]{d}{^{X}_{i}} = \tensor*[]{d}{^{Y}_{i}} \tensor*[]{f}{_{i}}$ for $i\in\{ r-1,\ldots, n \}$. By the induction hypothesis, we obtain a morphism
   \begin{equation}
   \label{eqn:completion-lemma-proof-i}
      \begin{tikzcd}[column sep = 2em]
      \tensor*[]{X}{_{0}} 
      	\arrow{r}{}
      	\arrow{d}{\tensor*[]{f}{_{0}}}
      & \tensor*[]{X}{_{1}} 
      	\arrow{r}{}
		\arrow[dotted]{d}{\tensor*[]{g}{_{1}}}
      & \cdots
      	\arrow{r}{}
      & \tensor*[]{X}{_{r-1}}
      	\arrow{r}{}
		\arrow[dotted]{d}{\tensor*[]{g}{_{r-1}}}
      & X_r
      	\arrow{r}{}
		\arrow{d}{\tensor*[]{f}{_{r}}}
      & \cdots
      	\arrow{r}{}
      & \tensor*[]{X}{_{n+1}}
      	\arrow[dashed]{r}{\delta}
		\arrow{d}{\tensor*[]{f}{_{n+1}}}
      & {}\\
      \tensor*[]{Y}{_{0}} 
      	\arrow{r}{}
      & \tensor*[]{Y}{_{1}} 
      	\arrow{r}{}
      & \cdots
      	\arrow{r}{}
      & \tensor*[]{Y}{_{r-1}}
      	\arrow{r}{}
      & \tensor*[]{Y}{_{r}}
      	\arrow{r}{}
      & \cdots
      	\arrow{r}{}
      & \tensor*[]{Y}{_{n+1}}
      	\arrow[dashed]{r}{\rho}
      & {}
      \end{tikzcd}
      \end{equation}
	of $\fs$-distinguished $n$-exangles. 
	We will denote this morphism by $\tensor*[]{g}{_{\bullet}}$.
	Next, note that we have $\tensor*[]{d}{^{Y}_{r-1}} (\tensor*[]{f}{_{r-1}} - \tensor*[]{g}{_{r-1}}) = (\tensor*[]{f}{_{r}} - \tensor*[]{f}{_{r}}) \tensor*[]{d}{^{X}_{r-1}} = 0$. 
	Since $\tensor*[]{d}{^{Y}_{r-2}}$ is a weak kernel of $\tensor*[]{d}{^{Y}_{r-1}}$, there exists $h \colon \tensor*[]{X}{_{r-1}} \to \tensor*[]{Y}{_{r-2}}$ so that $\tensor*[]{f}{_{r-1}} - \tensor*[]{g}{_{r-1}} = \tensor*[]{d}{_{r-2}^{Y}} h$.
	Set $\tensor*[]{f}{_{i}} \deff \tensor*[]{g}{_{i}}$ for $1 \leq i \leq r-3$ and $\tensor*[]{f}{_{r-2}} \deff \tensor*[]{g}{_{r-2}} + h\tensor*[]{d}{_{r-2}^{X}}$. Notice that we have $\tensor*[]{f}{_{i}} = \tensor*[]{g}{_{i}}$ for $i \notin \{r-1, r-2\}$.
	We claim that \eqref{eqn:completion-lemma-stmt} commutes. 
	By construction, we only need to check commutativity of the two squares involving $\tensor*[]{f}{_{r-2}}$. 
	These indeed commute since
   \[
   		\tensor*[]{f}{_{r-2}} \tensor*[]{d}{_{r-3}^{X}} 
			= (\tensor*[]{g}{_{r-2}} + h \tensor*[]{d}{_{r-2}^{X}}) \tensor*[]{d}{_{r-3}^{X}} 
			= \tensor*[]{g}{_{r-2}} \tensor*[]{d}{_{r-3}^{X}} 
			= \tensor*[]{d}{_{r-3}^{Y}} \tensor*[]{g}{_{r-3}} 
			= \tensor*[]{d}{_{r-3}^{Y}} \tensor*[]{f}{_{r-3}}
   \]
   and
   \[
   		\tensor*[]{d}{_{r-2}^{Y}} \tensor*[]{f}{_{r-2}}
			= \tensor*[]{d}{_{r-2}^{Y}} ( \tensor*[]{g}{_{r-2}} +h\tensor*[]{d}{_{r-2}^{X}} )
			= \tensor*[]{d}{_{r-2}^{Y}} \tensor*[]{g}{_{r-2}} + (\tensor*[]{f}{_{r-1}} - \tensor*[]{g}{_{r-1}} ) \tensor*[]{d}{_{r-2}^{X}} 
			= \tensor*[]{f}{_{r-1}} \tensor*[]{d}{_{r-2}^{X}},
   \]
   using the commutativity of \eqref{eqn:completion-lemma-proof-i}. This concludes the base case $l=0$. 
   
   The inductive step for $l\geq 0$ is carried out in a similar way to the inductive step above on $r$, using that $\tensor*[]{d}{_{l+1}^{X}}$ is a weak cokernel of $\tensor*[]{d}{^{X}_{l}}$.
\end{proof}

From the Completion  \cref{lem:CompletionLemma} and some earlier results from this section we derive the following, which is used in the main result of this subsection.

\begin{lem} 
\label{lem:inflationcompletion}
   Suppose $\lan \tensor*[]{X}{_{\bullet}}, \delta \ran$ is an $\fs$-distinguished $n$-exangle. 
   Assume $\tensor*[]{e}{_{0}} \colon \tensor*[]{X}{_{0}} \to \tensor*[]{X}{_{0}}$ and $\tensor*[]{e}{_{1}} \colon \tensor*[]{X}{_{1}} \to \tensor*[]{X}{_{1}}$ are idempotents, such that $(\tensor*[]{e}{_{0}}\tensor*[]{)}{_{\BE}} \delta = \delta$ and 
   \[\begin{tikzcd}
        \tensor*[]{X}{_{0}} \arrow{r}{\tensor*[]{d}{_{0}^{X}}} \arrow{d}{\tensor*[]{e}{_{0}}} & \tensor*[]{X}{_{1}} \arrow{d}{\tensor*[]{e}{_{1}}} \\
        \tensor*[]{X}{_{0}} \arrow{r}{\tensor*[]{d}{_{0}^{X}}}               & \tensor*[]{X}{_{1}}
   \end{tikzcd}\]
   commutes. 
   Then $\tensor*[]{e}{_{0}}$ and $\tensor*[]{e}{_{1}}$ can be extended to an idempotent morphism $\tensor*[]{e}{_{\bullet}} \colon \lan \tensor*[]{X}{_{\bullet}}, \delta \ran \to \lan \tensor*[]{X}{_{\bullet}}, \delta \ran$ 
   with $\tensor*[]{e}{_{i}} = \id{\tensor*[]{X}{_{i}}}$ for $3 \leq i \leq n+1$, such that $\tilde{\delta} = (\tensor*[]{e}{_{0}},\delta,\tensor*[]{e}{_{n+1}}) \in \BF((\tensor*[]{X}{_{n+1}}, \tensor*[]{e}{_{n+1}}), (\tensor*[]{X}{_{0}},\tensor*[]{e}{_{0}}))$. 
\end{lem} 

\begin{proof}
		 First, suppose $n = 1$. Then the solid morphisms of the diagram
        \[\begin{tikzcd}[ampersand replacement=\&]
            {\tensor*[]{X}{_{0}}} \& {\tensor*[]{X}{_{1}}} \& {\tensor*[]{X}{_{2}}} \& {}\\
   	        {\tensor*[]{X}{_{0}}} \& {\tensor*[]{X}{_{1}}} \& {\tensor*[]{X}{_{2}}} \& {}
   	        \arrow["{\tensor*[]{d}{_{0}^{X}}}", from=1-1, to=1-2]
   	        \arrow["{\tensor*[]{d}{_{1}^{X}}}",from=1-2, to=1-3]
   	        \arrow["{\delta}", dashed, from=1-3, to=1-4]
   	        \arrow["{\tensor*[]{d}{_{0}^{X}}}", from=2-1, to=2-2]
   	        \arrow["{\tensor*[]{e}{_{0}}}", from=1-1, to=2-1]
   	        \arrow["{\tensor*[]{e}{_{1}}}", from=1-2, to=2-2]
   	        \arrow[dotted, from=1-3, to=2-3, "{\tensor*[]{f}{_{2}}}"]
   	        \arrow["{\tensor*[]{d}{_{1}^{X}}}",from=2-2, to=2-3]
   	        \arrow["{\delta}", dashed, from=2-3, to=2-4]
        \end{tikzcd}\]
        form a commutative diagram, and by \cite[Prop.\ 3.6(1)]{HerschendLiuNakaoka-n-exangulated-categories-I-definitions-and-fundamental-properties} there is a morphism $\tensor*[]{f}{_{2}}$ such that $(\tensor*[]{e}{_{0}}, \tensor*[]{e}{_{1}}, \tensor*[]{f}{_{2}}) \colon \lan \tensor*[]{X}{_{\bullet}}, \delta \ran \to \lan \tensor*[]{X}{_{\bullet}}, \delta \ran$ is a morphism of $\fs$-distinguished $1$-exangles. 
        Recall the polynomial $\tensor*[]{p}{_{2}}$ from \cref{lemma:lift}. 
        We will show that 
        $\tensor*[]{e}{_{\bullet}} = (\tensor*[]{e}{_{0}},\tensor*[]{e}{_{1}},\tensor*[]{e}{_{2}})$, where $\tensor*[]{e}{_{2}} \deff \tensor*[]{p}{_{2}}(\tensor*[]{f}{_{2}})$, 
		is the desired idempotent morphism of $\fs$-distinguished $n$-exangles. 
        
        Since $\langle \tensor*[]{X}{_{\bullet}}, \delta \rangle$ is an $\fs$-distinguished $1$-exangle, 
        there is an exact sequence
        \[
        \begin{tikzcd}[column sep=1.5cm]
            \CC(\tensor*[]{X}{_{2}}, \tensor*[]{X}{_{1}})  
		        \ar{r}{\CC(\tensor*[]{X}{_{2}}, \tensor*[]{d}{_{1}^{X}})}
	        &[1.4em] \CC(\tensor*[]{X}{_{2}}, \tensor*[]{X}{_{2}}) 
	            \arrow{r}{\tensor*[^{\BE}]{\delta}{}}
	        &\BE(\tensor*[]{X}{_{2}}, \tensor*[]{X}{_{0}}).
        \end{tikzcd}
        \]
        As $\tensor*[^{\BE}]{\delta}{} (f^2_2-\tensor*[]{f}{_{2}}) = (f^2_2-\tensor*[]{f}{_{2}}\tensor*[]{)}{^{\BE}} \delta = (e_{0}^{2}-\tensor*[]{e}{_{0}}\tensor*[]{)}{_{\BE}} \delta = \tensor*[_{\tensor*[]{X}{_{0}}}]{0}{_{\tensor*[]{X}{_{2}}}}$, there exists a morphism $\tensor*[]{h}{_{2}} \colon \tensor*[]{X}{_{2}} \to \tensor*[]{X}{_{1}}$ with $\tensor*[]{d}{_{1}^{X}}\tensor*[]{h}{_{2}} = \tensor*[]{f}{_{2}}^2 -\tensor*[]{f}{_{2}}$. 
        This shows that $(\tensor*[]{f}{_{2}}^2 -\tensor*[]{f}{_{2}})^2 = (\tensor*[]{f}{_{2}}^2 -\tensor*[]{f}{_{2}}) \tensor*[]{d}{_{1}^{X}}\tensor*[]{h}{_{2}} = \tensor*[]{d}{_{1}^{X}}(\tensor*[]{e}{_{1}}^2 - \tensor*[]{e}{_{1}})\tensor*[]{h}{_{2}} = 0$ because $\tensor*[]{e}{_{1}}$ is idempotent. 
        Hence, 
         $\tensor*[]{e}{_{\bullet}} = (\tensor*[]{e}{_{0}},\tensor*[]{e}{_{1}},\tensor*[]{e}{_{2}}) = (\tensor*[]{e}{_{0}},\tensor*[]{e}{_{1}}, \tensor*[]{p}{_{2}}(\tensor*[]{f}{_{2}})) \colon \tensor*[]{X}{_{\bullet}} \to \tensor*[]{X}{_{\bullet}}$ is an idempotent morphism of complexes by \cref{lemma:lift}\ref{4.9.2}. 
        Furthermore, 
        $(\tensor*[]{e}{_{2}}\tensor*[]{)}{^{\BE}} \delta
        		= (\tensor*[]{p}{_{2}}(\tensor*[]{f}{_{2}})\tensor*[]{)}{^{\BE}} \delta
			=  (\tensor*[]{p}{_{2}}(\tensor*[]{e}{_{0}})\tensor*[]{)}{_{\BE}} \delta
			= (\tensor*[]{e}{_{0}}\tensor*[]{)}{_{\BE}} \delta
			= \delta$ using \cref{lemma:lift}\ref{4.9.1}, 
			so that $(\tensor*[]{e}{_{0}},\tensor*[]{e}{_{2}})\colon \delta\to\delta$ is a morphism of $\BE$-extensions. 
        This computation also shows the existence of $\tilde{\delta} \in \BF((\tensor*[]{X}{_{2}},\tensor*[]{e}{_{2}}), (\tensor*[]{X}{_{0}}, \tensor*[]{e}{_{0}}))$ with underlying $\BE$-extension $\delta$.

	    Now suppose $n \geq 2$. 
	    We have $(\tensor*[]{e}{_{0}}\tensor*[]{)}{_{\BE}} \delta = \delta = (\id{\tensor*[]{X}{_{n+1}}}\tensor*[]{)}{^{\BE}} \delta$.
	    Therefore, $\tilde{\delta} = (\tensor*[]{e}{_{0}}, \delta, \id{\tensor*[]{X}{_{n+1}}})$ is an element of $\BF((\tensor*[]{X}{_{n+1}}, \id{\tensor*[]{X}{_{n+1}}}), (\tensor*[]{X}{_{0}}, \tensor*[]{e}{_{0}}))$ with underlying $\BE$-extension $\delta$.
The solid morphisms of the diagram
        \[\begin{tikzcd}[ampersand replacement=\&]
       	    {\tensor*[]{X}{_{0}}}
	    			 \arrow{r}{\tensor*[]{d}{_{0}^{X}}}
				 \arrow{d}{\tensor*[]{e}{_{0}}}
			\& {\tensor*[]{X}{_{1}}} 
				\arrow{d}{\tensor*[]{e}{_{1}}}
			\& {\tensor*[]{X}{_{2}}} 
				\arrow[dotted]{d}{\tensor*[]{e}{_{2}}}
			\& {\tensor*[]{X}{_{3}}} 
			\& \cdots 
			\& {\tensor*[]{X}{_{n}}} 
			\& {\tensor*[]{X}{_{n+1}}} 
			\& {} \\
   	        {\tensor*[]{X}{_{0}}} 
	        		\arrow{r}{\tensor*[]{d}{_{0}^{X}}}
	        \& {\tensor*[]{X}{_{1}}} 
	        \& {\tensor*[]{X}{_{2}}} 
	        \& {\tensor*[]{X}{_{3}}} 
	        \& \cdots 
	        \& {\tensor*[]{X}{_{n}}} 
	        \& {\tensor*[]{X}{_{n+1}}} 
	        \& {}
       	    \arrow["{}",from=1-2, to=1-3]  
       	    \arrow[from=1-5, to=1-6]
       	    \arrow[from=1-6, to=1-7]
   	        \arrow["{\delta}", dashed, from=1-7, to=1-8]
   	        \arrow[from=1-3, to=1-4]
       	    \arrow[from=1-4, to=1-5]
   	        \arrow[Rightarrow, no head, from=1-4, to=2-4]
   	        \arrow[Rightarrow, no head, from=1-6, to=2-6]
       	    \arrow[Rightarrow, no head, from=1-7, to=2-7]
   	        \arrow["{\delta}", dashed, from=2-7, to=2-8]
            \arrow["{}", from=2-2, to=2-3]
            \arrow[from=2-3, to=2-4]
            \arrow[from=2-4, to=2-5]
            \arrow[from=2-5, to=2-6]
            \arrow[from=2-6, to=2-7]
        \end{tikzcd}\]
        form a commutative diagram, 
        and $(\tensor*[]{e}{_{0}}, \id{\tensor*[]{X}{_{n+1}}}) \colon \delta \to \delta$ is a morphism of $\BE$-extensions as $(\tensor*[]{e}{_{0}}\tensor*[]{)}{_{\BE}}\delta = \delta$.
        Since the rows are the $\fs$-distinguished $n$-exangle $\langle \tensor*[]{X}{_{\bullet}}, \delta \rangle$, 
        by \cref{lem:CompletionLemma} we can find a morphism $\tensor*[]{e}{_{2}}\in\tensor*[]{\End}{_{\CC}}(\tensor*[]{X}{_{2}})$, so that the diagram above is a morphism $\lan \tensor*[]{X}{_{\bullet}},\delta\ran \to \lan \tensor*[]{X}{_{\bullet}},\delta\ran$.
        Furthermore, as $\tensor*[]{e}{_{1}} \in \tensor*[]{\End}{_{\CC}}(\tensor*[]{X}{_{1}})$ and $\id{\tensor*[]{X}{_{3}}}\in\tensor*[]{\End}{_{\CC}}(\tensor*[]{X}{_{3}})$ are idempotent, we may assume that $\tensor*[]{e}{_{2}}$ is an idempotent by \cref{lem:IdempotentTrick}.
\end{proof}

Given a $\ft$-inflation $\tilde{f}$ that fits into a $\ft$-distinguished $n$-exangle $\langle \tensor*[]{\wt{Y}}{_{\bullet}}, \tilde{\delta} \rangle$, we cannot a priori say too much about how $\tensor*[]{\wt{Y}}{_{\bullet}}$ might look. This is one of the main issues in trying to prove \ref{EA1} for $(\wt{\CC},\BF,\ft)$. The next lemma gives us a way to deal with this and is the last preparatory result we need before the main result of this subsection.

\begin{lem} 
\label{lem:TwistInflation}
   Let $\tilde{f} \colon (\tensor*[]{X}{_{0}}, \tensor*[]{e}{_{0}}) \to (\tensor*[]{X}{_{1}}, \tensor*[]{e}{_{1}})$ be a $\ft$-inflation. 
   Then there is an $\fs$-distinguished $n$-exangle $\langle X'_{\bullet}, \delta \rangle$ with $X'_0 = \tensor*[]{X}{_{0}}$ and $X'_{1} = \tensor*[]{X}{_{1}} \oplus C$ for some $C \in \CC$, such that $(\tensor*[]{e}{_{0}}\tensor*[]{)}{_{\BE}} \delta = \delta$ and 
     $\tensor*[]{d}{^{X'}_{0}} = 
     \begin{bsmallmatrix}
   f &\; f'(\id{\tensor*[]{X}{_{0}}}-\tensor*[]{e}{_{0}})
   \end{bsmallmatrix}^\top
   \colon \tensor*[]{X}{_{0}} \to \tensor*[]{X}{_{1}} \oplus C
   $
   for some $f' \colon \tensor*[]{X}{_{0}} \to C$.
\end{lem}

\begin{proof}
   Since $f \colon (\tensor*[]{X}{_{0}}, \tensor*[]{e}{_{0}}) \to (\tensor*[]{X}{_{1}}, \tensor*[]{e}{_{1}})$ is an $\ft$-inflation, there is a $\ft$-distinguished $n$-exangle $\langle \tensor*[]{\wt{Y}}{_{\bullet}}, \tilde{\delta}' \rangle$ with $\tensor*[]{\wt{Y}}{_{0}} = (\tensor*[]{X}{_{0}}, \tensor*[]{e}{_{0}})$, $\tensor*[]{\wt{Y}}{_{1}} = (\tensor*[]{X}{_{1}}, \tensor*[]{e}{_{1}})$ and $\tilde{d}_0^{\wt{Y}} = \tilde{f}$.
   By definition of $\ft$, this means there is an $\fs$-distinguished $n$-exangle $\langle Y'_{\bullet}, \delta' \rangle$ 
   with an idempotent morphism $e'_{\bullet} \colon \langle Y'_{\bullet}, \delta' \rangle \to \langle Y'_{\bullet}, \delta' \rangle$, 
   such that $(Y'_{0}, e'_0) = \tensor*[]{\wt{Y}}{_{0}}$ and $(Y'_{n+1}, e'_{n+1}) = \tensor*[]{\wt{Y}}{_{n+1}}$, and there are mutually inverse homotopy equivalences $\tilde{r}_{\bullet} \colon \tensor*[]{\wt{Y}}{_{\bullet}} \to (Y'_{\bullet}, e'_{\bullet})$ 
   and 
   $\tensor*[]{\tilde{s}}{_{\bullet}} \colon (Y'_{\bullet}, e'_{\bullet}) \to \tensor*[]{\wt{Y}}{_{\bullet}}$ which satisfy $\tilde{r}_0 = \smash{\wid{\tensor*[]{\wt{Y}}{_{0}}}} = \tilde{s}_0$ and $\tilde{r}_{n+1} = \smash{\wid{\tensor*[]{\wt{Y}}{_{n+1}}}} = \tensor*[]{\tilde{s}}{_{n+1}}$. 
   Note that we, thus, have $Y'_0 = \tensor*[]{X}{_{0}}$ and $e'_0 = \tensor*[]{e}{_{0}}$. 
	In particular, we have a commutative diagram
   \[
   \begin{tikzcd}[row sep=0.4cm]
   (\tensor*[]{X}{_{0}},\tensor*[]{e}{_{0}}) \arrow{r}{\tilde{f}} \arrow[equals]{d}{}
   		& (\tensor*[]{X}{_{1}},\tensor*[]{e}{_{1}}) \arrow{r}{} \arrow{d}{\tilde{r}_1}
		& \tensor*[]{\wt{Y}}{_{2}} \arrow{r}{} \arrow{d}{\tilde{r}_2}
		& \cdots \arrow{r}{}
		& \tensor*[]{\wt{Y}}{_{n}} \arrow{r}{} \arrow{d}{\tilde{r}_n}
		& \tensor*[]{\wt{Y}}{_{n+1}} \arrow[dashed]{r}{\tilde{\delta}'} \arrow[equals]{d}{}
		& {}\\ 
	(\tensor*[]{X}{_{0}},\tensor*[]{e}{_{0}}) \arrow{r}{\tilde{t}} \arrow[equals]{d}{}
   		& (Y'_{1},e'_1) \arrow{r}{} \arrow{d}{\tensor*[]{\tilde{s}}{_{1}}}
		& (Y'_2,e'_2) \arrow{r}{} \arrow{d}{\tensor*[]{\tilde{s}}{_{2}}}
		& \cdots \arrow{r}{}
		& (Y'_n,e'_n) \arrow{r}{} \arrow{d}{\tensor*[]{\tilde{s}}{_{n}}}
		& (Y'_{n+1},e'_{n+1}) \arrow[dashed]{r}{\tilde{\delta}'} \arrow[equals]{d}{}
		& {}\\
	(\tensor*[]{X}{_{0}},\tensor*[]{e}{_{0}}) \arrow{r}{\tilde{f}}
   		& (\tensor*[]{X}{_{1}},\tensor*[]{e}{_{1}}) \arrow{r}{}
		& \tensor*[]{\wt{Y}}{_{2}} \arrow{r}{}
		& \cdots \arrow{r}{}
		& \tensor*[]{\wt{Y}}{_{n}} \arrow{r}{}
		& \tensor*[]{\wt{Y}}{_{n+1}} \arrow[dashed]{r}{\tilde{\delta}'}
		& {}
   \end{tikzcd}
   \]
   in $(\wt{\CC}, \BF, \ft)$, where $t = e'_1 \tensor*[]{d}{^{Y'}_{0}} \tensor*[]{e}{_{0}} = \tensor*[]{d}{^{Y'}_{0}} \tensor*[]{e}{_{0}} = e'_1 \tensor*[]{d}{^{Y'}_{0}} $.
   
   Consider the complex 
   $Y''_{\bullet} \deff \tensor*[]{\triv}{_{1}}(\tensor*[]{X}{_{1}})_{\bullet}$ 
   and the $\BE$-extension $\delta'' \deff \tensor*[_{0}]{0}{_{Y''_{n+1}}}$. 
   Note that if $n=1$, then 
   $Y''_{n+1} = \tensor*[]{X}{_{1}}$; otherwise we have $Y''_{n+1} =0$.
   In either case, we have an $\fs$-distinguished $n$-exangle
   $\lan Y''_{\bullet} , \delta'' \ran$ 
   using the axiom \ref{R2} for $\fs$, 
   and hence also an $\fs$-distinguished $n$-exangle 
   $\langle  Y''_{\bullet} \oplus Y'_{\bullet}, \delta'' \oplus \delta' \rangle$ 
   by \cite[Prop.\ 3.3]{HerschendLiuNakaoka-n-exangulated-categories-I-definitions-and-fundamental-properties}. 
   Using the canonical isomorphism $u \colon 0 \oplus \tensor*[]{X}{_{0}}  \to \tensor*[]{X}{_{0}}$
   we see that 
	the complex 
   \[
   \tensor*[]{Z}{_{\bullet}}\colon\hspace{.5cm}
   \begin{tikzcd}[ampersand replacement = \&,column sep=1.45cm]
   \tensor*[]{X}{_{0}} 
   	\arrow{r}{\begin{bsmallmatrix}0 \\ \tensor*[]{d}{^{Y'}_{0}}\end{bsmallmatrix}}
   \&[-.58em] \tensor*[]{X}{_{1}} \oplus Y'_{1}
   	\arrow{r}{
	\begin{bsmallmatrix}\id{\tensor*[]{X}{_{1}}} & 0 \\ 0 & d^{Y'}_1\end{bsmallmatrix}
	}
   \&[.5em] \tensor*[]{X}{_{1}} \oplus Y'_{2}
   	\arrow{r}{\begin{bsmallmatrix}0 & 0 \\ 0 & \tensor*[]{d}{^{Y'}_{2}}\end{bsmallmatrix}}
   \&[-.48em] 0 \oplus Y'_{3}
   	\arrow{r}{\begin{bsmallmatrix}0 & 0 \\ 0 & d^{Y'}_3\end{bsmallmatrix}}
   \&[-.48em] \cdots
   	\arrow{r}{\begin{bsmallmatrix}0 & 0 \\ 0 & d^{Y'}_{n}\end{bsmallmatrix}}
   \&[-.48em] 0 \oplus Y'_{n+1}
   \end{tikzcd}
   \]
   realises $\delta \deff \tensor*{u}{_\BE} (\delta'' \oplus \delta')$ in $(\CC,\BE,\fs)$ by \cite[Cor.\ 2.26(2)]{HerschendLiuNakaoka-n-exangulated-categories-I-definitions-and-fundamental-properties}. 
   Consider the diagram
   \begin{equation*}\label{eqn:comm-squares-in-3-13}
   \begin{tikzcd}[ampersand replacement = \&, column sep=3cm, row sep=0.7cm]
	\tensor*[]{X}{_{0}}
		\arrow{r}{\begin{bsmallmatrix} 0 \\ \tensor*[]{d}{^{Y'}_{0}} \end{bsmallmatrix}}
		\arrow[equals]{d}{}
	\& \tensor*[]{X}{_{1}} \oplus Y'_{1}
		\arrow{d}{a}
		\\
	\tensor*[]{X}{_{0}}
		\arrow{r}{\begin{bsmallmatrix} f \\ \tensor*[]{d}{^{Y'}_{0}} \end{bsmallmatrix}}
		\arrow[equals]{d}{}
	\&  \tensor*[]{X}{_{1}} \oplus Y'_{1}
		\arrow{d}{b}
		\\
	\tensor*[]{X}{_{0}}
		\arrow{r}{\begin{bsmallmatrix} f \\ \tensor*[]{d}{^{Y'}_{0}}(\id{\tensor*[]{X}{_{0}}} - \tensor*[]{e}{_{0}}) \end{bsmallmatrix}}
	\&  \tensor*[]{X}{_{1}} \oplus Y'_{1}
   \end{tikzcd}
   \end{equation*}
   in $\CC$, where $a \deff {\begin{bsmallmatrix} \id{\tensor*[]{X}{_{1}}} & s_1 \\ 0 & \id{Y'_{1}} \end{bsmallmatrix}}$ and $b \deff {\begin{bsmallmatrix} \id{\tensor*[]{X}{_{1}}} & 0 \\ -r_1 & \id{Y'_{1}} \end{bsmallmatrix}}$.
   This diagram commutes since
   \[
   \begin{bsmallmatrix} \id{\tensor*[]{X}{_{1}}} & s_1 \\ 0 & \id{Y'_{1}} \end{bsmallmatrix}\begin{bsmallmatrix} 0 \\ \tensor*[]{d}{^{Y'}_{0}} \end{bsmallmatrix}
   	 = \begin{bsmallmatrix} s_1 \tensor*[]{d}{^{Y'}_{0}} \\ \tensor*[]{d}{^{Y'}_{0}} \end{bsmallmatrix}
   	 = \begin{bsmallmatrix} s_1 e'_{1} \tensor*[]{d}{^{Y'}_{0}} \\ \tensor*[]{d}{^{Y'}_{0}} \end{bsmallmatrix}
	 = \begin{bsmallmatrix} s_1 t \\ \tensor*[]{d}{^{Y'}_{0}} \end{bsmallmatrix}
	 = \begin{bsmallmatrix} f \\ \tensor*[]{d}{^{Y'}_{0}} \end{bsmallmatrix}
   \]
   and 
   \[
   \begin{bsmallmatrix} \id{\tensor*[]{X}{_{1}}} & 0 \\ -r_1 & \id{Y'_{1}}\end{bsmallmatrix}\begin{bsmallmatrix} f \\ \tensor*[]{d}{^{Y'}_{0}} \end{bsmallmatrix}
   	= \begin{bsmallmatrix} f \\ \tensor*[]{d}{^{Y'}_{0}} - r_1 f \end{bsmallmatrix}
	= \begin{bsmallmatrix} f \\ \tensor*[]{d}{^{Y'}_{0}} - t \end{bsmallmatrix}
	= \begin{bsmallmatrix} f \\ \tensor*[]{d}{^{Y'}_{0}}(\id{\tensor*[]{X}{_{0}}} - \tensor*[]{e}{_{0}}) \end{bsmallmatrix}.
   \]
   Notice that the composition $ba$ is an automorphism of $\tensor*[]{X}{_{1}} \oplus Y'_{1}$, 
   and so the complex 
    \[
    X'_{\bullet}\colon\hspace{0.5cm}
   \begin{tikzcd}[column sep=1cm]
   \tensor*[]{X}{_{0}} 
   	\arrow{r}[yshift=3pt]{\begin{bsmallmatrix} f \amph\; \tensor*[]{d}{^{Y'}_{0}}(\id{\tensor*[]{X}{_{0}}} - \tensor*[]{e}{_{0}}) \end{bsmallmatrix}^\top}
   &[1.7cm] \tensor*[]{X}{_{1}} \oplus Y'_{1}
   	\arrow{r}{
	d_1^Z (ba)^{-1}
	}
   &[0.8cm] %
   \tensor*[]{Z}{_{2}}
   	\arrow{r}{d_2^Z}
   & \tensor*[]{Z}{_{3}}
   	\arrow{r}{d^{Z}_3}
   & \cdots
   	\arrow{r}{\tensor*[]{d}{^{Z}_{n}}}
   & \tensor*[]{Z}{_{n+1}}
   \end{tikzcd}
   \]
   forms part of an $\fs$-distinguished $n$-exangle $\langle X'_{\bullet}, \delta \rangle$ by \cite[Cor.\ 2.26(2)]{HerschendLiuNakaoka-n-exangulated-categories-I-definitions-and-fundamental-properties}. 
   We have
   \[(\tensor*[]{e}{_{0}}\tensor*[]{)}{_{\BE}} \delta = (\tensor*[]{e}{_{0}}\tensor*[]{)}{_{\BE}}\tensor*{u}{_{\BE}} (\delta'' \oplus \delta') = \tensor*{u}{_{\BE}} ( \tensor*{0}{_\BE}\delta'' \oplus (\tensor*[]{e}{_{0}}\tensor*[]{)}{_{\BE}} \delta' ) = \tensor*{u}{_\BE} ( \delta'' \oplus \delta') = \delta\] 
   as $\tensor{e}{_0} u = u (0 \oplus \tensor{e}{_0})$, $\delta'' = \tensor*[_{0}]{0}{_{Y''_{n+1}}}$  and $\tilde{\delta}' \in\BF(\tensor*[]{\wt{Y}}{_{n+1}},(\tensor*[]{X}{_{0}},\tensor*[]{e}{_{0}}))$. 
   Setting $f' \deff \tensor*[]{d}{^{Y'}_{0}}$ and $C \deff Y'_{1}$ finishes the proof. 
\end{proof}

We close this subsection with the following result, which together with its dual demonstrates that axiom \ref{EA1} holds for $(\wt{\CC},\BF,\ft)$.

\begin{prop}%
\label{prop:AxiomEA1}
   Suppose $\tilde{f} \colon (\tensor*[]{X}{_{0}}, \tensor*[]{e}{_{0}}) \to (\tensor*[]{X}{_{1}}, \tensor*[]{e}{_{1}})$ and $\tilde{g} \colon (\tensor*[]{Y}{_{0}}, e'_{0}) \to (\tensor*[]{Y}{_{1}}, e'_{1})$ are $\ft$-inflations with $(\tensor*[]{X}{_{1}}, \tensor*[]{e}{_{1}}) = (\tensor*[]{Y}{_{0}}, e'_{0})$. 
   Then $\tilde{g} \tilde{f} \colon (\tensor*[]{X}{_{0}}, \tensor*[]{e}{_{0}}) \to (\tensor*[]{Y}{_{1}}, e'_1)$ is a $\ft$-inflation. 
\end{prop}
\begin{proof}
   By \cref{lem:TwistInflation}, there exists an $\fs$-distinguished $n$-exangle $\langle X'_{\bullet}, \delta \rangle$ with $X'_{0} = \tensor*[]{X}{_{0}}$ and $X'_{1} = \tensor*[]{X}{_{1}} \oplus C$ for some $C \in \CC$, so that 
   $
   \tensor*[]{d}{^{X'}_{0}} 
		= \begin{bsmallmatrix}
   f & \; f'(\id{\tensor*[]{X}{_{0}}}-\tensor*[]{e}{_{0}})
   \end{bsmallmatrix}^\top
   $ 
   for some $f' \colon \tensor*[]{X}{_{0}} \to C$ and $(\tensor*[]{e}{_{0}}\tensor*[]{)}{_{\BE}}\delta = \delta$.
   Similarly, there is also an $\fs$-distinguished $n$-exangle $\langle Y'_{\bullet}, \delta' \rangle$ with $Y'_{0} = \tensor*[]{Y}{_{0}} = \tensor*[]{X}{_{1}}$ and $Y'_{1} = \tensor*[]{Y}{_{1}} \oplus C'$ for some $C' \in \CC$, so that 
   $
   d^{Y'}_0 
		= \begin{bsmallmatrix}
		   g &\; g'(\id{\tensor*[]{X}{_{1}}}-e'_0)
		   \end{bsmallmatrix}^\top
	$ 
	for some $g' \colon \tensor*[]{X}{_{1}} \to C'$ and $(e'_{0}\tensor*[]{)}{_{\BE}}\delta' = \delta'$.
   Setting $Y''_{\bullet} \deff \tensor*[]{\triv}{_{0}}(C)_{\bullet}$, we also have the $n$-exangle $\lan Y''_{\bullet}, \tensor*[_{C}]{0}{_{0}} \ran$ by axiom (R2) for $\fs$.
   Then $\langle Y'_{\bullet} \oplus Y''_{\bullet}, \delta' \oplus \tensor*[_{C}]{0}{_{0}} \rangle$ is $\fs$-distinguished by \cite[Prop.\ 3.3]{HerschendLiuNakaoka-n-exangulated-categories-I-definitions-and-fundamental-properties}.
   We have $Y'_0 \oplus Y''_0 = \tensor*[]{X}{_{1}} \oplus C$ 
   and $Y'_1 \oplus Y''_1 = \tensor*[]{Y}{_{1}} \oplus C' \oplus C$, and 
   \[
   d^{Y' \oplus Y''}_0 
   	= \begin{bsmallmatrix} g & 0 \\ g'(\id{\tensor*[]{X}{_{1}}}-e'_{0}) & 0 \\ 0 & \id{C} \end{bsmallmatrix}
	\]
   is the $\fs$-inflation of $\langle Y'_{\bullet} \oplus Y''_{\bullet}, \delta' \oplus \tensor*[_{C}]{0}{_{0}} \rangle$ with respect to the given decompositions.
   Since $\tensor*[]{d}{^{X'}_{0}}$ and $d^{Y'\oplus Y''}_0$ are $\fs$-inflations, by \ref{EA1} for $(\CC, \BE, \fs)$, we have that the morphism 
	\[
	d_0^{Y' \oplus Y''} \tensor*[]{d}{_{0}^{X'}}
		= \begin{bsmallmatrix} g & 0 \\ g'(\id{\tensor*[]{X}{_{1}}}-e'_{0}) & 0 \\ 0 & \id{C} \end{bsmallmatrix}
		\begin{bsmallmatrix}f\\ f'(\id{\tensor*[]{X}{_{0}}} - \tensor*[]{e}{_{0}})\end{bsmallmatrix}
		= \begin{bsmallmatrix} gf \\ g'(\id{\tensor*[]{X}{_{1}}} - e'_{0})f \\ f'(\id{\tensor*[]{X}{_{0}}} - \tensor*[]{e}{_{0}}) \end{bsmallmatrix}
		= \begin{bsmallmatrix}gf \\ 0 \\ f'(\id{\tensor*[]{X}{_{0}}} - \tensor*[]{e}{_{0}})\end{bsmallmatrix}
	\] 
	is an $\fs$-inflation, where we used that $e'_{0}f = \tensor*[]{e}{_{1}} f = f$.
	Therefore, there is an $\fs$-distinguished $n$-exangle $\langle Z''_{\bullet}, \delta'' \rangle$ with $Z''_{0} = \tensor*[]{X}{_{0}}$, $Z''_{1} = \tensor*[]{Y}{_{1}} \oplus C' \oplus C$ and 
	$
	d^{Z''}_0  
		= \begin{bsmallmatrix}
			   gf &\; 0 &\; f'(\id{\tensor*[]{X}{_{0}}}-\tensor*[]{e}{_{0}})
		   \end{bsmallmatrix}^\top
	$.

    Our next aim is to apply \cref{lem:inflationcompletion} to $\langle Z''_{\bullet}, \delta'' \rangle$. Thus, we claim that $(\tensor*[]{e}{_{0}}\tensor*[]{)}{_{\BE}}\delta'' = \delta''$.
   Since 
   $d^{Y'\oplus Y''}_0 \tensor*[]{d}{^{X'}_{0}} = d^{Z''}_0$, 
   we can apply \cite[Prop.\ 3.6(1)]{HerschendLiuNakaoka-n-exangulated-categories-I-definitions-and-fundamental-properties} to obtain a morphism
      \[
      \begin{tikzcd}[column sep=1.2cm]
   	\tensor*[]{X}{_{0}}
		\arrow{r}{\tensor*[]{d}{_{0}^{Z''}}}
		\arrow{d}{\tensor*[]{d}{_{0}^{X'}}}
	& \tensor*[]{Y}{_{1}} \oplus C' \oplus C
		\arrow{r}{}
		\arrow[equals]{d}{}
	&[-.9em] Z''_{2}
		\arrow{r}{}
		\arrow[dotted]{d}{}
	&[-.9em] \cdots 
		\arrow{r}{}
	&[-.9em] Z''_{n}
		\arrow{r}{}
		\arrow[dotted]{d}{}
	&[-.9em] Z''_{n+1}
		\arrow[dashed]{r}{\delta''}
		\arrow[dotted]{d}[swap]{\tensor*[]{l}{_{n+1}}}
	&[+.7em] {} \\
	   \tensor*[]{X}{_{1}} \oplus C
	   		\arrow{r}[xshift=3pt]{d_0^{Y' \oplus Y''}}
	   & \tensor*[]{Y}{_{1}} \oplus C' \oplus C
	   		\arrow{r}{}
	   & Y'_{2} \oplus 0
	   		\arrow{r}{}
	   & \cdots 
	   		\arrow{r}{}
	   & Y'_{n} \oplus 0
	   		\arrow{r}{}
	   & Y'_{n+1} \oplus 0
	   		\arrow[dashed]{r}[yshift=2pt]{\delta' \oplus \tensor*[_{C}]{0}{_{0}}}
	   & {}	
   \end{tikzcd}
   \]
	of $\fs$-distinguished $n$-exangles. In particular, we have that 
	\begin{equation}\label{eqn:EA1-morphism-of-extensions}
	(\tensor*[]{d}{_{0}^{X'}}\tensor*[]{)}{_{\BE}} \delta'' = (\tensor*[]{l}{_{n+1}}\tensor*[]{)}{^{\BE}}(\delta' \oplus \tensor*[_{C}]{0}{_{0}}).
	\end{equation}	
   As $e'_{0} = \tensor*[]{e}{_{1}}$, $\tensor*[]{e}{_{1}}f = f = f\tensor*[]{e}{_{0}} $ and $\tensor*[]{e}{_{0}}$ is idempotent, 
   we see that 
   \begin{equation}\label{eqn:EA1-differential-0-e-0}
   \tensor*[]{d}{_{0}^{X'}} \tensor*[]{e}{_{0}} 
   		= \begin{bsmallmatrix} f \tensor*[]{e}{_{0}} \\ 0 \end{bsmallmatrix} 
		= \begin{bsmallmatrix} \tensor*[]{e}{_{1}} f \\ 0 \end{bsmallmatrix} 
		= \begin{bsmallmatrix} \tensor*[]{e}{_{1}} & 0 \\ 0 & 0 \end{bsmallmatrix} 
			\begin{bsmallmatrix} f \\ f'(\id{\tensor*[]{X}{_{0}}} - \tensor*[]{e}{_{0}}) \end{bsmallmatrix} 
		= \begin{bsmallmatrix} e'_{0} & 0 \\ 0 & 0 \end{bsmallmatrix} \tensor*[]{d}{_{0}^{X'}}.
   \end{equation}
   This implies that 
   \begin{align*}
   (\tensor*[]{d}{_{0}^{X'}}\tensor*[]{)}{_{\BE}} (\id{\tensor*[]{X}{_{0}}}-\tensor*[]{e}{_{0}}\tensor*[]{)}{_{\BE}} \delta'' 
   		&= (\tensor*[]{d}{_{0}^{X'}} - \tensor*[]{d}{_{0}^{X'}}\tensor*[]{e}{_{0}}\tensor*[]{)}{_{\BE}} \delta'' \\
   		&= \left(\tensor*[]{d}{_{0}^{X'}} - \begin{bsmallmatrix} e'_{0} & 0 \\ 0 & 0 \end{bsmallmatrix}\tensor*[]{d}{_{0}^{X'}}\right)_{\BE}\delta'' && \text{by } \eqref{eqn:EA1-differential-0-e-0}\\
		&= \left(\id{\tensor*[]{X}{_{1}} \oplus C} -  \begin{bsmallmatrix} e'_{0} & 0 \\ 0 & 0 \end{bsmallmatrix}\right)_{\BE} (\tensor*[]{d}{_{0}^{X'}}\tensor*[]{)}{_{\BE}} \delta''\\
		&= \begin{bsmallmatrix} \id{\tensor*[]{X}{_{1}}} - e'_{0} & 0 \\ 0 & \id{C} \end{bsmallmatrix}_\BE (\tensor*[]{l}{_{n+1}}\tensor*[]{)}{^{\BE}} (\delta' \oplus \tensor*[_{C}]{0}{_{0}})&& \text{by } \eqref{eqn:EA1-morphism-of-extensions}\\
		&= (\tensor*[]{l}{_{n+1}}\tensor*[]{)}{^{\BE}} \left( (\id{\tensor*[]{X}{_{1}}} - e'_{0}\tensor*[]{)}{_{\BE}}\delta' \oplus (\id{C}\tensor*[]{)}{_{\BE}}\,\tensor*[_{C}]{0}{_{0}}\right)\\
		&= \tensor*[_{\tensor*[]{X}{_{1}}\oplus C}]{0}{_{Z''_{n+1}}} &&\text{as } (e'_{0}\tensor*[]{)}{_{\BE}}\delta' = \delta'.
   \end{align*}

   Since $\langle X'_{\bullet}, \delta \rangle$ is an $\fs$-distinguished $n$-exangle, by \cite[Lem.\  3.5]{HerschendLiuNakaoka-n-exangulated-categories-I-definitions-and-fundamental-properties} 
   there is an exact sequence
   \[
   \begin{tikzcd}[column sep=1.3cm]
   \CC(Z''_{n+1}, X'_{n+1})  
		\arrow{r}{\tensor*[^{\BE}]{\delta}{}}
	&\BE(Z''_{n+1}, \tensor*[]{X}{_{0}}) 
		\arrow{r}{(\tensor*[]{d}{_{0}^{X'}}\tensor*[]{)}{_{\BE}}}
	&\BE(Z''_{n+1}, \tensor*[]{X}{_{1}} \oplus C).
   \end{tikzcd}
   \]
   As seen above, $(\id{\tensor*[]{X}{_{0}}}-\tensor*[]{e}{_{0}}\tensor*[]{)}{_{\BE}}\delta''$ vanishes under $(\tensor*[]{d}{_{0}^{X'}}\tensor*[]{)}{_{\BE}}$, so  
   there is a morphism $\tensor*[]{r}{_{n+1}} \colon Z''_{n+1} \to X'_{n+1}$ with 
	$\tensor*[^{\BE}]{\delta}{}(\tensor*[]{r}{_{n+1}}) 
		= (\tensor*[]{r}{_{n+1}}\tensor*[]{)}{^{\BE}}\delta 
		= (\id{\tensor*[]{X}{_{0}}}-\tensor*[]{e}{_{0}}\tensor*[]{)}{_{\BE}} \delta''
	$.
   Since $(\id{\tensor*[]{X}{_{0}}}-\tensor*[]{e}{_{0}}\tensor*[]{)}{_{\BE}} \delta = \tensor*[_{\tensor*[]{X}{_{0}}}]{0}{_{X'_{n+1}}}$, this implies
   \[
   \tensor*[_{\tensor*[]{X}{_{0}}}]{0}{_{Z''_{n+1}}}
   		= (\tensor*[]{r}{_{n+1}}\tensor*[]{)}{^{\BE}} (\id{\tensor*[]{X}{_{0}}}-\tensor*[]{e}{_{0}}\tensor*[]{)}{_{\BE}} \delta 
		= (\id{\tensor*[]{X}{_{0}}}-\tensor*[]{e}{_{0}}\tensor*[]{)}{_{\BE}}  (\tensor*[]{r}{_{n+1}}\tensor*[]{)}{^{\BE}} \delta 
		= \left((\id{\tensor*[]{X}{_{0}}}-\tensor*[]{e}{_{0}})^2\right)_{\BE}\delta'' 
		= (\id{\tensor*[]{X}{_{0}}}-\tensor*[]{e}{_{0}}\tensor*[]{)}{_{\BE}} \delta'',
	\] 
	showing that $(\tensor*[]{e}{_{0}}\tensor*[]{)}{_{\BE}} \delta'' = \delta''$.
	
	Now consider the idempotent $e'_{1} \oplus 0 \oplus 0\in\tensor*[]{\End}{_{\CC}}(\tensor*[]{Y}{_{1}}\oplus C' \oplus C)$. 
	A quick computation yields 
    the equality 
	$(e'_{1} \oplus 0 \oplus 0)\tensor*[]{d}{_{0}^{Z''}} = \tensor*[]{d}{_{0}^{Z''}} \tensor*[]{e}{_{0}}$. 
	Therefore, by \cref{lem:inflationcompletion}, there is an idempotent morphism $e''_{\bullet} \colon \langle Z''_{\bullet}, \delta'' \rangle \to \langle Z''_{\bullet}, \delta'' \rangle$ with $e''_{0} = \tensor*[]{e}{_{0}}$, $e''_{1} = e'_{1} \oplus 0 \oplus 0$ as well as an $\BF$-extension $\tilde{\rho} \in \BF((\tensor*[]{X}{_{0}}, \tensor*[]{e}{_{0}}),(Z''_{n}, e''_{n+1}))$ with underlying $\BE$-extension $\rho = \delta''$. We obtain a $\ft$-distinguished $n$-exangle $\langle (Z''_{\bullet}, e''_{\bullet}), \tilde{\rho} \rangle $. Then the $\ft$-inflation of this $n$-exangle is given by the morphism $\tensor*[]{\tilde{d}}{_{0}^{(Z'',e'')}} \colon (\tensor*[]{X}{_{0}}, \tensor*[]{e}{_{0}}) \to (\tensor*[]{Y}{_{1}} \oplus C' \oplus C, e'_{1} \oplus 0 \oplus 0)$ satisfying
   \[
   \tensor*[]{d}{_{0}^{(Z'',e'')}}
   		= e''_1 \tensor*[]{d}{_{0}^{Z''}} e''_0
        = (e'_{1} \oplus 0 \oplus 0)\tensor*[]{d}{_{0}^{Z''}} \tensor*[]{e}{_{0}}
		= \begin{bsmallmatrix}
			   e'_1 g f \tensor*[]{e}{_{0}} &\; 0 &\;  0
		   \end{bsmallmatrix}^\top
		= \begin{bsmallmatrix}
			   g f &\; 0 &\; 0
		   \end{bsmallmatrix}^\top.
    \]
   As $\tilde{s} \colon (\tensor*[]{Y}{_{1}} \oplus C' \oplus C, e'_{1} \oplus 0 \oplus 0) \to (\tensor*[]{Y}{_{1}}, e'_{1})$ with 
   $s 		= \begin{bsmallmatrix}
		e'_1 & 0 & 0
		\end{bsmallmatrix}
	$ 
   is an isomorphism in $\wt{\CC}$, 
   the complex 
   \[
   \wt{X}''_{\bullet}\colon\hspace{1cm}
   \begin{tikzcd}[column sep = large]
   (\tensor*[]{X}{_{0}},\tensor*[]{e}{_{0}})
   		\arrow{r}{\tensor*[]{\tilde{d}}{_{0}^{\wt{X}''}}}
   & (\tensor*[]{Y}{_{1}},e'_{1})
   		\arrow{r}{\tensor*[]{\tilde{d}}{_{1}^{\wt{X}''}}}
   & (Z''_{2},e''_{2})
   		\arrow{r}{\tensor*[]{\tilde{d}}{_{2}^{(Z'', e'')}}}
   & \cdots
   		\arrow{r}{\tensor*[]{\tilde{d}}{_{n}^{(Z'', e'')}}}
   & (Z''_{n+1},e''_{n+1})
   \end{tikzcd}
   \]
   with $\ft$-inflation $\tensor*[]{\tilde{d}}{_{0}^{\wt{X}''}} \deff \tilde{s} \tensor*[]{\tilde{d}}{_{0}^{(Z'',e'')}} = \tilde{g}\tilde{f}$ and $\tensor*[]{\tilde{d}}{_{1}^{\wt{X}''}} = \tensor*[]{\tilde{d}}{_{1}^{({Z''},e'')}} \tilde{s}^{-1}$ forms part of the $\ft$-distinguished $n$-exangle $\lan \wt{X}''_{\bullet}, \tilde{\rho} \ran$  by \cite[Cor.\ 2.26(2)]{HerschendLiuNakaoka-n-exangulated-categories-I-definitions-and-fundamental-properties}. 
\end{proof}

\subsection{The axiom (EA2) for \texorpdfstring{$(\smash{\wt{\CC}},\BF,\ft)$}{(C,F,t)}}
\label{subsec:EA2}

The goal of this subsection is to show that axiom \ref{EA2} holds for the triplet $(\wt{\CC},\BF,\ft)$. Again, by dualising one can deduce that axiom \ref{EA2op}$\tensor*[]{}{^{\op}}$ also holds. 
We need two key technical lemmas first.

\begin{lem} \label{lem:GoodLiftinC} 
   Suppose that:
   \begin{enumerate}[label=\textup{(\roman*)}]
      \item $\tilde{\delta} \in \BF((\tensor*[]{X}{_{n+1}}, \tensor*[]{e}{_{n+1}}), (\tensor*[]{X}{_{0}}, \tensor*[]{e}{_{0}}))$ is an $\BF$-extension; 
      \item $\tilde{c} \colon (\tensor*[]{Y}{_{n+1}}, e'_{n+1}) \to (\tensor*[]{X}{_{n+1}}, \tensor*[]{e}{_{n+1}})$ is a morphism in $\wt{\CC}$ for some $(\tensor*[]{Y}{_{n+1}}, e'_{n+1}) \in \wt{\CC}$; 
      \item $\langle \tensor*[]{X}{_{\bullet}}, \delta \rangle$ and $\langle \tensor*[]{Y}{_{\bullet}}, \tensor*[]{c}{^{\BE}} \delta \rangle$ are $\fs$-distinguished $n$-exangles with $\tensor*[]{Y}{_{0}} = \tensor*[]{X}{_{0}}$; 
      \item \label{item:LiftOfeBulletInGoodLiftCube}
      $\tensor*[]{e}{_{\bullet}} \colon \langle \tensor*[]{X}{_{\bullet}}, \delta \rangle \to \langle \tensor*[]{X}{_{\bullet}}, \delta \rangle$ is an idempotent morphism  
      lifting $(\tensor*[]{e}{_{0}}, \tensor*[]{e}{_{n+1}}) \colon \delta \to \delta$, 
      such that $\id{\tensor*[]{X}{_{\bullet}}} - \tensor*[]{e}{_{\bullet}}$ is null homotopic; and 
      \item \label{item:GoodLiftInCNullhomotopyY}
      $e'_{\bullet} \colon \langle \tensor*[]{Y}{_{\bullet}}, \tensor*[]{c}{^{\BE}} \delta \rangle \to \langle \tensor*[]{Y}{_{\bullet}}, \tensor*[]{c}{^{\BE}} \delta \rangle$ is an idempotent morphism 
      lifting $(\tensor*[]{e}{_{0}}, e'_{n+1}) \colon \tensor*[]{c}{^{\BE}} \delta \to \tensor*[]{c}{^{\BE}} \delta$, 
      such that $\id{\tensor*[]{Y}{_{\bullet}}} - e'_{\bullet}$ is null homotopic. 
   \end{enumerate}
   Then a good lift $\tensor*[]{g}{_{\bullet}} \colon \langle \tensor*[]{Y}{_{\bullet}}, \tensor*[]{c}{^{\BE}} \delta \rangle \to \langle \tensor*[]{X}{_{\bullet}}, \delta \rangle$ of the morphism $(\id{\tensor*[]{X}{_{0}}}, c) \colon \tensor*[]{c}{^{\BE}} \delta \to \delta$ of $\BE$-extensions exists, so that 
	\begin{equation}
	\label{fig:GoodLift}
      \hspace{-0.4cm}
      \begin{tikzcd}[column sep = 0.65cm]
	& \tensor*[]{Y}{_{0}} 
		\arrow[dd, near start, equal] 
		\arrow[rr,"{\tensor*[]{d}{_{0}^{Y}}}"] 
		\arrow[pos=0.54]{ld}[swap]{\tensor*[]{e}{_{0}}}
	&  
	& \tensor*[]{Y}{_{1}} 
		\arrow[dd, pos=0.1, "{\tensor*[]{g}{_{1}}}", dotted] 
		\arrow[rr,"{\tensor*[]{d}{_{1}^{Y}}}"] 
		\arrow[ld, pos=0.59, "{e'_{1}}"' ]
	& 
	& \cdots 
		\arrow[rr,"{\tensor*[]{d}{_{n-1}^{Y}}}"] 
	& 
	& \tensor*[]{Y}{_{n}} 
		\arrow[ld, pos=0.59, "{e'_{n}}"'] 
		\arrow[dd, pos=0.1, "{\tensor*[]{g}{_{n}}}", dotted] 
		\arrow[rr,"{\tensor*[]{d}{_{n}^{Y}}}"] 
	& 
	& \tensor*[]{Y}{_{n+1}} 
		\arrow[ld, pos=0.59, "{e'_{n+1}}"'] 
		\arrow[dd, near start, "c"] \\
    \tensor*[]{Y}{_{0}} 
    		\arrow[rr,"{\tensor*[]{d}{_{0}^{Y}}}", pos=0.8] 
		\arrow[dd, near start, equal] 
    &                                                     
    & \tensor*[]{Y}{_{1}} 
    		\arrow[dd, pos=0.2, "{\tensor*[]{g}{_{1}}}", dotted] 
		\arrow[rr,"{\tensor*[]{d}{_{1}^{Y}}}", pos=0.8] 
    &                                                                     
    & \cdots 
    		\arrow[rr,"{\tensor*[]{d}{_{n-1}^{Y}}}"] 
    &                   
    & \tensor*[]{Y}{_{n}} 
    		\arrow[dd, near start, "{\tensor*[]{g}{_{n}}}", dotted] 
		\arrow[rr,"{\tensor*[]{d}{_{n}^{Y}}}", pos=0.8] 
    &                                                                     
    & \tensor*[]{Y}{_{n+1}} 
    		\arrow[dd, near start, "c"] 
    &  \\
    & \tensor*[]{X}{_{0}} 
    		\arrow[rr, pos=0.2,"{\tensor*[]{d}{_{0}^{X}}}"] 
		\arrow[ld, pos=0.3, "{\tensor*[]{e}{_{0}}}"]                    
    &                                          
    & \tensor*[]{X}{_{1}} 
    		\arrow[rr,"{\tensor*[]{d}{_{1}^{X}}}"] 
		\arrow[ld,pos=0.3, "{\tensor*[]{e}{_{1}}}"]                            
    &                   
    & \cdots 
    		\arrow[rr, pos=0.2,"{\tensor*[]{d}{_{n-1}^{X}}}"] 
    &                                          
    & \tensor*[]{X}{_{n}} 
    		\arrow[ld,pos=0.3, "{\tensor*[]{e}{_{n}}}"] 
		\arrow[rr, pos=0.2,"{\tensor*[]{d}{_{n}^{X}}}"]                          
    &                    
    & \tensor*[]{X}{_{n+1}} 
    		\arrow[ld, pos=0.3,"{\tensor*[]{e}{_{n+1}}}"]   \\
	\tensor*[]{X}{_{0}} 
		\arrow[rr,"{\tensor*[]{d}{_{0}^{X}}}"']                   
	&                                                     
	& \tensor*[]{X}{_{1}} 
		\arrow[rr,"{\tensor*[]{d}{_{1}^{X}}}"']                           
	&                                                                     
	& \cdots 
		\arrow[rr,"{\tensor*[]{d}{_{n-1}^{X}}}"'] 
	&                   
	& \tensor*[]{X}{_{n}} 
		\arrow[rr,"{\tensor*[]{d}{_{n}^{X}}}"']                           
	&                                                                     
	& \tensor*[]{X}{_{n+1}}            
	&                                          
      \end{tikzcd}
\end{equation}
   is commutative in $\CC$. 
   In particular, we have $\tensor*[]{g}{_{\bullet}} e'_{\bullet} = \tensor*[]{e}{_{\bullet}} \tensor*[]{g}{_{\bullet}} $ as morphisms 
   $\lan \tensor*[]{Y}{_{\bullet}},\tensor*[]{c}{^{\BE}}\delta \ran \to \lan \tensor*[]{X}{_{\bullet}} ,\delta \ran$. 
\end{lem}

\begin{rem}
Notice that $(\tensor*[]{e}{_{0}}\tensor*[]{)}{_{\BE}} \delta = \delta$ and $ c e'_{n+1}= c$ imply 
\begin{equation}\label{eqn:cdelta-is-an-F-extension}
(\tensor*[]{e}{_{0}}\tensor*[]{)}{_{\BE}} (\tensor*[]{c}{^{\BE}} \delta )= \tensor*[]{c}{^{\BE}} \delta = (e'_{n+1}\tensor*[]{)}{^{\BE}} (\tensor*[]{c}{^{\BE}} \delta).
\end{equation}
Therefore, $(\tensor*[]{e}{_{0}}, e'_{n+1}) \colon \tensor*[]{c}{^{\BE}} \delta \to \tensor*[]{c}{^{\BE}} \delta$ is indeed a morphism of $\BE$-extensions and condition \ref{item:GoodLiftInCNullhomotopyY} makes sense. 
Condition \ref{item:LiftOfeBulletInGoodLiftCube} makes sense due to \cref{rem:comments-on-BF}\ref{item:idempotents-give-morphism}.
\end{rem}

\begin{proof}[Proof of \cref{lem:GoodLiftinC}]
   Since $(\id{\tensor*[]{X}{_{0}}}, c) \colon \tensor*[]{c}{^{\BE}} \delta \to \delta$ is a morphism of $\BE$-extensions, 
   it admits a good lift 
   $g'_{\bullet}
   	=(g'_{0},\ldots,g'_{n+1}) 
	=(\id{\tensor*[]{X}{_{0}}}, g'_{1}, \dots, g'_{n}, c) \colon \langle \tensor*[]{Y}{_{\bullet}}, \tensor*[]{c}{^{\BE}} \delta \rangle \to \langle \tensor*[]{X}{_{\bullet}}, \delta \rangle$ 
	using axiom \ref{EA2} for the $n$-exangulated category $(\CC,\BE,\fs)$. 
	Define $\tensor*[]{g}{_{i}} \deff \tensor*[]{e}{_{i}} g'_{i} e'_{i} + (\id{\tensor*[]{X}{_{i}}}-\tensor*[]{e}{_{i}})g'_{i}(\id{\tensor*[]{Y}{_{i}}}-e'_{i})$ for $0\leq i \leq n+1$.
	Note that $\tensor*[]{e}{_{0}} = e'_{0}$ and $\tensor*[]{e}{_{n+1}}c = c = ce'_{n+1}$ by assumption. 
	For $i=0$, we have 
	$\tensor*[]{g}{_{0}} 
		= \tensor*[]{e}{_{0}} \id{\tensor*[]{X}{_{0}}} e'_{0} + (\id{\tensor*[]{X}{_{0}}} - \tensor*[]{e}{_{0}}) \id{\tensor*[]{X}{_{0}}} (\id{\tensor*[]{X}{_{0}}} - e'_{0})
		= \id{\tensor*[]{X}{_{0}}}$. 
	On the other hand, for $i=n+1$ 
	we have 
	$g_{n+1}
		= \tensor*[]{e}{_{n+1}} c e'_{n+1} + (\id{\tensor*[]{X}{_{n+1}}} - \tensor*[]{e}{_{n+1}}) c (\id{\tensor*[]{X}{_{n+1}}} - e'_{n+1})
		=c$. 
   Therefore, the morphism  $\tensor*[]{g}{_{\bullet}} =  \tensor*[]{e}{_{\bullet}} g'_{\bullet}e'_{\bullet} + (\id{\tensor*[]{X}{_{\bullet}}} - \tensor*[]{e}{_{\bullet}}) g'_{\bullet}(\id{\tensor*[]{Y}{_{\bullet}}} - e'_{\bullet}) \colon \tensor*[]{Y}{_{\bullet}} \to \tensor*[]{X}{_{\bullet}}$ is of the form $(\id{\tensor*[]{X}{_{0}}}, \tensor*[]{g}{_{1}}, \dots, \tensor*[]{g}{_{n}}, c)$.

   The squares on the top and bottom faces in \eqref{fig:GoodLift} commute as $e'_{\bullet} \colon \tensor*[]{Y}{_{\bullet}} \to \tensor*[]{Y}{_{\bullet}}$ and $\tensor*[]{e}{_{\bullet}} \colon \tensor*[]{X}{_{\bullet}} \to \tensor*[]{X}{_{\bullet}}$, respectively, are morphisms of complexes. 
   The squares on the front and back faces in \eqref{fig:GoodLift} commute because $\tensor*[]{g}{_{\bullet}}$ is the sum of morphisms of complexes from $\tensor*[]{Y}{_{\bullet}}$ to $\tensor*[]{X}{_{\bullet}}$.  
   This also implies $\tensor*[]{g}{_{\bullet}} = (\id{\tensor*[]{X}{_{0}}}, \tensor*[]{g}{_{1}}, \dots, \tensor*[]{g}{_{n}}, c)$ is a lift of $(\id{\tensor*[]{X}{_{0}}}, c)$. 
   Of the remaining squares, the leftmost clearly commutes and the rightmost commutes as $c \colon (\tensor*[]{Y}{_{n+1}}, e'_{n+1}) \to (\tensor*[]{X}{_{n+1}}, \tensor*[]{e}{_{n+1}})$ is a morphism in $\wt{\CC}$. 
   For $1\leq i \leq n$, we have 
   \begin{align*}
   	\tensor*[]{g}{_{i}} e'_{i} 
		&= \tensor*[]{e}{_{i}} g'_{i} e'_{i} e'_{i}  + (\id{\tensor*[]{X}{_{i}}}-\tensor*[]{e}{_{i}})g'_{i}(\id{\tensor*[]{Y}{_{i}}}-e'_{i})e'_{i} \\
		&= \tensor*[]{e}{_{i}} g'_{i} e'_{i} &&\text{as } (e'_{i})^2 = e'_{i}\\
		&= \tensor*[]{e}{_{i}}\tensor*[]{e}{_{i}} g'_{i} e'_{i}  + \tensor*[]{e}{_{i}}(\id{\tensor*[]{X}{_{i}}}-\tensor*[]{e}{_{i}})g'_{i}(\id{\tensor*[]{Y}{_{i}}}-e'_{i}) &&\text{as } (\tensor*[]{e}{_{i}})^2 = \tensor*[]{e}{_{i}}\\
		&= \tensor*[]{e}{_{i}} \tensor*[]{g}{_{i}}.
   \end{align*}
   Therefore, diagram \eqref{fig:GoodLift} commutes and, further, the last assertion follows.
    
	It remains to show that $\tensor*[]{g}{_{\bullet}}$ is a good lift of $(\id{\tensor*[]{X}{_{0}}},c)\colon \tensor*[]{c}{^{\BE}} \delta \to \delta$. 
  Recall that $\id{\tensor*[]{X}{_{\bullet}}} - \tensor*[]{e}{_{\bullet}}$ and $\id{\tensor*[]{Y}{_{\bullet}}} - e'_{\bullet}$ are both null homotopic by assumption, and so 
  $g'_{\bullet}- \tensor*[]{g}{_{\bullet}} = (\id{\tensor*[]{X}{_{\bullet}}} - \tensor*[]{e}{_{\bullet}}) g'_{\bullet}e'_{\bullet} + \tensor*[]{e}{_{\bullet}} g'_{\bullet}(\id{\tensor*[]{Y}{_{\bullet}}} - e'_{\bullet})$ is also null homotopic. 
  Then it follows from \cite[Rem.\ 2.33(1)]{HerschendLiuNakaoka-n-exangulated-categories-I-definitions-and-fundamental-properties} that $\tensor*[]{g}{_{\bullet}}$ is a good lift of $(\id{\tensor*[]{X}{_{0}}},c)$ since $g'_{\bullet}$ is. 
\end{proof}

The next result allows us to define a good lift in $\wt{\CC}$ from the one we created in \cref{lem:GoodLiftinC}.

\begin{lem}\label{lem:GoodLiftinCtilde}
   In the setup of  \textup{\cref{lem:GoodLiftinC}}, the morphism 
   ${\tensor*[]{\tilde{h}}{_{\bullet}} \colon (\tensor*[]{Y}{_{\bullet}}, e'_{\bullet}) \to (\tensor*[]{X}{_{\bullet}}, \tensor*[]{e}{_{\bullet}})}$ 
   with underlying morphism 
   $\tensor*[]{h}{_{\bullet}} = \tensor*[]{e}{_{\bullet}} \tensor*[]{g}{_{\bullet}} e'_{\bullet}$ is a good lift of the morphism $(\wid{(\tensor*[]{X}{_{0}}, \tensor*[]{e}{_{0}})}, \tilde{c}) \colon \tilde{c}^\BF \tilde{\delta} \to \tilde{\delta}$ of $\BF$-extensions. 
\end{lem}

\begin{proof}
	From \eqref{eqn:cdelta-is-an-F-extension}, we see that $\tilde{c}^\BF \tilde{\delta}\in \BF((\tensor*[]{Y}{_{n+1}},e'_{n+1}),(\tensor*[]{X}{_{0}},\tensor*[]{e}{_{0}}))$ is indeed an $\BF$-extension and $(\wid{(\tensor*[]{X}{_{0}}, \tensor*[]{e}{_{0}})}, \tilde{c}) \colon \tilde{c}^\BF \tilde{\delta} \to \tilde{\delta}$ a morphism of $\BF$-extensions. 
	Using $\tensor*[]{h}{_{0}} = \tensor*[]{e}{_{0}} \tensor*[]{g}{_{0}} e'_{0} = \tensor*[]{e}{_{0}} = \id{(\tensor*[]{X}{_{0}}, \tensor*[]{e}{_{0}})}$ and $\tensor*[]{h}{_{n+1}} = \tensor*[]{e}{_{n+1}} c e'_{n+1} = c$,  as well as the commutativity of \eqref{fig:GoodLift}, we see that $\tensor*[]{\tilde{h}}{_{\bullet}}$ is a morphism $\langle (\tensor*[]{Y}{_{\bullet}}, e'_{\bullet}), \tilde{c}^\BF \tilde{\delta} \rangle \to \langle (\tensor*[]{X}{_{\bullet}}, \tensor*[]{e}{_{\bullet}}), \tilde{\delta} \rangle$ of $\ft$-distinguished $n$-exangles, lifting $(\wid{(\tensor*[]{X}{_{0}}, \tensor*[]{e}{_{0}})}, \tilde{c})$.

	Recall from \cref{def:mapping-cone} that ${\tensor*[]{M}{^{\CC}_{g}}}_{\bullet}$ denotes the mapping cone of $\tensor*[]{g}{_{\bullet}} \colon \tensor*[]{Y}{_{\bullet}} \to \tensor*[]{X}{_{\bullet}}$ in $\CC$, and that $\langle {\tensor*[]{M}{^{\CC}_{g}}}_{\bullet}, (\tensor*[]{d}{^{Y}_{0}}\tensor*[]{)}{_{\BE}} \delta \rangle$ is $\fs$-distinguished as $\tensor*[]{g}{_{\bullet}} \colon \lan \tensor*[]{Y}{_{\bullet}}, \tensor*[]{c}{^{\BE}}\delta\ran \to \lan \tensor*[]{X}{_{\bullet}},\delta\ran $ is a good lift of $(\id{\tensor*[]{X}{_{0}}},c)\colon \tensor*[]{c}{^{\BE}}\delta \to \delta$. 
	Using the commutativity of \eqref{fig:GoodLift}, that $\tensor*[]{e}{_{\bullet}} \colon \tensor*[]{X}{_{\bullet}} \to \tensor*[]{X}{_{\bullet}}$ and $e'_{\bullet} \colon \tensor*[]{Y}{_{\bullet}} \to \tensor*[]{Y}{_{\bullet}}$ are morphisms of complexes, and that $\tensor*[]{e}{_{n+1}}c = c = ce'_{n+1}$, one can verify that the diagram 
\begin{equation}\label{fig:IdempotentOnCone}
      \hspace{-0.4cm}
      \begin{tikzcd}[ampersand replacement=\&, column sep = 1cm, row sep = 3.4em, scale cd=0.9]
   	   {\tensor*[]{Y}{_{1}}} \&[-.6em] {\tensor*[]{Y}{_{2}} \oplus \tensor*[]{X}{_{1}}} \&[+.2em] {\tensor*[]{Y}{_{3}} \oplus \tensor*[]{X}{_{2}}} \&[+.2em] \cdots \&[+.8em] {\tensor*[]{Y}{_{n+1}} \oplus \tensor*[]{X}{_{n}}} \&[+.2em] {\tensor*[]{X}{_{n+1}}} \& {} \\
      	{\tensor*[]{Y}{_{1}}} \&[+1em] {\tensor*[]{Y}{_{2}} \oplus \tensor*[]{X}{_{1}}} \& {\tensor*[]{Y}{_{3}} \oplus \tensor*[]{X}{_{2}}} \& \cdots \& {\tensor*[]{Y}{_{n+1}} \oplus \tensor*[]{X}{_{n}}} \& {\tensor*[]{X}{_{n+1}}} \& {}
      	\arrow["{d_0^{M_g^\CC}}", from=1-1, to=1-2]
      	\arrow["{d_1^{M_g^\CC}}", from=1-2, to=1-3]
      	\arrow["{d_2^{M_g^\CC}}", from=1-3, to=1-4]
      	\arrow["{d_{n-1}^{M_g^\CC}}", from=1-4, to=1-5]
      	\arrow["{d_{n}^{M_g^\CC}}", from=1-5, to=1-6]
      	\arrow["{e'_{1}}"', from=1-1, to=2-1]
      	\arrow["{\begin{bsmallmatrix}e'_{2} & 0 \\ 0 & \tensor*[]{e}{_{1}}\end{bsmallmatrix}}", from=1-2, to=2-2]
      	\arrow["{\begin{bsmallmatrix}e'_{3} & 0 \\ 0 & \tensor*[]{e}{_{2}}\end{bsmallmatrix}}", from=1-3, to=2-3]
      	\arrow["{\begin{bsmallmatrix}e'_{n+1} & 0 \\ 0 & \tensor*[]{e}{_{n}}\end{bsmallmatrix}}"', from=1-5, to=2-5]
      	\arrow["{\tensor*[]{e}{_{n+1}}}"', from=1-6, to=2-6]
      	\arrow[pos=0.4,"{d_0^{M_g^\CC}}"', from=2-1, to=2-2]
      	\arrow["{d_1^{M_g^\CC}}"', from=2-2, to=2-3]
      	\arrow["{d_2^{M_g^\CC}}"', from=2-3, to=2-4]
      	\arrow["{d_{n-1}^{M_g^\CC}}"', from=2-4, to=2-5]
      	\arrow["{d_{n}^{M_g^\CC}}"', from=2-5, to=2-6]
	      	\arrow[dashed,"{(\tensor*[]{d}{_{0}^{Y}}\tensor*[]{)}{_{\BE}}\delta}", from=1-6, to=1-7]
		     \arrow[dashed,"{(\tensor*[]{d}{_{0}^{Y}}\tensor*[]{)}{_{\BE}}\delta}"', from=2-6, to=2-7]
      \end{tikzcd}
      \end{equation}
      commutes. Thus, the vertical morphisms form an idempotent morphism ${e''_{\bullet} \colon {M_{g}^\CC}_{\bullet} \to {M_{g}^\CC}_{\bullet}}$ of complexes. 
    Furthermore, \eqref{fig:IdempotentOnCone} is a morphism of $\fs$-distinguished $n$-exangles 
as
	\begin{align*}
	(e'_{1}\tensor*[]{)}{_{\BE}} (\tensor*[]{d}{_{0}^{Y}}\tensor*[]{)}{_{\BE}}\delta
		&= (\tensor*[]{d}{_{0}^{Y}}\tensor*[]{)}{_{\BE}} (\tensor*[]{e}{_{0}}\tensor*[]{)}{_{\BE}}\delta&&\text{as \eqref{fig:GoodLift} is commutative}\\
		&= (\tensor*[]{d}{_{0}^{Y}}\tensor*[]{)}{_{\BE}} \delta &&\text{as } \tilde{\delta} \in \BF((\tensor*[]{X}{_{n+1}}, \tensor*[]{e}{_{n+1}}), (\tensor*[]{X}{_{0}}, \tensor*[]{e}{_{0}})) \\
		&= (\tensor*[]{d}{_{0}^{Y}}\tensor*[]{)}{_{\BE}} (\tensor*[]{e}{_{n+1}}\tensor*[]{)}{^{\BE}}\delta &&\text{as } \tilde{\delta} \in \BF((\tensor*[]{X}{_{n+1}}, \tensor*[]{e}{_{n+1}}), (\tensor*[]{X}{_{0}}, \tensor*[]{e}{_{0}})) \\
		&= (\tensor*[]{e}{_{n+1}}\tensor*[]{)}{^{\BE}}(\tensor*[]{d}{_{0}^{Y}}\tensor*[]{)}{_{\BE}}\delta.
	\end{align*}
	This calculation also shows that $\tilde{\rho} \deff (e'_1, (\tensor*[]{d}{_{0}^{Y}}\tensor*[]{)}{_{\BE}} \delta, \tensor*[]{e}{_{n+1}}) \in \BF((\tensor*[]{X}{_{n+1}},\tensor*[]{e}{_{n+1}}),(\tensor*[]{Y}{_{1}},e'_{1}))$. 
	Thus, by definition of $\ft$, we have that 
	$\ft(\tilde{\rho}) = [ ({M^{\CC}_g}_{\bullet}, e''_{\bullet}) ]$, i.e.\ 
	$\lan ({M^{\CC}_g}_{\bullet}, e''_{\bullet}), \tilde{\rho} \ran$ is $\ft$-distinguished.
	
It is straightforward to verify that the object $({M^{\CC}_g}_{\bullet}, e''_{\bullet})$ 
is equal to the mapping cone 
${M^{\wt{\CC}}_{\tilde{h}}}_{\bullet}$ of $\tensor*[]{\tilde{h}}{_{\bullet}}$ 
in $\com{\wt{\CC}}^{\raisebox{0.5pt}{\scalebox{0.6}{$n$}}}$, 
so $\lan {M^{\wt{\CC}}_{\tilde{h}}}_{\bullet}, \tilde{\rho} \ran$ is $\ft$-distinguished. 
Lastly, we note that $(\tensor*[]{\tilde{d}}{_{0}^{(Y,e')}}\tensor*[]{)}{_{\BF}}(\tilde{\delta}) = \tilde{\rho}$ because
\[ 
(\tensor*[]{d}{_{0}^{(Y,e')}}\tensor*[]{)}{_{\BE}} \delta 
    = (\tensor*[]{d}{_{0}^{Y}}\tensor*[]{e}{_{0}}\tensor*[]{)}{_{\BE}} \delta 
    = (\tensor*[]{d}{_{0}^{Y}}\tensor*[]{)}{_{\BE}}(\tensor*[]{e}{_{0}}\tensor*[]{)}{_{\BE}} \delta 
    = (\tensor*[]{d}{_{0}^{Y}}\tensor*[]{)}{_{\BE}} \delta\
    = \rho.
\] 
Hence,
$\lan 
{M^{\wt{\CC}}_{\tilde{h}}}_{\bullet}, 
    (\tensor*[]{\tilde{d}}{_{0}^{(Y,e')}}\tensor*[]{)}{_{\BF}} \tilde{\delta} 
\ran$ 
is a $\ft$-distinguished $n$-exangle.
\end{proof}

We are in position to prove axiom \ref{EA2} for $(\wt{\CC},\BF,\ft)$. Axiom \ref{EA2op}$\tensor*[]{}{^{\op}}$ can be shown dually. 

\begin{prop}[Axiom~\ref{EA2} for $(\wt{\CC},\BF,\ft)$]
\label{prop:AxiomEA2}
	Let  $\tilde{\delta} \in \BF(\tensor*[]{\wt{X}}{_{n+1}}, \tensor*[]{\wt{X}}{_{0}})$ be an $\BF$-extension 
	and suppose $\tilde{c} \colon \tensor*[]{\wt{Y}}{_{n+1}} \to \tensor*[]{\wt{X}}{_{n+1}}$ is a morphism in $\wt{\CC}$. 
	Suppose $\langle \tensor*[]{\wt{X}}{_{\bullet}}, \tilde{\delta} \rangle$ and $\langle \tensor*[]{\wt{Y}}{_{\bullet}}, \tilde{c}^\BF \tilde{\delta} \rangle$ are $\ft$-distinguished $n$-exangles. 
	Then $(\wid{\tensor*[]{\wt{X}}{_{0}}}, \tilde{c})$ has a good lift $\tensor*[]{\tilde{h}}{_{\bullet}} \colon \tensor*[]{\wt{Y}}{_{\bullet}} \to \tensor*[]{\wt{X}}{_{\bullet}}$.
\end{prop}
\begin{proof}
    Notice that the underlying $\BE$-extension of $\tensor*[]{\tilde{c}}{^{\BF}}\tilde{\delta}$ is $\tensor*[]{c}{^{\BE}} \delta$.
	By definition of $\ft$ and \cref{rem:independence-of-t}, there are $\fs$-distinguished $n$-exangles 
	$\langle X'_{\bullet}, \delta \rangle$ and $\langle Y'_{\bullet}, \tensor*[]{c}{^{\BE}} \delta \rangle$ 
	and idempotent morphisms 
    $\tensor*[]{e}{_{\bullet}} \colon \langle X'_{\bullet}, \delta \rangle\to \langle X'_{\bullet}, \delta \rangle$ 
	and 
	$e'_{\bullet} \colon \langle Y'_{\bullet}, \tensor*[]{c}{^{\BE}}\delta \rangle\to \langle Y'_{\bullet}, \tensor*[]{c}{^{\BE}}\delta \rangle$, 
	such that  
	$\ft(\tilde{\delta}) = [(X'_{\bullet}, \tensor*[]{e}{_{\bullet}})]$ and $\ft(\tilde{c}^\BF\tilde{\delta}) = [(Y'_{\bullet}, e'_{\bullet})]$, 
	and so that 
	$\id{X'_{\bullet}} - \tensor*[]{e}{_{\bullet}}$ and  $\id{Y'_{\bullet}} - e'_{\bullet}$
	are null homotopic in $\com{\CC}^{\raisebox{0.5pt}{\scalebox{0.6}{$n$}}}$. 
	We note that since 
	$[\tensor*[]{\wt{Y}}{_{\bullet}}] = \ft(\tensor*[]{\tilde{c}}{^{\BF}}\tilde{\delta}) = [(Y'_{\bullet},e'_{\bullet})]$ 
	and 
	$[\tensor*[]{\wt{X}}{_{\bullet}}] = \ft(\tilde{\delta}) = [(X'_{\bullet},\tensor*[]{e}{_{\bullet}})]$, we have that 
	$(Y'_0,e'_0)
	= \tensor*[]{\wt{Y}}{_{0}}
	= \tensor*[]{\wt{X}}{_{0}}
	= (X'_0,\tensor*[]{e}{_{0}})$
	and, in particular, that $\tensor*[]{e}{_{0}} =e'_0$. 
	Moreover, it follows that all the hypotheses of \cref{lem:GoodLiftinC} are satisfied. 
	
	Therefore, by \cref{lem:GoodLiftinCtilde}, the morphism $(\wid{\tensor*[]{\wt{X}}{_{0}}}, \tilde{c}) \colon \tilde{c}^\BF \tilde{\delta} \to \tilde{\delta}$ of $\BF$-extensions has a good lift 
	$\tilde{h}'_{\bullet} \colon (Y'_{\bullet}, e'_{\bullet}) \to (X'_{\bullet},\tensor*[]{e}{_{\bullet}})$. 
	Since $[(X'_{\bullet}, \tensor*[]{e}{_{\bullet}})] 
		= [\tensor*[]{\wt{X}}{_{\bullet}}]$ 
	and $ [(Y'_{\bullet}, e'_{\bullet})] 
		= [\tensor*[]{\wt{Y}}{_{\bullet}}]$, 
	there is homotopy equivalence 
	$\tensor*[]{\tilde{a}}{_{\bullet}}\colon (X'_{\bullet}, \tensor*[]{e}{_{\bullet}}) \to \tensor*[]{\wt{X}}{_{\bullet}}$ 
	in $\smash{\com{\wt{\CC}}_{(\tensor*[]{\wt{X}}{_{0}}, \tensor*[]{\wt{X}}{_{n+1}})}^{\raisebox{0.5pt}{\scalebox{0.6}{$n$}}}}$ 
	and a homotopy equivalence
	$\tensor*[]{\tilde{b}}{_{\bullet}}\colon  \tensor*[]{\wt{Y}}{_{\bullet}}\to (Y'_{\bullet}, e'_{\bullet})$ 
	in $\com{\wt{\CC}}_{(\tensor*[]{\wt{Y}}{_{0}}, \tensor*[]{\wt{Y}}{_{n+1}})}^{\raisebox{0.5pt}{\scalebox{0.6}{$n$}}}$. 
	By \cite[Cor.\ 2.31]{HerschendLiuNakaoka-n-exangulated-categories-I-definitions-and-fundamental-properties}, the composite $\tensor*[]{\tilde{a}}{_{\bullet}} \tilde{h}'_{\bullet} \tensor*[]{\tilde{b}}{_{\bullet}} \colon \tensor*[]{\wt{Y}}{_{\bullet}} \to \tensor*[]{\wt{X}}{_{\bullet}}$ is then also a good lift of $(\wid{\tensor*[]{\wt{X}}{_{0}}}, \tilde{c})$. 
\end{proof}

\subsection{Main results}
\label{subsec:main-results}

In this subsection we present our main results regarding the idempotent completion and an $n$-exangulated structure we can impose on it. 

\begin{defn}
\label{def:idempotent-complete-n-exangulated-category}
We call an $n$-exangulated category $(\CC, \BE, \fs)$ (resp.\ \emph{weakly}) \emph{idempotent complete} if the underlying additive category $\CC$ is (resp.\ weakly) idempotent complete. 
\end{defn}

In \cite[Prop.\ 2.2]{BrustleHassounShahTattar-the-jordan-holder-property-for-stratifying-systems-in-extriangulated-categories} a characterisation of weakly idempotent complete extriangulated categories is given. 
Next we note that the first part of \cref{thmx:mainthm-ctilde} from \cref{sec:introduction} summarises our work from Subsections~\ref{subsec:def-of-F}--\ref{subsec:EA2}. 

\begin{thm}
\label{thm:mainthm-ctilde}
Let $(\CC, \BE, \fs)$ be an $n$-exangulated category.
Then the triplet $(\wt{\CC}, \BF, \ft)$ is an idempotent complete $n$-exangulated category.
\end{thm}

\begin{proof}
This 
follows from 
	\cref{prop:karoubi-envelope-idempotent-complete,prop:exact-realisation,prop:AxiomEA1,prop:AxiomEA2}, 
	and the duals of the latter two.
\end{proof}

And \cref{corx:Krull-Schmidt} from Section~\ref{sec:introduction} is a nice consequence of this.

\begin{cor}\label{cor:Krull-Schmidt}
        Let $(\CC, \BE, \fs)$ be an $n$-exangulated category, such that each object in $\CC$ has a semi-perfect endomorphism ring. 
        Then the idempotent completion $(\wt{\CC}, \BF, \ft)$ is a Krull-Schmidt $n$-exangulated category.
\end{cor}
\begin{proof}%
    By \cref{thm:mainthm-ctilde}, the idempotent completion $(\wt{\CC}, \BF, \ft)$ is an idempotent complete $n$-exangulated category. 
    By \cite[Thm.\ A.1]{ChenYeZhang-Algebras-of-derived-dimension-zero} (or \cite[Cor.\ 4.4]{Krause-KS-cats-and-projective-covers}), 
    it is enough to show that endomorphism rings of objects in $\wt{\CC}$ are semi-perfect rings. 
Let $(X, e)$ be an object in $\wt{\CC}$. 
We have that 
$\tensor*[]{\End}{_{\wt{\CC}}}((X,\id{X})) \iso \tensor*[]{\End}{_{\CC}}(X)$ is semi-perfect since $\tensor*[]{\SI}{_{\CC}}$ is fully faithful (see \cref{prop:karoubi-envelope-idempotent-complete}). 
By \cref{rem:direct-sum-decomposition-of-X-id-wrt-idempotent}, we have that $(X, \id{X}) \iso (X, e) \oplus (X, \id{X} -e)$. 
In particular, we see that $\tensor*[]{\End}{_{\wt{\CC}}} ((X,e))$ is an idempotent subring of the semi-perfect ring $\tensor*[]{\End}{_{\wt{\CC}}}((X,\id{X}))$. 
    Hence, by Anderson--Fuller \cite[Cor.\ 27.7]{AndersonFuller-rings-and-cats-of-modules}, we have that the endomorphism ring of each object in $\wt{\CC}$ is semi-perfect. 
\end{proof}

We recall that, by \cite[Cor.\ 4.12]{klapproth2023nextension}, an $n$-exangulated category is $n$-exact if and only if its inflations are monomorphisms and its deflations are epimorphisms. 

\begin{cor}\label{cor:idempotent-n-exact}
    Suppose $(\CC, \BE, \fs)$ is $n$-exact. 
    Then $(\wt{\CC}, \BF, \ft)$ is $n$-exact. 
\end{cor}

\begin{proof}
    We use \cite[Cor.\ 4.12]{klapproth2023nextension} and only show that $\ft$-inflations are monomorphisms; showing $\ft$-deflations are epimorphisms is dual. 
    Let $\tilde{f} \colon (\tensor*[]{X}{_{0}}, \tensor*[]{e}{_{0}}) \to (\tensor*[]{X}{_{1}}, \tensor*[]{e}{_{1}})$ be a $\ft$-inflation and suppose 
    there is a morphism 
    $\tilde{g} \colon (\tensor*[]{Y}{_{0}}, e'_0) \to (X_0, e_0)$ in $\wt{\CC}$ with $\tilde{f} \tilde{g} = \wt{0}$. 
    By \Cref{lem:TwistInflation}, there is an 
    $\fs$-inflation 
    $\tensor*[]{d}{^{X'}_{0}} 
        = \begin{bsmallmatrix}
            f &\; f'(\id{\tensor*[]{X}{_{0}}}-\tensor*[]{e}{_{0}})
           \end{bsmallmatrix}^\top
        \colon 
        \tensor*[]{X}{_{0}} 
            \to \tensor*[]{X}{_{1}} \oplus C
   $, 
   which is monic as $(\CC, \BE, \fs)$ is $n$-exact. 
   We have $fg = 0$ as $\tilde{f} \tilde{g} = \wt{0}$, and we also have $f'(\id{\tensor*[]{X}{_{0}}}-\tensor*[]{e}{_{0}}) g = 0$ because the underlying morphism $g$ of $\tilde{g}$ satisfies $g = \tensor*[]{e}{_{0}}g$. 
   Thus, we see that $\tensor*[]{d}{^{X'}_{0}} g = 0$ and this implies $g = 0$ as $\tensor*[]{d}{^{X'}_{0}}$ is monic. Hence, $\tilde{g} = 0$ and we are done. 
\end{proof}

The main aim of this subsection is to establish the relevant $2$-universal property of the inclusion functor $\tensor*[]{\SI}{_{\CC}}\colon \CC \to \wt{\CC}$. 
We will show that $\tensor*[]{\SI}{_{\CC}}$ forms part of an $n$-exangulated functor $(\tensor*[]{\SI}{_{\CC}},\Gamma)\colon(\CC,\BE,\fs)\to(\wt{\CC},\BF,\ft)$, and that this is $2$-universal in an appropriate sense. 
The next lemma is straightforward to check.

\begin{lem}
\label{lem:def-of-Gamma}
   The family of abelian group homomorphisms  
   \begin{align*}
\tensor*[]{\Gamma}{_{(\tensor*[]{X}{_{n+1}},\tensor*[]{X}{_{0}})}} \colon \BE(\tensor*[]{X}{_{n+1}},\tensor*[]{X}{_{0}}) 
    &\longrightarrow \BF(\tensor*[]{\SI}{_{\CC}} (\tensor*[]{X}{_{n+1}}), \tensor*[]{\SI}{_{\CC}}(\tensor*[]{X}{_{0}})) \\ 
\delta
    &\longmapsto (\id{\tensor*[]{X}{_{0}}}, \delta, \id{\tensor*[]{X}{_{n+1}}}),
\end{align*}
	for $\tensor*[]{X}{_{0}}, \tensor*[]{X}{_{n+1}} \in \CC$, defines a natural isomorphism 
$\Gamma \colon \BE(-,-) \overset{\iso}{\Longrightarrow} \BF(\tensor*[]{\SI}{_{\CC}} -, \tensor*[]{\SI}{_{\CC}} -)$.
\end{lem}

\begin{prop}\label{prop:SI-Gamma-is-n-exangulated}
The pair $(\tensor*[]{\SI}{_{\CC}},\Gamma)$ is an $n$-exangulated functor from $(\CC,\BE,\fs)$ to $(\wt{\CC},\BF,\ft)$.
\end{prop}

\begin{proof}
We verify that if $\lan \tensor*[]{X}{_{\bullet}}, \delta \ran$ is an $\fs$-distinguished $n$-exangle, 
then $\lan \tensor*[]{\SI}{_{\CC}}(\tensor*[]{X}{_{\bullet}}), \tensor*[]{\Gamma}{_{(\tensor*[]{X}{_{n+1}},\tensor*[]{X}{_{0}})}}(\delta) \ran$ is $\ft$-dis\-tin\-guish\-ed, where 
$
	\tensor*[]{\Gamma}{_{(\tensor*[]{X}{_{n+1}},\tensor*[]{X}{_{0}})}}(\delta) 
	= (\id{\tensor*[]{X}{_{0}}}, \delta, \id{\tensor*[]{X}{_{n+1}}})
	\in\BF(\tensor*[]{\SI}{_{\CC}}(\tensor*[]{X}{_{n+1}}), \tensor*[]{\SI}{_{\CC}}(\tensor*[]{X}{_{0}}))$. 
We have the idempotent morphism $\id{\tensor*[]{X}{_{\bullet}}}\colon \lan \tensor*[]{X}{_{\bullet}}, \delta \ran \to \lan \tensor*[]{X}{_{\bullet}}, \delta \ran$ that is a lift of $(\id{\tensor*[]{X}{_{0}}},\id{\tensor*[]{X}{_{n+1}}})\colon \delta\to\delta$, 
so from \cref{def:t} we see that $\ft(\tensor*[]{\Gamma}{_{(\tensor*[]{X}{_{n+1}},\tensor*[]{X}{_{0}})}}(\delta)) = [(\tensor*[]{X}{_{\bullet}},\id{\tensor*[]{X}{_{\bullet}}})] = [\tensor*[]{\SI}{_{\CC}}(\tensor*[]{X}{_{\bullet}})]$. 
\end{proof}

We lay out some notation that will be used in the remainder of this section and also in \cref{sec:the-WIC}.

\begin{notn}
\label{item:projection-inclusion-of-X-e-as-summand} 
Let $(X,e)$ be an object in the idempotent completion $\wt{\CC}$ of $\CC$. 
Then $(X,e)$ is a direct summand of $\tensor*[]{\SI}{_{\CA}}(X) = (X,\id{X})$ by \cref{rem:direct-sum-decomposition-of-X-id-wrt-idempotent}. 
By $\ie{e} \deff (\id{X},e,e) \colon (X,e) \to (X,\id{X})$ and $\pe{e} \deff (e,e,\id{X}) \colon (X,\id{X}) \to (X,e)$, 
we denote the canonical inclusion and projection morphisms, respectively.
\end{notn}

Recall that, for an additive category $\CC'$ and a biadditive functor $\BE'\colon(\CC')^{\op}\times\CC'\to \Ab$, 
the $\BE'$-attached complexes and morphisms between them were defined in \cref{def:attached-complexes-n-exangles-and-morphisms}, and together they form an additive category.

\begin{lem}
\label{lem:summand-3}
Let $\tilde{\delta} \in \BF((\tensor*[]{X}{_{n+1}}, \tensor*[]{e}{_{n+1}}), (\tensor*[]{X}{_{0}}, \tensor*[]{e}{_{0}}))$ be an $\BF$-extension. 
Suppose $\fs(\delta) = [\tensor*[]{X}{_{\bullet}}]$
and $\tensor*[]{e}{_{\bullet}} \colon \langle \tensor*[]{X}{_{\bullet}}, \delta \rangle \to \langle \tensor*[]{X}{_{\bullet}}, \delta \rangle$ is an idempotent morphism. 
With 
$e'_{\bullet} \deff \id{\tensor*[]{X}{_{\bullet}}} - \tensor*[]{e}{_{\bullet}}$, 
we have that 
\begin{equation}
\label{eqn:isomorphism-of-t-n-exangles}
\langle \tensor*[]{\SI}{_{\CC}}(\tensor*[]{X}{_{\bullet}}), \tensor*[]{\Gamma}{_{(\tensor*[]{X}{_{n+1}},\tensor*[]{X}{_{0}})}}(\delta) \rangle 
	\iso 
	\langle (\tensor*[]{X}{_{\bullet}}, \tensor*[]{e}{_{\bullet}}), \tilde{\delta} \rangle 
		\oplus \langle (\tensor*[]{X}{_{\bullet}}, e'_{\bullet}), \tensor*[_{(\tensor*[]{X}{_{0}}, e'_{0})}]{\wt{0}}{_{(\tensor*[]{X}{_{n+1}}, e'_{n+1})}}\rangle
\end{equation}
as $\ft$-distinguished $n$-exangles.
\end{lem}

\begin{proof}
Let $\tilde{\rho} \deff \tensor*[]{\Gamma}{_{(\tensor*[]{X}{_{n+1}},\tensor*[]{X}{_{0}})}}(\delta)$. 
First, note that $\langle (\tensor*[]{X}{_{\bullet}}, \tensor*[]{e}{_{\bullet}}), \tilde{\delta} \rangle$ and $\langle (\tensor*[]{X}{_{\bullet}}, e'_{\bullet}), \tensor*[_{(\tensor*[]{X}{_{0}}, e'_{0})}]{\wt{0}}{_{(\tensor*[]{X}{_{n+1}}, e'_{n+1})}}\rangle$ are $\BF$-attached complexes, and 
$\langle \tensor*[]{\SI}{_{\CC}}(\tensor*[]{X}{_{\bullet}}), \tilde{\rho} \rangle$ is 
an $\ft$-distinguished $n$-exangle 
since $(\SF,\Gamma)$ is an $n$-exangulated functor.

Consider the morphisms 
$\ie{\tensor*[]{e}{_{\bullet}}}\colon (\tensor*[]{X}{_{\bullet}},\tensor*[]{e}{_{\bullet}})\to \tensor*[]{\SI}{_{\CC}}(\tensor*[]{X}{_{\bullet}})$
and 
$\pe{\tensor*[]{e}{_{\bullet}}}\colon \tensor*[]{\SI}{_{\CC}}(\tensor*[]{X}{_{\bullet}}) \to (\tensor*[]{X}{_{\bullet}},\tensor*[]{e}{_{\bullet}})$ of complexes 
induced by $\tensor*[]{e}{_{\bullet}}$, 
as well as the corresponding ones $\ie{e'_{\bullet}}$ and $\pe{e'_{\bullet}}$ for $e'_{\bullet}$. 
We claim that there is a biproduct diagram
\begin{equation}\label{eqn:biproductdiagram}
\begin{tikzcd}[column sep=1.5cm]
\lan (\tensor*[]{X}{_{\bullet}},\tensor*[]{e}{_{\bullet}}), \tilde{\delta}\ran 
	\arrow[shift left]{r}{\ie{\tensor*[]{e}{_{\bullet}}}}
&\lan \tensor*[]{\SI}{_{\CC}}(\tensor*[]{X}{_{\bullet}}),\tilde{\rho}\ran 
	\arrow[shift left]{l}{\pe{\tensor*[]{e}{_{\bullet}}}}
	\arrow[shift right]{r}[swap]{\pe{e'_{\bullet}}}
&\lan (\tensor*[]{X}{_{\bullet}},e'_{\bullet}), \tensor*[_{(\tensor*[]{X}{_{0}}, e'_{0})}]{\wt{0}}{_{(\tensor*[]{X}{_{n+1}}, e'_{n+1})}} \ran 
	\arrow[shift right]{l}[swap]{\ie{e'_{\bullet}}},
\end{tikzcd}
\end{equation}
in the category of $\BF$-attached complexes. 
To see that $\ie{\tensor*[]{e}{_{\bullet}}}$ is a morphism of $\BF$-attached complexes, we just need to check that 
$(\ie{\tensor*[]{e}{_{0}}})_{\BF}\tilde{\delta} = (\ie{\tensor*[]{e}{_{n+1}}}\tensor*[]{)}{^{\BF}}\tilde{\rho}$. By \cref{rem:equality-of-extensions-in-tildeC}, it is enough to see that $(\tensor*[]{e}{_{0}}\tensor*[]{)}{_{\BE}}\delta = (\tensor*[]{e}{_{n+1}}\tensor*[]{)}{^{\BE}}\delta$ holds, and this is indeed true because 
$\delta = (\tensor*[]{e}{_{0}}\tensor*[]{)}{_{\BE}} \delta = (\tensor*[]{e}{_{n+1}}\tensor*[]{)}{^{\BE}} \delta$. 
Similarly, we see that $\pe{\tensor*[]{e}{_{\bullet}}}$ is a morphism of $\BF$-attached complexes. 
To see that $\ie{e'_{\bullet}}$ and $\pe{e'_{\bullet}}$ are morphisms of $\BF$-attached complexes, one uses that 
$(\id{\tensor*[]{X}{_{0}}} - \tensor*[]{e}{_{0}}\tensor*[]{)}{_{\BE}}\delta = \tensor*[_{\tensor*[]{X}{_{0}}}]{0}{_{\tensor*[]{X}{_{n+1}}}} = (\id{\tensor*[]{X}{_{n+1}}} - \tensor*[]{e}{_{n+1}}\tensor*[]{)}{^{\BE}}\delta$. 
Furthermore, we have the identities 
$\wid{(\tensor*[]{X}{_{\bullet}}, \tensor*[]{e}{_{\bullet}})} = \pe{\tensor*[]{e}{_{\bullet}}} \ie{\tensor*[]{e}{_{\bullet}}}$,
$\wid{(\tensor*[]{X}{_{\bullet}}, e'_{\bullet})} = \pe{e'_{\bullet}} \ie{e'_{\bullet}}$ 
and
$\wid{\tensor*[]{\SI}{_{\CC}}(\tensor*[]{X}{_{\bullet}})} =  \ie{\tensor*[]{e}{_{\bullet}}} \pe{\tensor*[]{e}{_{\bullet}}} + \ie{e'_{\bullet}} \pe{e'_{\bullet}}$, so \eqref{eqn:biproductdiagram} is a biproduct diagram in the additive category of $\BF$-attached complexes.
Therefore, we have that \eqref{eqn:isomorphism-of-t-n-exangles} is an isomorphism 
as $\BF$-attached complexes. 

Lastly, since $\langle \tensor*[]{\SI}{_{\CC}}(\tensor*[]{X}{_{\bullet}}), \tilde{\rho} \rangle$ is $\ft$-distinguished, it follows from \cite[Prop.\ 3.3]{HerschendLiuNakaoka-n-exangulated-categories-I-definitions-and-fundamental-properties} that \eqref{eqn:isomorphism-of-t-n-exangles} is an isomorphism of $\ft$-distinguished $n$-exangles.
\end{proof}

Thus, we can now present and prove the main result of this section, which shows that the $n$-exangulated inclusion functor $(\tensor*[]{\SI}{_{\CC}}, \Gamma) \colon (\CC, \BE, \fs) \to (\wt{\CC}, \BF, \ft)$ is $2$-universal amongst $n$-exangulated functors from $(\CC,\BE,\fs)$ to idempotent complete $n$-exangulated categories.

\begin{thm}\label{thm:IC-is-2-universal}
    Suppose $(\SF,\Lambda) \colon (\CC, \BE, \fs) \to (\CC', \BE', \fs')$ is an $n$-exangulated functor to an idempotent complete $n$-exangulated category $(\CC', \BE', \fs')$. Then the following statements hold. 
    \begin{enumerate}[label=\textup{(\roman*)}]
        \item \label{item:universal-property-1}
        There is an $n$-exangulated functor $(\SE, \Psi) \colon (\wt{\CC}, \BF, \ft) \to (\CC', \BE', \fs')$ and an $n$-exangulated natural isomorphism $\Tsadi \colon (\SF, \Lambda) \overset{\iso}{\Longrightarrow} (\SE, \Psi) \circ (\tensor*[]{\SI}{_{\CC}}, \Gamma)$.

        \item \label{item:universal-property-2}
        In addition, 
        for any $n$-exangulated functor $(\SG, \Theta) \colon (\wt{\CC},\BF,\ft) \to (\CC',  \BE', \fs')$ and any $n$-exangulated natural transformation $\Daleth \colon (\SF, \Lambda) \Rightarrow (\SG, \Theta) \circ (\tensor*[]{\SI}{_{\CC}}, \Gamma)$, 
        there is a unique $n$-exangulated natural transformation $\Mem \colon (\SE,\Psi) \Rightarrow (\SG,\Theta)$ with $\Daleth = \tensor*[]{\Mem}{_{\tensor*[]{\SI}{_{\CC}}}} \Tsadi$.
    \end{enumerate}
\end{thm}

\begin{proof}
\ref{item:universal-property-1}\; 
By \cref{prop:Buehler-universal-property-of-karoubi-envelope}\ref{2.8.1}, there exists an additive functor $\SE \colon \wt{\CC} \to \CC'$ and a natural isomorphism $\Tsadi \colon \SF \Rightarrow \SE \tensor*[]{\SI}{_{\CC}}$. It remains to show that $\SE$ forms part of an $n$-exangulated functor $(\SE,\Psi)$ and that $\Tsadi$ is $n$-exangulated.
 
First, we define a natural transformation $\Psi \colon \BF(-,-) \Rightarrow \BE'(\SE-, \SE-)$ as the composition of several abelian group homomorphisms. 
For $\tensor*[]{X}{_{0}},\tensor*[]{X}{_{n+1}} \in \CC$, we set 
\[
\tensor*[]{T}{_{(\tensor*[]{X}{_{n+1}},\tensor*[]{X}{_{0}})}} 
	\deff \BE'(\tensor*[]{\Tsadi}{_{\tensor*[]{X}{_{n+1}}}^{-1}}, \tensor*[]{\Tsadi}{_{\tensor*[]{X}{_{0}}}})
		\colon \BE'(\SF(\tensor*[]{X}{_{n+1}}),\SF(\tensor*[]{X}{_{0}}))
			\to \BE'(\SE \tensor*[]{\SI}{_{\CC}}(\tensor*[]{X}{_{n+1}}), \SE \tensor*[]{\SI}{_{\CC}}(\tensor*[]{X}{_{0}})).
\]
For $\tensor*[]{\wt{X}}{_{0}} = (\tensor*[]{X}{_{0}}, \tensor*[]{e}{_{0}})$ and $\tensor*[]{\wt{X}}{_{n+1}} = (\tensor*[]{X}{_{n+1}}, \tensor*[]{e}{_{n+1}})$ in $\wt{\CC}$, 
we define an abelian group homomorphism 
\begin{align*}
\tensor*[]{I}{_{(\tensor*[]{\wt{X}}{_{n+1}},\tensor*[]{\wt{X}}{_{0}})}} \colon \BF(\tensor*[]{\wt{X}}{_{n+1}}, \tensor*[]{\wt{X}}{_{0}})
    &\longrightarrow \BE(\tensor*[]{X}{_{n+1}}, \tensor*[]{X}{_{0}}) \\
\tilde{\delta} = (\tensor*[]{e}{_{0}}, \delta, \tensor*[]{e}{_{n+1}})
    &\longmapsto \delta,
\end{align*}
and put 
\[
\tensor*[]{P}{_{(\tensor*[]{\wt{X}}{_{n+1}},\tensor*[]{\wt{X}}{_{0}})}} \deff \BE'(\SE(\ie{\tensor*[]{e}{_{n+1}}}), \SE(\pe{\tensor*[]{e}{_{0}}})) 
	\colon \BE'(\SE \tensor*[]{\SI}{_{\CC}}(\tensor*[]{X}{_{n+1}}), \SE \tensor*[]{\SI}{_{\CC}}(\tensor*[]{X}{_{0}})) 
		\to \BE'(\SE(\tensor*[]{\wt{X}}{_{n+1}}),\SE(\tensor*[]{\wt{X}}{_{0}})).
\]
For morphisms $a \colon \tensor*[]{X}{_{0}} \to \tensor*[]{Y}{_{0}}$ and $c \colon \tensor*[]{Y}{_{n+1}} \to \tensor*[]{X}{_{n+1}}$ in $\CC$, we have 
\begin{align}\label{eqn:T}
    \tensor*[]{T}{_{(\tensor*[]{Y}{_{n+1}}, \tensor*[]{Y}{_{0}})}} \BE'(\SF(c), \SF(a)) = \BE'(\SE \tensor*[]{\SI}{_{\CC}} (c), \SE \tensor*[]{\SI}{_{\CC}} (a)) \tensor*[]{T}{_{(\tensor*[]{X}{_{n+1}}, \tensor*[]{X}{_{0}})}}
\end{align}
as $\Tsadi$ is natural.
For morphisms $\tilde{a} \colon \tensor*[]{\wt{X}}{_{0}} \to \tensor*[]{\wt{Y}}{_{0}}$ and $\tilde{c} \colon \tensor*[]{\wt{Y}}{_{n+1}} \to \tensor*[]{\wt{X}}{_{n+1}}$ in $\wt{\CC}$,   
we have
\begin{align}\label{eqn:inclusion}
    \tensor*[]{I}{_{(\tensor*[]{\wt{Y}}{_{n+1}},\tensor*[]{\wt{Y}}{_{0}})}} \BF(\ct, \at) = \BE(\cu, \au) \tensor*[]{I}{_{(\tensor*[]{\wt{X}}{_{n+1}},\tensor*[]{\wt{X}}{_{0}})}},
\end{align}
using how $\BF$ is defined on morphisms (see \cref{def:BF}). 
We claim that the family of abelian group homomorphisms 
\[
\tensor*[]{\Psi}{_{(\tensor*[]{\wt{X}}{_{n+1}},\tensor*[]{\wt{X}}{_{0}})}} 
	\deff 	\tensor*[]{P}{_{(\tensor*[]{\wt{X}}{_{n+1}},\tensor*[]{\wt{X}}{_{0}})}} 
			\tensor*[]{T}{_{(\tensor*[]{X}{_{n+1}},\tensor*[]{X}{_{0}})}} 
			\tensor*[]{\Lambda}{_{(\tensor*[]{X}{_{n+1}}, \tensor*[]{X}{_{0}})}} 
			\tensor*[]{I}{_{(\tensor*[]{\wt{X}}{_{n+1}},\tensor*[]{\wt{X}}{_{0}})}}
\]
for $\tensor*[]{\wt{X}}{_{0}},\tensor*[]{\wt{X}}{_{n+1}}\in\wt{\CC}$
defines a natural transformation $\Psi \colon \BF(-,-) \Rightarrow \BE'(\SE-, \SE-)$. 
To this end, fix objects 
$
\tensor*[]{\wt{X}}{_{0}} = (\tensor*[]{X}{_{0}}, 
    \tensor*[]{e}{_{0}}), 
    \tensor*[]{\wt{Y}}{_{0}} = (\tensor*[]{Y}{_{0}}, e'_0), 
    \tensor*[]{\wt{X}}{_{n+1}} = (\tensor*[]{X}{_{n+1}}, \tensor*[]{e}{_{n+1}})
    $
and
$
\tensor*[]{\wt{Y}}{_{n+1}} = (\tensor*[]{Y}{_{n+1}}, e'_{n+1})
$,
and morphisms $\tilde{a} \colon \tensor*[]{\wt{X}}{_{0}} \to \tensor*[]{\wt{Y}}{_{0}}$ and $\tilde{c} \colon \tensor*[]{\wt{Y}}{_{n+1}} \to \tensor*[]{\wt{X}}{_{n+1}}$ in $\wt{\CC}$. 
First, note that we have 
\begin{equation}\label{eqn:horrible1}
\pe{e'_{0}}\tensor*[]{\SI}{_{\CC}}(\au)
	= \pe{e'_{0}}\tensor*[]{\SI}{_{\CC}}(\au) \tensor*[]{\SI}{_{\CC}}(\tensor*[]{e}{_{0}})
	= \pe{e'_{0}}\tensor*[]{\SI}{_{\CC}}(\au) \ie{\tensor*[]{e}{_{0}}}\pe{\tensor*[]{e}{_{0}}}
	= \at \pe{\tensor*[]{e}{_{0}}}
\end{equation}
and, similarly,  
\begin{align}\label{eqn:horrible2}
    \tensor*[]{\SI}{_{\CC}}(\cu) \ie{e'_{n+1}} = \ie{\tensor*[]{e}{_{n+1}}} \ct.
\end{align}
Therefore, we see that
\begin{align*}
    &\tensor*[]{\Psi}{_{(\tensor*[]{\wt{Y}}{_{n+1}},\tensor*[]{\wt{Y}}{_{0}})}} \BF(\ct,\at)\\
    &\;\;= \tensor*[]{P}{_{(\tensor*[]{\wt{Y}}{_{n+1}},\tensor*[]{\wt{Y}}{_{0}})}} \tensor*[]{T}{_{(\tensor*[]{Y}{_{n+1}},\tensor*[]{Y}{_{0}})}} \tensor*[]{\Lambda}{_{(\tensor*[]{Y}{_{n+1}}, \tensor*[]{Y}{_{0}})}} \tensor*[]{I}{_{(\tensor*[]{\wt{Y}}{_{n+1}},\tensor*[]{\wt{Y}}{_{0}})}} \BF(\ct,\at) \\
    &\;\;= \tensor*[]{P}{_{(\tensor*[]{\wt{Y}}{_{n+1}},\tensor*[]{\wt{Y}}{_{0}})}} \tensor*[]{T}{_{(\tensor*[]{Y}{_{n+1}},\tensor*[]{Y}{_{0}})}} \tensor*[]{\Lambda}{_{(\tensor*[]{Y}{_{n+1}}, \tensor*[]{Y}{_{0}})}} \BE(\cu,\au) \tensor*[]{I}{_{(\tensor*[]{\wt{X}}{_{n+1}},\tensor*[]{\wt{X}}{_{0}})}} && \text{by }\eqref{eqn:inclusion} \\
    &\;\;= \tensor*[]{P}{_{(\tensor*[]{\wt{Y}}{_{n+1}},\tensor*[]{\wt{Y}}{_{0}})}} \tensor*[]{T}{_{(\tensor*[]{Y}{_{n+1}},\tensor*[]{Y}{_{0}})}} \BE'(\SF(\cu),\SF(\au)) \tensor*[]{\Lambda}{_{(\tensor*[]{X}{_{n+1}}, \tensor*[]{X}{_{0}})}} \tensor*[]{I}{_{(\tensor*[]{\wt{X}}{_{n+1}},\tensor*[]{\wt{X}}{_{0}})}} && \text{as }\Lambda\text{ is natural}\\
    &\;\;= \tensor*[]{P}{_{(\tensor*[]{\wt{Y}}{_{n+1}},\tensor*[]{\wt{Y}}{_{0}})}} \BE'(\SE\tensor*[]{\SI}{_{\CC}}(\cu),\SE\tensor*[]{\SI}{_{\CC}}(\au)) \tensor*[]{T}{_{(\tensor*[]{X}{_{n+1}},\tensor*[]{X}{_{0}})}} \tensor*[]{\Lambda}{_{(\tensor*[]{X}{_{n+1}}, \tensor*[]{X}{_{0}})}} \tensor*[]{I}{_{(\tensor*[]{\wt{X}}{_{n+1}},\tensor*[]{\wt{X}}{_{0}})}} &&\text{by }\eqref{eqn:T} \\
    &\;\;= \BE'(\SE(\tensor*[]{\SI}{_{\CC}}(\cu)\ie{e'_{n+1}}),\SE(\pe{e'_{0}} \tensor*[]{\SI}{_{\CC}}(\au) )) \tensor*[]{T}{_{(\tensor*[]{X}{_{n+1}},\tensor*[]{X}{_{0}})}} \tensor*[]{\Lambda}{_{(\tensor*[]{X}{_{n+1}}, \tensor*[]{X}{_{0}})}} \tensor*[]{I}{_{(\tensor*[]{\wt{X}}{_{n+1}},\tensor*[]{\wt{X}}{_{0}})}} \\
    &\;\;= \BE'(\SE(\ie{\tensor*[]{e}{_{n+1}}}\ct),\SE(\at \pe{\tensor*[]{e}{_{0}}})) \tensor*[]{T}{_{(\tensor*[]{X}{_{n+1}},\tensor*[]{X}{_{0}})}} \tensor*[]{\Lambda}{_{(\tensor*[]{X}{_{n+1}}, \tensor*[]{X}{_{0}})}} \tensor*[]{I}{_{(\tensor*[]{\wt{X}}{_{n+1}},\tensor*[]{\wt{X}}{_{0}})}} && \text{by }\eqref{eqn:horrible1}\text{ and } \eqref{eqn:horrible2}\\
    &\;\;= \BE'(\SE(\ct), \SE(\at)) \tensor*[]{P}{_{(\tensor*[]{\wt{X}}{_{n+1}},\tensor*[]{\wt{X}}{_{0}})}} \tensor*[]{T}{_{(\tensor*[]{X}{_{n+1}},\tensor*[]{X}{_{0}})}} \tensor*[]{\Lambda}{_{(\tensor*[]{X}{_{n+1}}, \tensor*[]{X}{_{0}})}} \tensor*[]{I}{_{(\tensor*[]{\wt{X}}{_{n+1}},\tensor*[]{\wt{X}}{_{0}})}}  \\
    &\;\;
    = \BE'(\SE(\ct), \SE(\at)) \tensor*[]{\Psi}{_{(\tensor*[]{\wt{X}}{_{n+1}},\tensor*[]{\wt{X}}{_{0}})}}. 
\end{align*}

Next, we must show that $(\SE,\Psi)$ sends $\ft$-distinguished $n$-exangles to $\fs'$-distinguished $n$-exangles. 
Thus, let $\tensor*[]{\wt{X}}{_{0}} = (\tensor*[]{X}{_{0}}, \tensor*[]{e}{_{0}}), \tensor*[]{\wt{X}}{_{n+1}} = (\tensor*[]{X}{_{n+1}}, \tensor*[]{e}{_{n+1}})\in\wt{\CC}$ and $\tilde{\delta} \in \BF(\tensor*[]{\wt{X}}{_{n+1}}, \tensor*[]{\wt{X}}{_{0}})$, and suppose $\ft(\tilde{\delta}) = [\tensor*[]{\wt{X}}{_{\bullet}}]$. 
We need that 
$\smash{\fs'(\tensor*[]{\Psi}{_{(\tensor*[]{\wt{X}}{_{n+1}},\tensor*[]{\wt{X}}{_{0}})}} (\tilde{\delta}))
	= [\SE(\tensor*[]{\wt{X}}{_{\bullet}})]}$, 
which 
will follow from seeing that
$\langle \SE(\tensor*[]{\wt{X}}{_{\bullet}}), \tensor*[]{\Psi}{_{(\tensor*[]{\wt{X}}{_{n+1}},\tensor*[]{\wt{X}}{_{0}})}}(\tilde{\delta}) \rangle$
is a direct summand of an $\fs'$-distinguished $n$-exangle.

By \cref{rem:independence-of-t}, we may take a complex $\tensor*[]{X}{_{\bullet}}$ in $\com{\CC}^{\raisebox{0.5pt}{\scalebox{0.6}{$n$}}}$ with $\fs(\delta) = [\tensor*[]{X}{_{\bullet}}]$ 
and an idempotent morphism $\tensor*[]{e}{_{\bullet}} \colon \langle \tensor*[]{X}{_{\bullet}}, {\delta} \rangle \to \langle \tensor*[]{X}{_{\bullet}}, {\delta} \rangle$ lifting $(\tensor*[]{e}{_{0}}, \tensor*[]{e}{_{n+1}}) \colon \delta \to \delta$, such that $\ft(\tilde{\delta}) = [(\tensor*[]{X}{_{\bullet}}, \tensor*[]{e}{_{\bullet}})]$.
Note for later that we thus have 
$
[(\tensor*[]{X}{_{\bullet}}, \tensor*[]{e}{_{\bullet}})]
    = [\tensor*[]{\wt{X}}{_{\bullet}}]
$, 
and hence 
$
[\SE((\tensor*[]{X}{_{\bullet}}, \tensor*[]{e}{_{\bullet}}))]
    = [\SE(\tensor*[]{\wt{X}}{_{\bullet}})]
$.
Let $\tilde{\rho} \deff \tensor*[]{\Gamma}{_{(\tensor*[]{X}{_{n+1}},\tensor*[]{X}{_{0}})}}({\delta})$.
Since $(\SF, \Lambda) \colon (\CC, \BE, \fs) \to (\CC', \BE', \fs')$ is an $n$-exangulated functor, 
the $n$-exangle 
$\langle \SF(\tensor*[]{X}{_{\bullet}}), \tensor*[]{\Lambda}{_{(\tensor*[]{X}{_{n+1}}, \tensor*[]{X}{_{0}})}}({\delta}) \rangle$ 
is $\fs'$-distinguished.
As we have an isomorphism of complexes $\tensor*[]{\Tsadi}{_{\tensor*[]{X}{_{\bullet}}}} \colon \SF(\tensor*[]{X}{_{\bullet}}) \to \SE\tensor*[]{\SI}{_{\CC}}(\tensor*[]{X}{_{\bullet}})$,
 the $\BE'$-attached complex
$\langle \SE \tensor*[]{\SI}{_{\CC}}(\tensor*[]{X}{_{\bullet}}), \tensor*[]{T}{_{(\tensor*[]{X}{_{n+1}},\tensor*[]{X}{_{0}})}} \tensor*[]{\Lambda}{_{(\tensor*[]{X}{_{n+1}}, \tensor*[]{X}{_{0}})}} (\delta) \rangle$ 
is $\fs'$-distinguished by \cite[Cor.\ 2.26(2)]{HerschendLiuNakaoka-n-exangulated-categories-I-definitions-and-fundamental-properties}. 
Since 
$\tensor*[]{P}{_{(\tensor*[]{\SI}{_{\CC}}(\tensor*[]{X}{_{n+1}}),\tensor*[]{\SI}{_{\CC}}(\tensor*[]{X}{_{0}}))}} 
	= \BE'(\SE(\ie{\id{\tensor*[]{X}{_{n+1}}}}),\SE(\pe{\id{\tensor*[]{X}{_{0}}}}))
$ 
is just the identity homomorphism, a quick computation yields
\begin{equation}
\label{eqn:ThetaGamma-is-TLambda}
\tensor*[]{\Psi}{_{(\tensor*[]{\SI}{_{\CC}}(\tensor*[]{X}{_{n+1}}),\tensor*[]{\SI}{_{\CC}}(\tensor*[]{X}{_{0}}))}} 
(\tilde{\rho})
	= \tensor*[]{T}{_{(\tensor*[]{X}{_{n+1}},\tensor*[]{X}{_{0}})}}\tensor*[]{\Lambda}{_{(\tensor*[]{X}{_{n+1}}, \tensor*[]{X}{_{0}})}} ({\delta}) 
	= \BE'(\tensor*[]{\Tsadi}{_{\tensor*[]{X}{_{n+1}}}^{-1}}, \tensor*[]{\Tsadi}{_{\tensor*[]{X}{_{0}}}}) \tensor*[]{\Lambda}{_{(\tensor*[]{X}{_{n+1}}, \tensor*[]{X}{_{0}})}} ({\delta}).
\end{equation}
In particular, this implies that 
$\langle \SE \tensor*[]{\SI}{_{\CC}} (\tensor*[]{X}{_{\bullet}}), 
\tensor*[]{\Psi}{_{(\tensor*[]{\SI}{_{\CC}}(\tensor*[]{X}{_{n+1}}),\tensor*[]{\SI}{_{\CC}}(\tensor*[]{X}{_{0}}))}}(\tilde{\rho}) \rangle$ 
is $\fs'$-distinguished. 

Note that $\langle \tensor*[]{\SI}{_{\CC}}(\tensor*[]{X}{_{\bullet}}), \tilde{\rho} \rangle \iso \langle (\tensor*[]{X}{_{\bullet}}, \tensor*[]{e}{_{\bullet}}), \tilde{\delta} \rangle \oplus \langle (\tensor*[]{X}{_{\bullet}}, e'_{\bullet}), \tensor*[_{(\tensor*[]{X}{_{0}}, e'_{0})}]{\wt{0}}{_{(\tensor*[]{X}{_{n+1}}, e'_{n+1})}}\rangle$ as $\BF$-attached complexes by \cref{lem:summand-3}, 
where $e'_{\bullet} \deff \id{\tensor*[]{X}{_{\bullet}}} - \tensor*[]{e}{_{\bullet}}$. 
We see that 
$\langle \SE((\tensor*[]{X}{_{\bullet}}, \tensor*[]{e}{_{\bullet}})), \tensor*[]{\Psi}{_{(\tensor*[]{\wt{X}}{_{n+1}},\tensor*[]{\wt{X}}{_{0}})}}(\tilde{\delta}) \rangle$ 
is a direct summand of the $\fs'$-distinguished $n$-exangle 
$\langle \SE \tensor*[]{\SI}{_{\CC}} (\tensor*[]{X}{_{\bullet}}), 
\tensor*[]{\Psi}{_{(\tensor*[]{\SI}{_{\CC}}(\tensor*[]{X}{_{n+1}}),\tensor*[]{\SI}{_{\CC}}(\tensor*[]{X}{_{0}}))}}(\tilde{\rho}) \rangle$ 
by \cref{prop:morphisms-of-n-exangles-preserved-under-n-exang-functor}\ref{item:SF-Gamma-preserves-direct-sums-of-n-exangle}. 
Hence, 
$\fs'(\tensor*[]{\Psi}{_{(\tensor*[]{\wt{X}}{_{n+1}},\tensor*[]{\wt{X}}{_{0}})}}(\tilde{\delta})) 
    = [\SE((\tensor*[]{X}{_{\bullet}},\tensor*[]{e}{_{\bullet}}))] 
    = [\SE(\tensor*[]{\wt{X}}{_{\bullet}})]
$ 
by \cite[Prop.\ 3.3]{HerschendLiuNakaoka-n-exangulated-categories-I-definitions-and-fundamental-properties},  and so $(\SE,\Psi)$ is an $n$-exangulated functor.

Lastly, it follows immediately from \eqref{eqn:ThetaGamma-is-TLambda} that $\Tsadi$ is an $n$-exangulated natural transformation 
$(\SF, \Lambda) \Rightarrow (\SE , \Psi) \circ (\tensor*[]{\SI}{_{\CC}}, \Gamma) 
	= (\SE\tensor*[]{\SI}{_{\CC}}, \tensor*[]{\Psi}{_{\tensor*[]{\SI}{_{\CC}}\times\tensor*[]{\SI}{_{\CC}}}}\Gamma)$.

\ref{item:universal-property-2}\; 
By \cref{prop:Buehler-universal-property-of-karoubi-envelope}\ref{2.8.2}, there exists a unique natural transformation $\Mem \colon \SE \Rightarrow \SG$ with $\Daleth = \tensor*[]{\Mem}{_{\tensor*[]{\SI}{_{\CC}}}} \Tsadi$, so it remains to show that $\Mem$ induces an $n$-exangulated natural transformation $(\SE, \Psi) \Rightarrow (\SG, \Theta)$.
For this, let $\tensor*[]{\wt{X}}{_{0}} = (\tensor*[]{X}{_{0}},\tensor*[]{e}{_{0}}), \tensor*[]{\wt{X}}{_{n+1}} = (\tensor*[]{X}{_{n+1}},\tensor*[]{e}{_{n+1}}) \in \wt{\CC}$ and $\tilde{\delta} \in \BF(\tensor*[]{\wt{X}}{_{n+1}}, \tensor*[]{\wt{X}}{_{0}})$ be arbitrary. Note that we have 
\begin{equation}
\label{eqn:deltatilde-equals-BFGamma}
\tilde{\delta} = \BF(\ie{\tensor*[]{e}{_{n+1}}}, \pe{\tensor*[]{e}{_{0}}}) \tensor*[]{\Gamma}{_{(\tensor*[]{X}{_{n+1}},\tensor*[]{X}{_{0}})}} (\delta).
\end{equation}
Hence, we obtain
\begin{align*}
    &(\tensor*[]{\Mem}{_{\tensor*[]{\wt{X}}{_{0}}}}\tensor*[]{)}{_{\BE'}} \tensor*[]{\Psi}{_{(\tensor*[]{\wt{X}}{_{n+1}},\tensor*[]{\wt{X}}{_{0}})}} (\tilde{\delta})\\
    &\;\;= (\tensor*[]{\Mem}{_{\tensor*[]{\wt{X}}{_{0}}}}\tensor*[]{)}{_{\BE'}} \tensor*[]{\Psi}{_{(\tensor*[]{\wt{X}}{_{n+1}},\tensor*[]{\wt{X}}{_{0}})}} \BF(\ie{\tensor*[]{e}{_{n+1}}}, \pe{\tensor*[]{e}{_{0}}}) \tensor*[]{\Gamma}{_{(\tensor*[]{X}{_{n+1}},\tensor*[]{X}{_{0}})}} (\delta) &&\text{by }\eqref{eqn:deltatilde-equals-BFGamma} \\
    &\;\;= (\tensor*[]{\Mem}{_{\tensor*[]{\wt{X}}{_{0}}}}\tensor*[]{)}{_{\BE'}} \BE'(\SE(\ie{\tensor*[]{e}{_{n+1}}}), \SE(\pe{\tensor*[]{e}{_{0}}})) (\tensor*[]{\Psi}{_{\tensor*[]{\SI}{_{\CC}} \times \tensor*[]{\SI}{_{\CC}}}}  \Gamma\tensor*[]{)}{_{(\tensor*[]{X}{_{n+1}}, \tensor*[]{X}{_{0}})}} (\delta) &&\text{as $\Psi$ is natural} \\
    &\;\;= \BE'(\SE(\ie{\tensor*[]{e}{_{n+1}}}), \tensor*[]{\Mem}{_{\tensor*[]{\wt{X}}{_{0}}}} \SE(\pe{\tensor*[]{e}{_{0}}})) (\tensor*[]{\Psi}{_{\tensor*[]{\SI}{_{\CC}} \times \tensor*[]{\SI}{_{\CC}}}}  \Gamma\tensor*[]{)}{_{(\tensor*[]{X}{_{n+1}}, \tensor*[]{X}{_{0}})}} (\delta) &&\\
    &\;\;= \BE'(\SE(\ie{\tensor*[]{e}{_{n+1}}}), \SG(\pe{\tensor*[]{e}{_{0}}}) \tensor*[]{\Mem}{_{\tensor*[]{\SI}{_{\CC}}(\tensor*[]{X}{_{0}})}}  )  (\tensor*[]{\Psi}{_{\tensor*[]{\SI}{_{\CC}} \times \tensor*[]{\SI}{_{\CC}}}}  \Gamma\tensor*[]{)}{_{(\tensor*[]{X}{_{n+1}}, \tensor*[]{X}{_{0}})}} (\delta) && \text{as $\Mem$ is natural}\\
    &\;\;= \BE'(\SE(\ie{\tensor*[]{e}{_{n+1}}}), \SG(\pe{\tensor*[]{e}{_{0}}}) \tensor*[]{\Daleth}{_{\tensor*[]{X}{_{0}}}} \tensor*[]{\Tsadi}{_{\tensor*[]{X}{_{0}}}^{-1}}  )   (\tensor*[]{\Psi}{_{\tensor*[]{\SI}{_{\CC}} \times \tensor*[]{\SI}{_{\CC}}}}  \Gamma\tensor*[]{)}{_{(\tensor*[]{X}{_{n+1}}, \tensor*[]{X}{_{0}})}} (\delta) &&\text{as $\Daleth = \tensor*[]{\Mem}{_{\tensor*[]{\SI}{_{\CC}}}} \Tsadi$}\\
    &\;\;= \BE'(\SE(\ie{\tensor*[]{e}{_{n+1}}}), \SG(\pe{\tensor*[]{e}{_{0}}})) (\tensor*[]{\Daleth}{_{\tensor*[]{X}{_{0}}}}\tensor*[]{)}{_{\BE'}} (\tensor*[]{\Tsadi}{_{\tensor*[]{X}{_{0}}}^{-1}}\tensor*[]{)}{_{\BE'}}   (\tensor*[]{\Psi}{_{\tensor*[]{\SI}{_{\CC}} \times \tensor*[]{\SI}{_{\CC}}}}  \Gamma\tensor*[]{)}{_{(\tensor*[]{X}{_{n+1}}, \tensor*[]{X}{_{0}})}} (\delta) && \\
    &\;\;= \BE'(\SE(\ie{\tensor*[]{e}{_{n+1}}}), \SG(\pe{\tensor*[]{e}{_{0}}})) (\tensor*[]{\Daleth}{_{\tensor*[]{X}{_{0}}}}\tensor*[]{)}{_{\BE'}} (\tensor*[]{\Tsadi}{_{\tensor*[]{X}{_{n+1}}}^{-1}}\tensor*[]{)}{^{\BE'}}  \tensor*[]{\Lambda}{_{(\tensor*[]{X}{_{n+1}}, \tensor*[]{X}{_{0}})}} (\delta) && \text{as $\Tsadi$ is $n$-exangulated}\\
    &\;\;= \BE'(\SE(\ie{\tensor*[]{e}{_{n+1}}}), \SG(\pe{\tensor*[]{e}{_{0}}}))  (\tensor*[]{\Tsadi}{_{\tensor*[]{X}{_{n+1}}}^{-1}}\tensor*[]{)}{^{\BE'}} (\tensor*[]{\Daleth}{_{\tensor*[]{X}{_{0}}}}\tensor*[]{)}{_{\BE'}} \tensor*[]{\Lambda}{_{(\tensor*[]{X}{_{n+1}}, \tensor*[]{X}{_{0}})}} (\delta) && \\
    &\;\;= \BE'(\SE(\ie{\tensor*[]{e}{_{n+1}}}), \SG(\pe{\tensor*[]{e}{_{0}}}))  (\tensor*[]{\Tsadi}{_{\tensor*[]{X}{_{n+1}}}^{-1}}\tensor*[]{)}{^{\BE'}} (\tensor*[]{\Daleth}{_{\tensor*[]{X}{_{n+1}}}}\tensor*[]{)}{^{\BE'}} (\tensor*[]{\Theta}{_{\tensor*[]{\SI}{_{\CC}} \times \tensor*[]{\SI}{_{\CC}}}}  \Gamma\tensor*[]{)}{_{(\tensor*[]{X}{_{n+1}}, \tensor*[]{X}{_{0}})}} (\delta) && \text{as $\Daleth$ is $n$-exangulated}\\
    &\;\;= \BE'(\tensor*[]{\Daleth}{_{\tensor*[]{X}{_{n+1}}}} \tensor*[]{\Tsadi}{_{\tensor*[]{X}{_{n+1}}}^{-1}}\SE(\ie{\tensor*[]{e}{_{n+1}}}), \SG(\pe{\tensor*[]{e}{_{0}}})) (\tensor*[]{\Theta}{_{\tensor*[]{\SI}{_{\CC}} \times \tensor*[]{\SI}{_{\CC}}}}  \Gamma\tensor*[]{)}{_{(\tensor*[]{X}{_{n+1}}, \tensor*[]{X}{_{0}})}} (\delta) &&\\
    &\;\;= \BE'(\tensor*[]{\Mem}{_{\tensor*[]{\SI}{_{\CC}}(\tensor*[]{X}{_{n+1}})}}\SE(\ie{\tensor*[]{e}{_{n+1}}}), \SG(\pe{\tensor*[]{e}{_{0}}})) (\tensor*[]{\Theta}{_{\tensor*[]{\SI}{_{\CC}} \times \tensor*[]{\SI}{_{\CC}}}}  \Gamma\tensor*[]{)}{_{(\tensor*[]{X}{_{n+1}}, \tensor*[]{X}{_{0}})}} (\delta)  &&\text{as $\Daleth = \tensor*[]{\Mem}{_{\tensor*[]{\SI}{_{\CC}}}} \Tsadi$}\\
    &\;\;= \BE'(\SG(\ie{\tensor*[]{e}{_{n+1}}}) \tensor*[]{\Mem}{_{\tensor*[]{\wt{X}}{_{n+1}}}}, \SG(\pe{\tensor*[]{e}{_{0}}})) (\tensor*[]{\Theta}{_{\tensor*[]{\SI}{_{\CC}} \times \tensor*[]{\SI}{_{\CC}}}}  \Gamma\tensor*[]{)}{_{(\tensor*[]{X}{_{n+1}}, \tensor*[]{X}{_{0}})}} (\delta) && \text{as $\Mem$ is natural}\\
    &\;\;= (\tensor*[]{\Mem}{_{\tensor*[]{\wt{X}}{_{n+1}}}}\tensor*[]{)}{^{\BE'}} \BE'(\SG(\ie{\tensor*[]{e}{_{n+1}}}) , \SG(\pe{\tensor*[]{e}{_{0}}})) (\tensor*[]{\Theta}{_{\tensor*[]{\SI}{_{\CC}} \times \tensor*[]{\SI}{_{\CC}}}}  \Gamma\tensor*[]{)}{_{(\tensor*[]{X}{_{n+1}}, \tensor*[]{X}{_{0}})}} (\delta) && \\
    &\;\;= (\tensor*[]{\Mem}{_{\tensor*[]{\wt{X}}{_{n+1}}}}\tensor*[]{)}{^{\BE'}} \tensor*[]{\Theta}{_{(\tensor*[]{\wt{X}}{_{n+1}},\tensor*[]{\wt{X}}{_{0}})}}\BF(\ie{\tensor*[]{e}{_{n+1}}}, \pe{\tensor*[]{e}{_{0}}}) \tensor*[]{\Gamma}{_{(\tensor*[]{X}{_{n+1}},\tensor*[]{X}{_{0}})}} (\delta)&& \text{as $\Theta$ is natural}\\
    &\;\;
    = (\tensor*[]{\Mem}{_{\tensor*[]{\wt{X}}{_{n+1}}}}\tensor*[]{)}{^{\BE'}} \tensor*[]{\Theta}{_{(\tensor*[]{\wt{X}}{_{n+1}},\tensor*[]{\wt{X}}{_{0}})}} (\tilde{\delta}) && \text{by }\eqref{eqn:deltatilde-equals-BFGamma},
\end{align*}
and the proof is complete.
\end{proof}

We close this section with some remarks on our main results and constructions.

\begin{rem}\label{rem:recovering-previous-cases}
Before commenting on how our results unify the constructions in cases in the literature and on how our proof methods compare, we set up and recall a little terminology. 
Suppose $(\CC,\BE,\fs)$ and $(\CC',\BE',\fs')$ are $n$-exangulated categories. 
We call an $n$-exangulated functor $(\SF,\Gamma) \colon (\CC,\BE,\fs) \to (\CC',\BE',\fs')$
an \emph{$n$-exangulated isomorphism} 
if $\SF$ is an isomorphism of categories and $\Gamma$ is a natural isomorphism. 
This terminology is justified by \mbox{\cite[Prop.\ 4.11]{Bennett-TennenhausHauglandSandoyShah-the-category-of-extensions-and-a-characterisation-of-n-exangulated-functors}}. 
Lastly, we recall that $n$-exangulated functors between $(n+2)$-angulated categories are \emph{$(n+2)$-angulated} in the sense of \cite[Def.\ 2.7]{Bennett-TennenhausShah-transport-of-structure-in-higher-homological-algebra} 
(or $exact$ as in Bergh--Thaule \mbox{\cite[Sec.\ 4]{BerghThaule-the-grothendieck-group-of-an-n-angulated-category}}), 
and that $n$-exangulated functors between $n$-exact categories are \emph{$n$-exact} in the sense of \cite[Def.\ 2.18]{Bennett-TennenhausShah-transport-of-structure-in-higher-homological-algebra}; 
see \cite[Thms.\ 2.33, 2.34]{Bennett-TennenhausShah-transport-of-structure-in-higher-homological-algebra}.

It has been shown that a triplet $(\CC,\BE,\fs)$ is a $1$-exangulated category if and only if it is extriangulated (see \cite[Prop.\ 4.3]{HerschendLiuNakaoka-n-exangulated-categories-I-definitions-and-fundamental-properties}).  
Suppose that $(\CC,\BE,\fs)$ is an extriangulated category and consider the idempotent completion $\wt{\CC}$ of $\CC$. By \cite[Thm.\ 3.1]{Msapato-the-karoubi-envelope-and-weak-idempotent-completion-of-an-extriangulated-category}, there is an extriangulated structure $(\BF',\ft')$ on $\wt{\CC}$. 
By our 
\cref{thm:mainthm-ctilde},
there is a $1$-exangulated (or extriangulated) category $(\wt{\CC},\BF,\ft)$. By direct comparison of the constructions, one can check that $(\wt{\CC},\BF,\ft)$ and $(\wt{\CC},\BF',\ft')$ are $n$-exangulated isomorphic. Indeed, the bifunctors $\BF$ and $\BF'$ differ only by a labelling of the elements due to our convention in \cref{def:BF}; and, ignoring this re-labelling, the realisations $\fs$ and $\fs'$ are the same by \cref{lem:independent}. Furthermore, since $(\wt{\CC},\BF',\ft')$ recovers the triangulated and exact category cases, we see that our construction agrees with the classical (i.e.\ $n=1$) cases up to $n$-exangulated isomorphism.

For larger $n$, we just need to compare $(\wt{\CC},\BF,\ft)$ with the construction in \mbox{\cite{Lin-Idempotent-completion-of-n-Angulated-categories}}. Thus, suppose $(\CC,\BE,\fs)$ is the $n$-exangulated category coming from an $(n+2)$-angulated category $(\CC,\Sigma,\pentagon)$. Recall that in this case $\BE(Z,X) = \CC(Z,\Sigma X)$ for $X,Z\in\CC$. 
Using \mbox{\cite[Thm.\ 3.1]{Lin-Idempotent-completion-of-n-Angulated-categories}}, one obtains an $(n+2)$-angulated category $(\wt{\CC},\wt{\Sigma},\wt{\pentagon})$, where $\wt{\Sigma}$ is induced by $\Sigma$. From this $(n+2)$-angulated category, just like above, we obtain an induced $n$-exangulated category $(\wt{\CC},\BF',\ft')$. Notice that 
$\BF'(-,-) = \wt{\CC}(-,\wt{\Sigma}-)$. 
Comparing $(\wt{\CC},\BF',\ft')$ to the $n$-exangulated category $(\wt{\CC},\BF,\ft)$ found from \cref{thm:mainthm-ctilde}, again we see that $\BF$ and $\BF'$ differ by the labelling convention we chose in \cref{def:BF}. 
By \cite[Prop.\ 4.8]{HerschendLiuNakaoka-n-exangulated-categories-I-definitions-and-fundamental-properties} we have that $(\wt{\CC},\BF,\ft)$ induces an $(n+2)$-angulated category $(\wt{\CC},\wt{\Sigma},\pentagon')$, and therefore the $n$-exangulated inclusion functor $\tensor*[]{\SI}{_{\CC}}\colon\CC\to\wt{\CC}$ is, moreover, $(n+2)$-angulated. It follows from \cite[Thm.\ 3.1(2)]{Lin-Idempotent-completion-of-n-Angulated-categories} that 
$\wt{\pentagon}$ and $\pentagon'$ must be equal, and hence
$(\wt{\CC},\BF,\ft)$ and $(\wt{\CC},\BF',\ft')$ are $n$-exangulated isomorphic.
\end{rem}

\begin{rem}\label{rem:proofs-are-hard}
Our proofs in this article differ from the proofs in both the extriangulated and the $(n+2)$-angulated cases. 
First, the axioms for an $n$-exangulated category look very different from the axioms for an extriangulated category. 
Therefore, the proofs from \cite{Msapato-the-karoubi-envelope-and-weak-idempotent-completion-of-an-extriangulated-category} cannot be directly generalised to the $n>1$ case. 
Even of the results that seem like they might generalise nicely, one comes across immediate obstacles. 
Indeed, Lin \cite[p.\ 1064]{Lin-Idempotent-completion-of-n-Angulated-categories} already points out that lifting idempotent morphisms of extensions to idempotent morphisms of $n$-exangles is non-trivial. Despite this, we are able to overcome this here. 
This, amongst other problems, forces Lin to use another approach, and hence demonstrates why our methods are distinct.
\end{rem}

\begin{rem} 
    He--He--Zhou \cite{HeHeZhou-idempotent-completion-of-certain-n-exangulated-categories} have considered 
    idempotent completions of $n$-exan\-gu\-lat\-ed categories in a specific setup. 
    In their setup, there is an ambient Krull-Schmidt $(n+2)$-angulated category $\CC$ and an additive subcategory $\CA$ that is \emph{$n$-extension-closed} (see \cref{def:closed-under-extensions}) 
    and closed under direct summands in $\CC$.
    The main aim of \cite{HeHeZhou-idempotent-completion-of-certain-n-exangulated-categories} is to show that the idempotent completion $\wt{\CA}$ of $\CA$ is an $n$-exangulated subcategory of $\wt{\CC}$. 

    Since $\CA$ is an additive subcategory of and closed under direct summands in a Krull-Schmidt category, it is 
    Krull-Schmidt itself. 
    In particular, $\CA\simeq\wt{\CA}$ is already idempotent complete by \cite[Cor.\ 4.4]{Krause-KS-cats-and-projective-covers}. 
    Moreover, in the setup of \cite{HeHeZhou-idempotent-completion-of-certain-n-exangulated-categories}, it already follows that $\CA$ inherits an $n$-exangulated structure from $\CC\simeq\wt{\CC}$. 
    Indeed, \ref{EA1} is proven in \cite[Lem.\ 3.8]{Klapproth-n-exact-categories-arising-from-nplus2-angulated-categories}, and \ref{EA2} and \ref{EA2op}$\tensor*[]{}{^{\op}}$ are straightforward to check directly. 
    It is then clear that $\CA$ inherits an $n$-exangulated structure from $\CC$.
\end{rem}

\newpage
\section{The weak idempotent completion of an \texorpdfstring{$n$}{n}-exangulated category}
\label{sec:the-WIC}

Just as in \cref{sec:the-idempotent-completion}, we assume $n\geq 1$ is an integer and that $(\CC,\BE,\fs)$ is an $n$-exangulated category. 
By Theorems~\ref{thm:mainthm-ctilde} and \ref{thm:IC-is-2-universal}, 
the idempotent completion of $(\CC,\BE,\fs)$ is an $n$-exangulated category $(\wt{\CC},\BF,\ft)$ 
and the inclusion functor $\tensor*[]{\SI}{_{\CC}}$ of $\CC$ into $\wt{\CC}$ is part of an $n$-exangulated functor $(\tensor*[]{\SI}{_{\CC}}, \Gamma) \colon (\CC, \BE, \fs) \to (\wt{\CC}, \BF, \ft)$, which satisfies the $2$-universal property from \cref{thm:IC-is-2-universal}. 
In this section, we turn our attention to the weak idempotent completion $\wh{\CC}$ of $\CC$ and we show that it forms part of a triplet $(\wh{\CC},\BG,\fr)$ that is $n$-extension-closed (see \cref{def:closed-under-extensions}) in $(\wt{\CC}, \BF, \ft)$. 
It will then follow that $(\wh{\CC},\BG,\fr)$ is itself $n$-exangulated, and, moreover, there is an analogue of \cref{thm:IC-is-2-universal} for $(\wh{\CC},\BG,\fr)$; see \cref{thm:WIC-is-2-universal}.

We begin with the following proposition, which is an analogue of \cref{lem:summand-3} for the weak idempotent completion.

\begin{prop} \label{prop:chat-extclosed}
    Suppose $(\tensor*[]{X}{_{0}}, \tensor*[]{e}{_{0}}), (\tensor*[]{X}{_{n+1}}, \tensor*[]{e}{_{n+1}}) \in \wh{\CC}$ are objects, 
    $\tilde{\delta} \in \BF((\tensor*[]{X}{_{n+1}}, \tensor*[]{e}{_{n+1}}), (\tensor*[]{X}{_{0}}, \tensor*[]{e}{_{0}}))$ is an $\BF$-extension and 
    $\fs(\delta) = [\tensor*[]{X}{_{\bullet}}]$.
    Then there is a $\ft$-distinguished $n$-exangle 
    $\lan  \tensor*[]{\wt{Y}}{_{\bullet}}, \tilde{\delta}\ran$ with $\tensor*[]{\wt{Y}}{_{\bullet}} \in \com{\wh{\CC}}^{\raisebox{0.5pt}{\scalebox{0.6}{$n$}}}$  
    and an $\fs$-distinguished $n$-exangle 
    $\smash{\lan Y'_{\bullet}, \tensor*[_{Y'_0}]{0}{_{Y'_{n+1}}} \ran}$, 
    such that 
    \[\lan \tensor*[]{\SI}{_{\CC}}(\tensor*[]{X}{_{\bullet}}), 
\tensor*[]{\Gamma}{_{(\tensor*[]{X}{_{n+1}},\tensor*[]{X}{_{0}})}}(\delta)
    \ran 
    		\iso \lan \tensor*[]{\wt{Y}}{_{\bullet}}, \tilde{\delta}\ran 
			\oplus 
			\lan \tensor*[]{\SI}{_{\CC}}(Y'_{\bullet}), \tensor*[]{\Gamma}{_{(Y'_{n+1}, Y'_{0})}} (\tensor*[_{Y'_0}]{0}{_{Y'_{n+1}}}) \ran
	\]
		as $\ft$-distinguished $n$-exangles. 
\end{prop}
\begin{proof}
    By \cref{cor:newlift}, there exists an idempotent morphism 
    $\tensor*[]{e}{_{\bullet}} \colon \lan \tensor*[]{X}{_{\bullet}}, \delta \ran \to \lan \tensor*[]{X}{_{\bullet}}, \delta \ran$ with $\tensor*[]{e}{_{i}} = \id{\tensor*[]{X}{_{i}}}$ for $2 \leq i \leq n-1$, 
    as well as a homotopy $\tensor*[]{h}{_{\bullet}} = (\tensor*[]{h}{_{1}}, 0, \dots, 0, \tensor*[]{h}{_{n+1}}) \colon e'_{\bullet} \sim \tensor*[]{0}{_{\bullet}}$, where $e'_{\bullet} \deff \id{\tensor*[]{X}{_{\bullet}}} - \tensor*[]{e}{_{\bullet}}$.
    Notice $(\tensor*[]{X}{_{i}}, \tensor*[]{e}{_{i}}) \in \wh{\CC}$ for $i = 0, n+1$ by assumption. Furthermore, $(\tensor*[]{X}{_{i}}, \tensor*[]{e}{_{i}}) = (\tensor*[]{X}{_{i}},\id{\tensor*[]{X}{_{i}}}) \in \tensor*[]{\SI}{_{\CC}}(\CC) \sse \wh{\CC}$ for ${2 \leq i \leq n-1}$. 
    Set $\tilde{\rho} \deff\tensor*[]{\Gamma}{_{(\tensor*[]{X}{_{n+1}},\tensor*[]{X}{_{0}})}}(\delta)$. 
    By \cref{lem:summand-3} we have $\langle \tensor*[]{\SI}{_{\CC}}(\tensor*[]{X}{_{\bullet}}), \tilde{\rho} \rangle \iso \langle (\tensor*[]{X}{_{\bullet}}, \tensor*[]{e}{_{\bullet}}), \tilde{\delta} \rangle \oplus \langle (\tensor*[]{X}{_{\bullet}}, e'_{\bullet}), \tensor*[_{(\tensor*[]{X}{_{0}}, e'_{0})}]{\wt{0}}{_{(\tensor*[]{X}{_{n+1}}, e'_{n+1})}}\rangle$ as $\ft$-dis\-tin\-guish\-ed $n$-exangles. 
We will show that there is an isomorphism
$\tensor*[]{\tilde{s}}{_{\bullet}} \colon (\tensor*[]{X}{_{\bullet}},\tensor*[]{e}{_{\bullet}}) \to \tensor*[]{\wt{Y}}{_{\bullet}}$ 
in $\com{\wt{\CC}}_{((\tensor*[]{X}{_{0}},\tensor*[]{e}{_{0}}),(\tensor*[]{X}{_{n+1}},\tensor*[]{e}{_{n+1}}))}^{\raisebox{0.5pt}{\scalebox{0.6}{$n$}}}$ 
for some
$\tensor*[]{\wt{Y}}{_{\bullet}} \in \com{\wh{\CC}}^{\raisebox{0.5pt}{\scalebox{0.6}{$n$}}}$, as well as an isomorphism
$\tilde{s}'_{\bullet} 
	\colon 
		(\tensor*[]{X}{_{\bullet}}, e'_{\bullet})
	\to  
	\tensor*[]{\SI}{_{\CC}}(Y'_{\bullet})
$ 
in $\com{\wt{\CC}}^{\raisebox{0.5pt}{\scalebox{0.6}{$n$}}}$ 
for some object $Y'_\bullet \in \com{\CC}^{\raisebox{0.5pt}{\scalebox{0.6}{$n$}}}$.

If $i = 0,n+1$, then 
$e'_{i} 
    = \id{\tensor*[]{X}{_{i}}} - \tensor*[]{e}{_{i}}
$ 
is split by assumption, so 
by \cref{lem:split-summand} 
there are objects 
\mbox{$Y'_i \in\CC$}
and isomorphisms 
$\tilde{s}'_{i} \colon (\tensor*[]{X}{_{i}}, e'_{i}) \to \tensor*[]{\SI}{_{\CC}}(Y'_{i})$.
For $2\leq i \leq n-1$, we see that  
\mbox{$e'_{i} 
    = \id{\tensor*[]{X}{_{i}}} - \tensor*[]{e}{_{i}}
    = 0
$}, so by \cref{lem:split-summand} again  
we have isomorphisms 
$\tilde{s}'_{i} \colon (\tensor*[]{X}{_{i}}, e'_{i}) \to \tensor*[]{\SI}{_{\CC}}(Y'_{i})$, 
but now where 
\mbox{$Y'_{i} = 0 \in\CC$}. 
Since
$(\tensor*[]{X}{_{0}},\tensor*[]{e}{_{0}}), 
(\tensor*[]{X}{_{n+1}},\tensor*[]{e}{_{n+1}})
\in \wh{\CC}$ by assumption 
and because
\mbox{$(\tensor*[]{X}{_{i}}, \tensor*[]{e}{_{i}}) 
    = (\tensor*[]{X}{_{i}},\id{\tensor*[]{X}{_{i}}}) 
    \in \tensor[]{\SI}{_{\CC}}(\CC) 
    \sse \wh{\CC}
$} for $2 \leq i \leq n-1$,
we put 
$\tensor*[]{\wt{Y}}{_{i}} 
    \deff (\tensor*[]{X}{_{i}}, \tensor*[]{e}{_{i}})
$ 
and 
$\tensor*[]{\tilde{s}}{_{i}} 
    \deff \wid{(\tensor*[]{X}{_{i}}, \tensor*[]{e}{_{i}})}
$
for $i \in \{ 0,n+1\} \cup \{2,\ldots, n-1\}$. 
It remains to find appropriate isomorphisms $\tensor*[]{\tilde{s}}{_{i}}$ and $\tilde{s}'_i$ for $i=1,n$. 

We have a morphism $\tensor*[]{\tilde{k}}{_{1}} \colon (\tensor*[]{X}{_{1}}, e'_1) \to (\tensor*[]{X}{_{0}}, e'_0)$ with underlying morphism $\tensor*[]{k}{_{1}} \deff e'_0 \tensor*[]{h}{_{1}} e'_1$ and another $\tensor*[]{\tilde{k}}{_{n+1}} \colon (\tensor*[]{X}{_{n+1}}, e'_{n+1}) \to (\tensor*[]{X}{_{n}}, e'_n)$ with underlying morphism $\tensor*[]{k}{_{n+1}} \deff e'_n \tensor*[]{h}{_{n+1}} e'_{n+1}$ by \cref{lem:induced-morphism}.
Since $(\tensor*[]{h}{_{1}}, 0, \dots, 0, \tensor*[]{h}{_{n+1}}) \colon e'_{\bullet} \sim \tensor*[]{0}{_{\bullet}}$ is a homotopy, we see that 
$\tensor*[]{h}{_{1}} \tensor*[]{d}{_{0}^{X}} = e'_0$.
This implies 
\begin{equation}
\label{eqn:d0-is-a-section}
	\tensor*[]{k}{_{1}} \tensor*[]{d}{_{0}^{(X, e')}} 
		= e'_0 \tensor*[]{h}{_{1}} e'_1 \tensor*[]{d}{_{0}^{X}} e'_{0}
		= e'_0 \tensor*[]{h}{_{1}} \tensor*[]{d}{_{0}^{X}} e'_0 
		= e'_0 
		= \id{(\tensor*[]{X}{_{0}}, e'_0)},
    \end{equation}
and so 
$\tensor*[]{\tilde{k}}{_{1}} \tensor*[]{\tilde{d}}{_{0}^{(X, e')}} =\wid{(\tensor*[]{X}{_{0}}, e'_0)}$. 
Similarly, we also have 
$\tensor*[]{\tilde{d}}{_{n}^{(X, e')}} \tensor*[]{\tilde{k}}{_{n+1}} = \wid{(\tensor*[]{X}{_{n+1}}, e'_{n+1})}$. 

\begin{enumerate}
    \item 
    If $n=1$, then \eqref{eqn:d0-is-a-section} shows that $\tensor*[]{\tilde{d}}{_{0}^{(X, e')}}$ is a section in the complex $(\tensor*[]{X}{_{\bullet}},e'_{\bullet})$, and hence this complex is a split short exact sequence by \cite[Claim 2.15]{HerschendLiuNakaoka-n-exangulated-categories-I-definitions-and-fundamental-properties}. 
	In particular, we have that 
	$(\tensor*[]{X}{_{1}},e'_{1}) 
		\iso (\tensor*[]{X}{_{0}},e'_{0}) \oplus (\tensor*[]{X}{_{2}},e'_2)
		\iso \tensor*[]{\SI}{_{\CC}}(Y'_{0}) \oplus \tensor*[]{\SI}{_{\CC}}(Y'_{2})
		\iso \tensor*[]{\SI}{_{\CC}}(Y'_{0}\oplus Y'_{2})
	$. 
	So we put 
	$Y'_{1} \deff Y'_{0}\oplus Y'_{2}$ 
	and define 
	$\tilde{s}'_{1} \colon (\tensor*[]{X}{_{1}}, e'_{1}) \to \tensor*[]{\SI}{_{\CC}}(Y'_{1})$ to be this composition of isomorphisms. 
	As 
	$
	\tensor*[]{\SI}{_{\CC}}(\tensor*[]{X}{_{1}})
		\iso (\tensor*[]{X}{_{1}}, \tensor*[]{e}{_{1}}) \oplus (\tensor*[]{X}{_{1}}, e'_1)
	$, 
	and $\tensor*[]{\SI}{_{\CC}}(\tensor*[]{X}{_{1}})$ and $(\tensor*[]{X}{_{1}}, e'_1)$ are isomorphic to objects in $\wh{\CC}$,  
	by \cref{lem:isoclosure-chat} 
	there is an isomorphism 
	$\tensor*[]{\tilde{s}}{_{1}} 
		\colon (\tensor*[]{X}{_{1}}, \tensor*[]{e}{_{1}}) 
		\to \tensor*[]{\wt{Y}}{_{1}}
	$ 
	for some $\tensor*[]{\wt{Y}}{_{1}} \in \wh{\CC}$.
    
    \item If $n \geq 2$, then the form of the homotopy $\tensor*[]{h}{_{\bullet}}$ implies that the identities 
    $\tensor*[]{d}{_{0}^{X}} \tensor*[]{h}{_{1}} = e'_1$ and $\tensor*[]{h}{_{n+1}}\tensor*[]{d}{_{n+1}^{X}} = e'_n$ hold. 
    Therefore, we see that 
	\[
	\tensor*[]{d}{_{0}^{(X, e')}} \tensor*[]{k}{_{1}} 
		=  \tensor*[]{d}{_{0}^{X}} e'_{0} \tensor*[]{h}{_{1}} e'_1 
		= e'_1 \tensor*[]{d}{_{0}^{X}} \tensor*[]{h}{_{1}} e'_1 
		= e'_1
		= \id{(\tensor*[]{X}{_{1}}, e'_1)},
	\]
    which shows that $\tensor*[]{\tilde{k}}{_{1}}$ and $\tensor*[]{\tilde{d}}{_{0}^{(X, e')}}$ are mutually inverse isomorphisms. 
	We now define $Y'_{1} \deff Y'_{0}$ and 
    $\tilde{s}'_{1} \deff \tilde{s}'_0 \tensor*[]{\tilde{k}}{_{1}}
	    \colon (\tensor*[]{X}{_{1}}, e'_{1}) \to \tensor*[]{\SI}{_{\CC}}(Y'_{1})
	$. 
    Because there are isomorphisms 
    $
    \tensor*[]{\SI}{_{\CC}}(\tensor*[]{X}{_{1}}) 
		\iso (\tensor*[]{X}{_{1}}, \tensor*[]{e}{_{1}}) \oplus (\tensor*[]{X}{_{1}}, e'_1) 
		\iso (\tensor*[]{X}{_{1}}, \tensor*[]{e}{_{1}}) \oplus (\tensor*[]{X}{_{0}}, e'_0)
	$, 
	and $\tensor*[]{\SI}{_{\CC}}(\tensor*[]{X}{_{1}})$ and $(\tensor*[]{X}{_{0}}, e'_0)$ are isomorphic to objects in $\wh{\CC}$, 
	by \cref{lem:isoclosure-chat} there is an isomorphism 
	$\tensor*[]{\tilde{s}}{_{1}} 
		\colon (\tensor*[]{X}{_{1}}, \tensor*[]{e}{_{1}}) 
		\to \tensor*[]{\wt{Y}}{_{1}}
	$ 
	for some $\tensor*[]{\wt{Y}}{_{1}} \in \wh{\CC}$.

    In a similar way, one can show that 
    $\tensor*[]{\tilde{k}}{_{n+1}}$ and $\tensor*[]{\tilde{d}}{_{n}^{(X, e')}}$ are mutually inverse isomorphisms. 
    We set 
    $Y'_{n} \deff Y'_{n+1}$ and 
    $\tilde{s}'_{n} \deff \tilde{s}'_{n+1} \tensor*[]{\tilde{d}}{_{n}^{(X,e')}}
		\colon (\tensor*[]{X}{_{n}}, e'_{n}) \to \tensor*[]{\SI}{_{\CC}}(Y'_{n})
	$. 
    In addition, there is an isomorphism $\tensor*[]{\tilde{s}}{_{n}} \colon (\tensor*[]{X}{_{n}}, \tensor*[]{e}{_{n}}) \to \tensor*[]{\wt{Y}}{_{n}}$ for some $\tensor*[]{\wt{Y}}{_{n}} \in \wh{\CC}$. 
\end{enumerate}
    
    The complex $\tensor*[]{\wt{Y}}{_{\bullet}}$ with object $\tensor*[]{\wt{Y}}{_{i}}$ 
    in degree $0\leq i \leq n+1$
    and differential $\tensor*[]{\tilde{d}}{^{\wt{Y}}_{i}} \deff \tensor*[]{\tilde{s}}{_{i+1}} \tensor*[]{\tilde{d}}{^{(X,e)}_{i}} \tensor*[]{\tilde{s}}{_{i}}^{-1}$ in degree $0\leq i\leq n$ 
    is isomorphic to $(\tensor*[]{X}{_{\bullet}}, \tensor*[]{e}{_{\bullet}})$ via $\tensor*[]{\tilde{s}}{_{\bullet}}$. 
    Furthermore, as $\tensor*[]{\tilde{s}}{_{0}}$ and $\tensor*[]{\tilde{s}}{_{n+1}}$ are identity morphisms, 
    we have that 
    $\lan \tensor*[]{\wt{Y}}{_{\bullet}}, \tilde{\delta}\ran$ is $\ft$-distinguished by \cite[Cor.\ 2.26(2)]{HerschendLiuNakaoka-n-exangulated-categories-I-definitions-and-fundamental-properties}. 
    The complex 
    $\wt{Y}'_{\bullet}$
    with object 
    $\wt{Y}'_i \deff \tensor*[]{\SI}{_{\CC}}(Y'_{i})$ 
    in degree $0\leq i \leq n+1$
    and
    differential 
    \mbox{$\tensor*[]{\tilde{d}}{^{\wt{Y}'}_{i}} \deff \tilde{s}'_{i+1} \tensor*[]{\tilde{d}}{^{(X,e')}}_{i} \tilde{s}'^{-1}_i$}
    in degree $0\leq i \leq n$ is isomorphic to 
    $(\tensor*[]{X}{_{\bullet}}, e'_{\bullet})$ via $\tilde{s}'_{\bullet}$.
    Moreover, this induces an isomorphism
    	$\tilde{s}'_{\bullet} 
		\colon 
		\langle (\tensor*[]{X}{_{\bullet}}, e'_{\bullet}), \tensor*[_{(\tensor*[]{X}{_{0}}, e'_{0})}]{\wt{0}}{_{(\tensor*[]{X}{_{n+1}}, e'_{n+1})}}\rangle 
		\to  
		\lan \wt{Y}'_{\bullet}, \tensor*[_{\wt{Y}'_0}]{\wt{0}}{_{\wt{Y}'_{n+1}}} \ran
	$ 
	of $\BF$-attached complexes, and hence of $\ft$-distinguished $n$-exangles. 
    It is clear that 
    \[\lan \wt{Y}'_{\bullet}, \tensor*[_{\wt{Y}'_0}]{\wt{0}}{_{\wt{Y}'_{n+1}}} \ran 
        \iso \lan \tensor{\SI}{_{\CC}}\big(\tensor*[]{\triv}{_{0}}(Y'_0\tensor*[]{)}{_{\bullet}} \oplus \tensor*[]{\triv}{_{n}}(Y'_{n+1}\tensor*[]{)}{_{\bullet}}\big), \tensor{\Gamma}{_{(Y'_{n+1}, Y'_{0})}}(\tensor*[_{Y'_0}]{0}{_{Y'_{n+1}}})  \ran\]
    by the construction of 
    $\wt{Y}'_\bullet$.  Lastly, $\lan \tensor*[]{\triv}{_{0}}(Y'_0\tensor*[]{)}{_{\bullet}} \oplus \tensor*[]{\triv}{_{n}}(Y'_{n+1})_\bullet, \tensor*[_{Y'_0}]{0}{_{Y'_{n+1}}}  \ran$ is $\fs$-distinguished using \cite[Prop.\ 3.3]{HerschendLiuNakaoka-n-exangulated-categories-I-definitions-and-fundamental-properties} and that $\fs$ is an exact realisation of $\BE$.
\end{proof}

From \cref{prop:chat-extclosed} we see that $\wh{\CC}$ is $n$-extension-closed in 
$(\wt{\CC},\BF,\ft)$ in the following sense.

\begin{defn}
\label{def:closed-under-extensions}
\cite[Def.\ 4.1]{HerschendLiuNakaoka-n-exangulated-categories-II} 
Let $(\CC',\BE',\fs')$ be an $n$-exangulated category. 
A full subcategory $\CD\sse \CC'$ is said to be \emph{$n$-extension-closed} 
if, for all $A,C\in\CD$ and each $\BE'$-extension $\delta\in\BE'(C,A)$, 
there is an object $\tensor*[]{X}{_{\bullet}}\in\com{\CC'}^{\raisebox{0.5pt}{\scalebox{0.6}{$n$}}}$ such that $\tensor*[]{X}{_{i}}\in\CD$ for all $1\leq i \leq n$ and $\fs'(\delta) = [\tensor*[]{X}{_{\bullet}}]$. 
\end{defn}

Let us now define the biadditive functor and realisation with which we wish to equip $\wh{\CC}$. 

\begin{defn}
\begin{enumerate}[label=\textup{(\roman*)}]
	\item Let $\BG\deff \restr{\BF}{\wh{\CC}^{\op}\times\wh{\CC}}\colon \wh{\CC}^{\op} \times \wh{\CC}\to \Ab$ be the restriction of $\BF\colon\tensor*[]{\wt{\CC}}{^{\op}} \times \wt{\CC}\to \Ab$. 
	
	\item For a $\BG$-extension $\tilde{\delta} \in \BG(\tensor*[]{\wt{X}}{_{n+1}}, \tensor*[]{\wt{X}}{_{0}})$, there is a $\ft$-distinguished $n$-exangle $\lan \tensor*[]{\wt{X}}{_{\bullet}}, \tilde{\delta} \ran$ with $\tensor*[]{\wt{X}}{_{\bullet}} \in \com{\wh{\CC}}^{\raisebox{0.5pt}{\scalebox{0.6}{$n$}}}$ by \cref{prop:chat-extclosed}. We put $\fr(\tilde{\delta}) = [\tensor*[]{\wt{X}}{_{\bullet}}]$, the isomorphism class of $\tensor*[]{\wt{X}}{_{\bullet}}$ in $\kom{\wh{\CC}}_{(\tensor*[]{\wt{X}}{_{0}}, \tensor*[]{\wt{X}}{_{n+1}})}^{\raisebox{0.5pt}{\scalebox{0.6}{$n$}}}$.

	\item 
	Recall from Subsection~\ref{subsec:weak-idempotent-completion} that $\tensor*[]{\SK}{_{\CC}}\colon \CC \to \wh{\CC}$ is the inclusion functor defined by $\tensor*[]{\SK}{_{\CC}}(X) = (X,\id{X})$ on objects $X \in \CC$. 
	Let $\Delta \colon \BE(-,-) \Rightarrow \BG(\tensor*[]{\SK}{_{\CC}}-,\tensor*[]{\SK}{_{\CC}}-)$ be the restriction of the natural transformation $\Gamma \colon \BE(-,-) \Rightarrow \BF(\tensor*[]{\SI}{_{\CC}}-,\tensor*[]{\SI}{_{\CC}}-)$ defined in \cref{lem:def-of-Gamma}.  
	This means
	$\Delta(\delta) 
		\deff (\id{\tensor*[]{X}{_{0}}},\delta,\id{\tensor*[]{X}{_{n+1}}})
		\in \BG(\tensor*[]{\SK}{_{\CC}}(\tensor*[]{X}{_{n+1}}),\tensor*[]{\SK}{_{\CC}}(\tensor*[]{X}{_{0}}))$ 
	for $\delta\in\BE(\tensor*[]{X}{_{n+1}},\tensor*[]{X}{_{0}})$. 
\end{enumerate}
\end{defn}

Since $\wh{\CC}$ is an $n$-extension closed subcategory of $(\wt{\CC}, \BF, \ft)$ by \Cref{prop:chat-extclosed}, 
one can use \mbox{\cite[Prop.\ 4.2(1)]{HerschendLiuNakaoka-n-exangulated-categories-II}} to deduce axioms \ref{EA2} and \ref{EA2op}$^{\op}$ hold for the triplet $(\wh{\CC}, \BG, \fr)$. The difficult part is then to show that \ref{EA1} is satisfied; this follows from \cref{lem:chat-EA1} below.
We note here, however, that it has been shown more generally in \cite[Thm.\ A]{klapproth2023nextension} that any $n$-extension-closed subcategory of an $n$-exangulated category inherits an $n$-exangulated structure in the expected way.

\begin{lem} \label{lem:chat-EA1}
    Let $\tilde{f} \colon \tensor*[]{\wt{X}}{_{0}} \to \tensor*[]{\wt{X}}{_{1}}$ be a $\ft$-inflation with $\tensor*[]{\wt{X}}{_{0}} = (\tensor*[]{X}{_{0}}, \tensor*[]{e}{_{0}}), \tensor*[]{\wt{X}}{_{1}} = (\tensor*[]{X}{_{1}}, \tensor*[]{e}{_{1}}) \in \wh{\CC}$.
    Then there is a $\ft$-distinguished $n$-exangle $\lan \tensor*[]{\wt{X}}{_{\bullet}}, \tilde{\delta} \ran$ with $\tensor*[]{\wt{X}}{_{\bullet}} \in \com{\wh{\CC}}^{\raisebox{0.5pt}{\scalebox{0.6}{$n$}}}$ and $\tensor*[]{\tilde{d}}{^{\wt{X}}_{0}} = \tilde{f}$.
\end{lem}
\begin{proof}
    By \cref{lem:TwistInflation} there is an object $C \in \CC$, a morphism $f' \colon \tensor*[]{X}{_{0}} \to C$ and an $\fs$-distinguished $n$-exangle $\lan \tensor*[]{Z}{_{\bullet}}, \rho \ran$ with 
    $\tensor*[]{Z}{_{0}} =\tensor*[]{X}{_{0}}$, 
    $\tensor*[]{Z}{_{1}} = \tensor*[]{X}{_{1}} \oplus C$, 
    $\tensor*[]{d}{^{Z}_{0}} = 
    \begin{bsmallmatrix}
   f &\; f'(\id{\tensor*[]{X}{_{0}}}-\tensor*[]{e}{_{0}})
   \end{bsmallmatrix}^\top$
    and  
    $(\tensor*[]{e}{_{0}}\tensor*[]{)}{_{\BE}} \rho = \rho$. The solid morphisms of the diagram
    \[\begin{tikzcd}[ampersand replacement=\&]
   	{\tensor*[]{X}{_{0}}} \& {\tensor*[]{X}{_{1}} \oplus C} \& {\tensor*[]{Z}{_{2}}} \& {\tensor*[]{Z}{_{3}}} \& \cdots \& {\tensor*[]{Z}{_{n}}} \& {\tensor*[]{Z}{_{n+1}}} \& {} \\
   	{\tensor*[]{X}{_{0}}} \& {\tensor*[]{X}{_{1}} \oplus C} \& {\tensor*[]{Z}{_{2}}} \& {\tensor*[]{Z}{_{3}}} \& \cdots \& {\tensor*[]{Z}{_{n}}} \& {\tensor*[]{Z}{_{n+1}}} \& {}
   	\arrow["{\tensor*[]{d}{^{Z}_{0}}}", from=1-1, to=1-2]
   	\arrow["",from=1-2, to=1-3]
   	\arrow[from=1-5, to=1-6]
   	\arrow[from=1-6, to=1-7]
   	\arrow["{\rho}", dashed, from=1-7, to=1-8]
   	\arrow[from=1-3, to=1-4]
   	\arrow[from=1-4, to=1-5]
   	\arrow["{\tensor*[]{d}{^{Z}_{0}}}", from=2-1, to=2-2]
   	\arrow["{\tensor*[]{e}{_{0}}}", from=1-1, to=2-1]
   	\arrow["{\begin{bsmallmatrix}
   	\tensor*[]{e}{_{1}} & 0 \\ 0 & 0
   	\end{bsmallmatrix}}", from=1-2, to=2-2]
   	\arrow[dotted, from=1-3, to=2-3, "e'_2"]
   	\arrow[Rightarrow, no head, from=1-4, to=2-4]
   	\arrow[Rightarrow, no head, from=1-6, to=2-6]
   	\arrow[Rightarrow, no head, from=1-7, to=2-7]
   	\arrow["{\rho}", dashed, from=2-7, to=2-8]
      \arrow["", from=2-2, to=2-3]
      \arrow[from=2-3, to=2-4]
      \arrow[from=2-4, to=2-5]
      \arrow[from=2-5, to=2-6]
      \arrow[from=2-6, to=2-7]
   \end{tikzcd}\]
form a commutative diagram. 
By \cref{lem:inflationcompletion} there exists an idempotent morphism of $n$-exangles $e'_{\bullet} \colon \lan \tensor*[]{Z}{_{\bullet}}, \rho \ran \to \lan \tensor*[]{Z}{_{\bullet}}, \rho \ran$ with $e'_{0} =\tensor*[]{e}{_{0}}$, 
$e'_{1}
	= \begin{bsmallmatrix} \tensor*[]{e}{_{1}} & 0 \\ 0 & 0\end{bsmallmatrix}
$ 
	and $e'_{i} = \id{\tensor*[]{Z}{_{i}}}$ for $3 \leq i \leq n+1$, which makes the diagram above commute. 
Let 
$\tilde{\rho} \deff \tensor*[]{\Gamma}{_{(\tensor*[]{Z}{_{n+1}}, \tensor*[]{X}{_{0}})}}(\rho) \in\BF(\tensor*[]{\SI}{_{\CC}}(\tensor*[]{Z}{_{n+1}}), \tensor*[]{\SI}{_{\CC}}(\tensor*[]{X}{_{0}}))$ 
and 
$\tilde{\rho}' \deff \BF(\ie{e'_{n+1}}, \pe{e'_0})(\tilde{\rho})\in\BF((\tensor*[]{Z}{_{n+1}},e'_{n+1}), (\tensor*[]{X}{_{0}},\tensor*[]{e}{_{0}}))$. 
Notice that the underlying $\BE$-extension of $\tilde{\rho}'$ is $\rho$. 
Set $e''_{\bullet} \deff \id{\tensor*[]{Z}{_{\bullet}}} - e'_{\bullet}$. 
Then $\langle \tensor*[]{\SI}{_{\CC}}(\tensor*[]{Z}{_{\bullet}}), \tilde{\rho} \rangle \iso \langle (\tensor*[]{Z}{_{\bullet}}, e'_{\bullet}), \tilde{\rho}' \rangle \oplus \langle (\tensor*[]{Z}{_{\bullet}}, e''_{\bullet}), \tensor*[_{(\tensor*[]{Z}{_{0}}, e''_{0})}]{\wt{0}}{_{(\tensor*[]{Z}{_{n+1}}, e''_{n+1})}}\rangle$ as $\ft$-distinguished $n$-exangles by \cref{lem:summand-3}.

We claim that there is a split short exact sequence 
\begin{equation}\label{eqn:splitSESin5-1}
    \begin{tikzcd}[column sep=1.3cm]
    0 \arrow{r}
    & (\tensor*[]{Z}{_{0}}, e''_{0}) \arrow{r}{\tensor*[]{\tilde{d}}{^{(Z,e'')}_{0}}}
    & (\tensor*[]{Z}{_{1}}, e''_{1}) \arrow{r}{\tensor*[]{\tilde{d}}{^{(Z,e'')}_{1}}}
    & (\tensor*[]{Z}{_{2}}, e''_{2}) \arrow{r}
    & 0.
    \end{tikzcd}
\end{equation}
Since $(\tensor*[]{Z}{_{\bullet}}, e''_{\bullet})$ realises the trivial $\BF$-extension $\tensor*[_{(\tensor*[]{Z}{_{0}}, e''_{0})}]{\wt{0}}{_{(\tensor*[]{Z}{_{n+1}}, e''_{n+1})}}$, 
we have that $\tensor*[]{\tilde{d}}{^{(Z,e'')}_{0}}$ is a section by \cite[Claim 2.15]{HerschendLiuNakaoka-n-exangulated-categories-I-definitions-and-fundamental-properties}. 
If $n=1$, then this is enough to see that  \eqref{eqn:splitSESin5-1} is split short exact.
For $n\geq 2$ we notice that there is an isomorphism 
$(\tensor*[]{Z}{_{3}}, e''_{3}) \iso 0$ in $\wt{\CC}$, so $\tensor*[]{\tilde{d}}{^{(Z,e'')}_{2}} = 0$. 
Thus, since $\tensor*[]{\tilde{d}}{^{(Z,e'')}_{1}}$ is a weak kernel of $\tensor*[]{\tilde{d}}{^{(Z,e'')}_{2}}$, 
we see that $\wid{(\tensor*[]{Z}{_{2}},e''_{2})}$ factors through $\tensor*[]{\tilde{d}}{^{(Z,e'')}_{1}}$.  
In particular, this implies $\tensor*[]{\tilde{d}}{^{(Z,e'')}_{1}}$ is a cokernel of $\tensor*[]{\tilde{d}}{^{(Z,e'')}_{0}}$. 
Again, \eqref{eqn:splitSESin5-1} is split short exact.

In particular, we have an isomorphism $(\tensor*[]{Z}{_{1}}, e''_{1}) \iso (\tensor*[]{Z}{_{0}}, e''_{0}) \oplus (\tensor*[]{Z}{_{2}}, e''_{2})$.
    We know that the objects
    $(\tensor*[]{Z}{_{0}}, e''_{0}) = (\tensor*[]{X}{_{0}}, \id{\tensor*[]{X}{_{0}}} - \tensor*[]{e}{_{0}})$ 
    and 
    $(\tensor*[]{Z}{_{1}}, e''_1) = (\tensor*[]{X}{_{1}} \oplus C,
    (\id{\tensor*[]{X}{_{1}}} - \tensor*[]{e}{_{1}})\oplus \id{C}
	)$ 
    are isomorphic to objects in $\tensor*[]{\SI}{_{\CC}}(\CC) \sse \wh{\CC} \sse \wt{\CC}$ by \cref{lem:split-summand}, as $\id{\tensor*[]{X}{_{0}}} - \tensor*[]{e}{_{0}}$ and $(\id{\tensor*[]{X}{_{1}}} - \tensor*[]{e}{_{1}})\oplus \id{C}$ are split idempotents.
This implies that $(\tensor*[]{Z}{_{2}}, e''_{2})$ is isomorphic to an object in $\wh{\CC}$ by \cref{lem:isoclosure-chat}.
    Again \cref{lem:isoclosure-chat} and the isomorphism $\tensor*[]{\SI}{_{\CC}}(\tensor*[]{Z}{_{2}}) \iso (\tensor*[]{Z}{_{2}}, e'_2) \oplus (\tensor*[]{Z}{_{2}}, e''_2)$ imply that there is an isomorphism $\tensor*[]{\tilde{s}}{_{2}} \colon (\tensor*[]{Z}{_{2}}, e'_2) \to \tensor*[]{\wt{X}}{_{2}}$ for some $\tensor*[]{\wt{X}}{_{2}} \in \wh{\CC}$. 
    
    The morphism $\tensor*[]{\tilde{s}}{_{1}} \colon \left(\tensor*[]{X}{_{1}} \oplus C, \begin{bsmallmatrix} \tensor*[]{e}{_{1}} & 0 \\ 0 & 0\end{bsmallmatrix}\right) \to \tensor*[]{\wt{X}}{_{1}}$ with underlying morphism 
    $\tensor*[]{s}{_{1}} 
		= \begin{bsmallmatrix}
		   \tensor*[]{e}{_{1}} &\; 0
	   \end{bsmallmatrix}$
     is an isomorphism.
    Finally, put $\tensor*[]{\tilde{s}}{_{i}} = \wid{(\tensor*[]{Z}{_{i}},e'_{i})}$ for $i = 0$ and $3 \leq i \leq n+1$.
    Then the complex 
   \[
   \tensor*[]{\wt{X}}{_{\bullet}}\colon\hspace{1cm}
   \begin{tikzcd}
   \tensor*[]{\wt{X}}{_{0}}
   		\arrow{r}{\tilde{f}}
   &[-0.7em] \tensor*[]{\wt{X}}{_{1}}
   		\arrow{r}{\tensor*[]{\tilde{s}}{_{2}} \tensor*[]{\tilde{d}}{^{(Z, e')}_{1}} \tensor*[]{\tilde{s}}{_{1}^{-1}}}
   &[2.3em] \tensor*[]{\wt{X}}{_{2}}
   		\arrow{r}{\tensor*[]{\tilde{d}}{^{(Z, e')}_{2}} \tensor*[]{\tilde{s}}{_{2}^{-1}} }
   &[1.8em] (\tensor*[]{Z}{_{3}}, 
        e'_{3})
   		\arrow{r}{\tensor*[]{\tilde{d}}{^{(Z, e')}_{3}}}
   &[0.3em] \cdots
   		\arrow{r}{\tensor*[]{\tilde{d}}{^{(Z, e')}_{n}}}
   &[0.3em] (\tensor*[]{Z}{_{n+1}},e'_{n+1})
   \end{tikzcd}
   \]
is 
isomorphic to $(\tensor*[]{Z}{_{\bullet}}, e'_{\bullet})$ 
via $\tensor*[]{\tilde{s}}{_{\bullet}} \colon (\tensor*[]{Z}{_{\bullet}}, e'_{\bullet}) \to \tensor*[]{\wt{X}}{_{\bullet}}$ 
in $\com{\wt{\CC}}^{\raisebox{0.5pt}{\scalebox{0.6}{$n$}}}$. 
With $\tilde{\delta} \deff(\tensor*[]{\tilde{s}}{_{n+1}^{-1}}\tensor*[]{)}{^{\BF}} \tilde{\rho}'$, 
we see that $\lan \tensor*[]{\wt{X}}{_{\bullet}}, \tilde{\delta} \ran$ is $\ft$-distinguished by \cite[Cor.\ 2.26(2)]{HerschendLiuNakaoka-n-exangulated-categories-I-definitions-and-fundamental-properties}, as desired. 
\end{proof}

We may state and prove our main result of this section.

\begin{thm}
\label{thm:WIC-is-2-universal}
    Suppose that $(\CC, \BE, \fs)$ is an $n$-exangulated category. 
    Then $(\wh{\CC}, \BG, \fr)$ is a weakly idempotent complete $n$-exangulated category, 
    and ${(\tensor*[]{\SK}{_{\CC}}, \Delta) \colon (\CC, \BE, \fs) \to (\wh{\CC}, \BG, \fr)}$ is an $n$-exangulated functor, such that the following $2$-universal property is satisfied.    
    Suppose $(\SF,\Lambda) \colon (\CC, \BE, \fs) \to (\CC', \BE', \fs')$ is an $n$-exangulated functor to a weakly idempotent complete $n$-exangulated category $(\CC', \BE', \fs')$. Then the following statements hold. 
    \begin{enumerate}[label=\textup{(\roman*)}]
        \item \label{item:weak-universal-property-1}
        There is an $n$-exangulated functor $(\SE, \Psi) \colon (\wh{\CC}, \BG, \fr) \to (\CC', \BE', \fs')$ and an $n$-exangulated natural isomorphism $\Tsadi \colon (\SF, \Lambda) \overset{\iso}{\Longrightarrow} (\SE, \Psi) \circ (\tensor*[]{\SK}{_{\CC}}, \Delta)$.

        \item \label{item:weak-universal-property-2}
        In addition, 
        for any $n$-exangulated functor $(\SG, \Theta) \colon (\wh{\CC},\BG,\fr) \to (\CC',  \BE', \fs')$ and any $n$-exangulated natural transformation $\Daleth \colon (\SF, \Lambda) \Rightarrow (\SG, \Theta) \circ (\tensor*[]{\SK}{_{\CC}}, \Delta)$,
        there is a unique $n$-exangulated natural transformation 
        $\Mem \colon (\SE,\Psi) \Rightarrow (\SG,\Theta)$ with $\Daleth = \tensor*[]{\Mem}{_{\tensor*[]{\SK}{_{\CC}}}} \Tsadi$.
    \end{enumerate}
\end{thm}

\begin{proof}
    Since $\wh{\CC}$ is a full subcategory of $\wt{\CC}$ and because $(\wt{\CC}, \BF, \ft)$ is $n$-exangulated, we can apply \cite[Prop.\ 4.2]{HerschendLiuNakaoka-n-exangulated-categories-II} and \cref{def:closed-under-extensions}. 
    We showed above that $\wh{\CC}$ is $n$-extension-closed in $\wt{\CC}$; see \cref{prop:chat-extclosed}. 
    Moreover, it follows immediately from \cref{lem:chat-EA1} and its dual that 
    $(\wh{\CC},\BG,\fr)$ satisfies \ref{EA1}. 
    Therefore, we deduce that $(\wh{\CC}, \BG, \fr)$ is an $n$-exangulated category.

One may argue that $(\tensor*[]{\SK}{_{\CC}}, \Delta)$ is an $n$-exangulated functor as in \cref{prop:SI-Gamma-is-n-exangulated}, by using the definition of $\fr$ and noting that $(X,\id{X}) = \tensor*[]{\SK}{_{\CC}}(X)$ lies in $\wh{\CC}$ for all $X\in\CC$.

\ref{item:weak-universal-property-1}\; 
One argues like in the proof of \cref{thm:IC-is-2-universal}\ref{item:universal-property-1}, but using \cref{prop:universal-property-of-WIC} instead of \cref{prop:Buehler-universal-property-of-karoubi-envelope}, and \cref{prop:chat-extclosed} instead of \cref{lem:summand-3}. 
In particular, we note that the isomorphism
$\lan 
    \tensor*[]{\SI}{_{\CC}}(\tensor*[]{X}{_{\bullet}}), 
    \tensor*[]{\Gamma}{_{(\tensor*[]{X}{_{n+1}},\tensor*[]{X}{_{0}})}}(\delta)
\ran 
\iso 
\lan 
    \tensor*[]{\wt{Y}}{_{\bullet}}, 
    \tilde{\delta}
\ran 
    \oplus 
    \lan 
    \tensor*[]{\SI}{_{\CC}}(Y'_{\bullet}), 
    \tensor*[]{\Gamma}{_{(Y'_{n+1},Y'_{0})}}(\tensor*[_{Y'_{0}}]{0}{_{Y'_{n+1}}})
    \ran
$ 
of $\ft$-distinguished $n$-exangles from the statement of \cref{prop:chat-extclosed} induces an isomorphism 
\[
\lan 
    \tensor*[]{\SK}{_{\CC}}(\tensor*[]{X}{_{\bullet}}), 
    \tensor*[]{\Delta}{_{(\tensor*[]{X}{_{n+1}},\tensor*[]{X}{_{0}})}}(\delta)
\ran 
\iso 
\lan 
    \tensor*[]{\wt{Y}}{_{\bullet}}, 
    \tilde{\delta}
\ran 
    \oplus 
    \lan 
        \tensor*[]{\SK}{_{\CC}}(Y'_{\bullet}), 
    \tensor*[]{\Delta}{_{(Y'_{n+1},Y'_{0})}}(\tensor*[_{Y'_{0}}]{0}{_{Y'_{n+1}}})
    \ran
\] 	
of $\fr$-distinguished $n$-exangles. 
    
\ref{item:weak-universal-property-2}\; 
Similarly, one adapts the proof of \cref{thm:IC-is-2-universal}\ref{item:universal-property-2}, using \cref{prop:universal-property-of-WIC} instead of \cref{prop:Buehler-universal-property-of-karoubi-envelope}.
\end{proof}

Finally, we have an analogue of \cref{cor:idempotent-n-exact} as a consequence. 

\begin{cor}\label{cor:WIC-n-exact}
    Suppose $(\CC, \BE, \fs)$ is $n$-exact.
    Then $(\wh{\CC}, \BG, \fr)$ is $n$-exact. 
\end{cor}
\begin{proof}
    The $n$-exangulated category b$(\wt{\CC},\BF,\ft)$ is $n$-exact by \Cref{cor:idempotent-n-exact}.
    As $(\wh{\CC},\BG,\fr)$ inherits its structure as an $n$-extension closed subcategory of $(\wt{\CC},\BF,\ft)$, the result follows from the proof of \cite[Cor.\ 4.15]{klapproth2023nextension}.
\end{proof}

\newpage
\bibliography{references}
\bibliographystyle{mybstwithlabels}
\end{document}